\documentclass[letterpaper, 11pt,  reqno]{amsart}
\pdfoutput=1

\usepackage[left=32mm, right=32mm, 
bottom=27mm]{geometry}

\usepackage{amsmath,amssymb,amscd,amsthm,amsxtra, esint}

\usepackage{booktabs} 
\usepackage{microtype}
\usepackage{amssymb}
\usepackage{mathrsfs}

\usepackage{color}
\usepackage[implicit=true]{hyperref}

\usepackage{xsavebox}

\usepackage{cases}

\allowdisplaybreaks[2]

\definecolor{gr}{rgb}   {0.,   0.69,   0.23 }
\definecolor{bl}{rgb}   {0.,   0.5,   1. }
\definecolor{mg}{rgb}   {0.85,  0.,    0.85}
\definecolor{yl}{rgb}   {0.8,  0.7,   0.}
\definecolor{or}{rgb}  {0.7,0.2,0.2}

\usepackage{tikz} 
\usepackage{marginnote}
\usepackage{scalerel} 

\usetikzlibrary{shapes.misc}
\usetikzlibrary{shapes.symbols}
\usetikzlibrary{decorations}
\usetikzlibrary{decorations.markings}

\tikzset{
	dot/.style={circle,fill=black,draw=black,inner sep=0pt,minimum size=0.5mm},
	>=stealth,
	}

\tikzset{
	dot2/.style={circle,fill=black,draw=black,inner sep=0pt,minimum size=0.2mm},
	>=stealth,
	}

\tikzset{
	ddot/.style={circle,fill=black,draw=black,inner sep=0pt,minimum size=0.8mm},
	>=stealth,
	}

\tikzset{decision/.style={ 
        draw,
        diamond,
        aspect=1.5
    }}

\tikzset{dia2/.style
={diamond,fill=white,draw=black,inner sep=0pt,minimum size=1mm},
	>=stealth,
	}

\tikzset{dia/.style
={star,fill=black,draw=black,inner sep=0pt,minimum size=1mm},
	>=stealth,
	}

\tikzset{dia/.style
={diamond,fill=black,draw=black,inner sep=0pt,minimum size=1.3mm},
	>=stealth,
	}

\makeatletter
\def\DeclareSymbol#1#2#3{\xsavebox{#1}{\tikz[baseline=#2,scale=0.15]{#3}}}
\def\<#1>{\xusebox{#1}}
\makeatother

\newcommand{\pe}{\mathbin{\scaleobj{0.7}{\tikz \draw (0,0) node[shape=circle,draw,inner sep=0pt,minimum size=8.5pt] {\scriptsize  $=$};}}}

\newcommand{\pl}{\mathbin{\scaleobj{0.7}{\tikz \draw (0,0) node[shape=circle,draw,inner sep=0pt,minimum size=8.5pt] {\scriptsize $<$};}}}
\newcommand{\pg}{\mathbin{\scaleobj{0.7}{\tikz \draw (0,0) node[shape=circle,draw,inner sep=0pt,minimum size=8.5pt] {\scriptsize $>$};}}}

\newcommand{\pez}{\mathbin{\scaleobj{0.7}{\tikz \draw (0,0) node[shape=circle,draw,
fill=white, 
inner sep=0pt,minimum size=8.5pt]{} ;}}}

\usetikzlibrary{decorations.pathmorphing, patterns,shapes}

\DeclareSymbol{dot}{-2.7}{\draw (0,0) node[dot]{};}

\DeclareSymbol{1}{0}{\draw[white] (-.4,0) -- (.4,0); \draw (0,0)  -- (0,1.2) node[dot] {};}
\DeclareSymbol{2}{0}{\draw (-0.5,1.2) node[dot] {} -- (0,0) -- (0.5,1.2) node[dot] {};}
\DeclareSymbol{3}{0}{\draw (0,0) -- (0,1.2) node[dot] {}; \draw (-.7,1) node[dot] {} -- (0,0) -- (.7,1) node[dot] {};}
\DeclareSymbol{30}{-3}{\draw (0,0) -- (0,-1); \draw (0,0) -- (0,1.2) node[dot] {}; \draw (-.7,1) node[dot] {} -- (0,0) -- (.7,1) node[dot] {};}
\DeclareSymbol{31}{-3}{\draw (0,0) -- (0,-1) -- (1,0) node[dot] {}; \draw (0,0) -- (0,1.2) node[dot] {}; \draw (-.7,1) node[dot] {} -- (0,0) -- (.7,1) node[dot] {};}
\DeclareSymbol{32}{-3}{\draw (0,0) -- (0,-1) -- (1,0) node[dot] {}; \draw (0,0) -- (0,-1) -- (-1,0) node[dot] {}; \draw (0,0) -- (0,1.2) node[dot] {}; \draw (-.7,1) node[dot] {} -- (0,0) -- (.7,1) node[dot] {};}
\DeclareSymbol{20}{-3}{\draw (0,0) -- (0,-1);\draw (-.7,1) node[dot] {} -- (0,0) -- (.7,1) node[dot] {};}
\DeclareSymbol{22}{-3}{\draw (0,0.3) -- (0,-1) -- (1,0) node[dot] {}; \draw (0,0.3) -- (0,-1) -- (-1,0) node[dot] {};\draw (-.7,1) node[dot] {} -- (0,0.3) -- (.7,1) node[dot] {};}
\DeclareSymbol{31p}{-3}{\draw (0,0) -- (0,-1) -- (1,0) node[dot] {}; \draw (0,0) -- (0,1.2) node[dot] {}; \draw (-.7,1) node[dot] {} -- (0,0) -- (.7,1) node[dot] {}; 
\draw (0,-1) node{\scaleobj{0.5}{\pez}}; 
\draw (0,-1) node{\scaleobj{0.5}{\pe}}; }
\DeclareSymbol{32p}{-3}{\draw (0,0) -- (0,-1) -- (1,0) node[dot] {}; \draw (0,0) -- (0,-1) -- (-1,0) node[dot] {}; \draw (0,0) -- (0,1.2) node[dot] {}; \draw (-.7,1) node[dot] {} -- (0,0) -- (.7,1) node[dot] {}; 
\draw (0,-1) node{\scaleobj{0.5}{\pez}};
\draw (0,-1) node{\scaleobj{0.5}{\pe}};}
\DeclareSymbol{22p}{-3}{\draw (0,0.3) -- (0,-1) -- (1,0) node[dot] {}; \draw (0,0.3) -- (0,-1) -- (-1,0) node[dot] {};\draw (-.7,1) node[dot] {} -- (0,0.3) -- (.7,1) node[dot] {}; 
\draw (0,-1) node{\scaleobj{0.5}{\pez}};
\draw (0,-1) node{\scaleobj{0.5}{\pe}};}
\DeclareSymbol{21p}{-3}{\draw (0,0.3) -- (0,-1) -- (1,0) node[dot] {}; 
\draw (-.7,1) node[dot] {} -- (0,0.3) -- (.7,1) node[dot] {}; 
\draw (0,-1) node{\scaleobj{0.5}{\pez}};
\draw (0,-1) node{\scaleobj{0.5}{\pe}};}

\DeclareSymbol{21}{-3}{\draw (0,0.3) -- (0,-1) -- (1,0) node[dot] {}; 
\draw (-.7,1) node[dot] {} -- (0,0.3) -- (.7,1) node[dot] {}; 
}

\DeclareSymbol{11p}{-3}{\draw (0,0) node[dot2] {} -- (0,-1) -- (1,0) node[dot] {}; 
\draw (0,0) -- (0,1.2) node[dot] {};  
\draw (0,-1) node{\scaleobj{0.5}{\pez}}; 
\draw (0,-1) node{\scaleobj{0.5}{\pe}}; }
\DeclareSymbol{100}{-3}
{\draw[white] (-.4,0) -- (.4,0); 
\draw (0,0) node[dot2] {} -- (0,-1)  ; 
\draw (0,0) -- (0,1.2) node[dot] {};}  
\usetikzlibrary{mindmap,backgrounds,trees,arrows,decorations.pathmorphing,decorations.pathreplacing,positioning,shapes.geometric,shapes.misc,decorations.markings,decorations.fractals,calc,patterns}

\tikzset{>=stealth',
         cvertex/.style={circle,draw=black,inner sep=1pt,outer sep=3pt},
         vertex/.style={circle,fill=black,inner sep=1pt,outer sep=3pt},
         star/.style={circle,fill=yellow,inner sep=0.75pt,outer sep=0.75pt},
         tvertex/.style={inner sep=1pt,font=\scriptsize},
         gap/.style={inner sep=0.5pt,fill=white}}
\tikzstyle{mybox} = [draw=black, fill=blue!10, very thick,
    rectangle, rounded corners, inner sep=10pt, inner ysep=20pt]
\tikzstyle{boxtitle} =[fill=blue!50, text=white,rectangle,rounded corners]

\tikzstyle{decision} = [diamond, draw, fill=blue!20,
    text width=4.5em, text badly centered, node distance=2.5cm, inner sep=0pt]
\tikzstyle{block} = [rectangle, draw, fill=blue!20,
    text width=2.7cm, text centered, rounded corners, minimum height=3em]

\tikzstyle{line} = [draw, very thick, color=black!50, -latex']

\tikzstyle{cloud} = [draw, ellipse,fill=red!40, 
node distance=2.5cm,
    minimum height=1em]

\tikzstyle{cloud2} = [draw, ellipse,fill=red!30, text=white,text width=10em, node distance=2.5cm, text centered, minimum height=4em]

\tikzstyle{cloud3} = [draw, ellipse, fill=cyan!30, 
node distance=2.5cm,
    minimum height=3em, 
    text width=1.7cm]

\tikzstyle{cloud4} = [draw, ellipse,fill=orange!70, node distance=2.5cm,
   text width=2.1cm,  
           minimum height=3em]

\tikzstyle{cloud5} = [draw, ellipse,fill=red!20, node distance=2.5cm,
    text centered, 
    text width=3.0cm, 
    minimum height=3em]

\tikzstyle{cloud6} = [draw, ellipse,fill=red!20, node distance=2.5cm,
    text width=1.5cm, 
    text centered, 
    minimum height=3em]

\newcommand{\arrow}[2][20]
 {
  \hspace{-5pt}
  \begin{tikzpicture}
   \node (A) at (0,0) {};
   \node (B) at (#1pt,0) {};
   \draw [#2] (A) -- (B);
  \end{tikzpicture}
  \hspace{-5pt}
 }

\tikzset{
    position/.style args={#1:#2 from #3}{
        at=(#3.#1), anchor=#1+180, shift=(#1:#2)
    }
}

\newtheorem{theorem}{Theorem} [section]

\newtheorem{metatheorem}[theorem]{``Theorem''}
\newtheorem{lemma}[theorem]{Lemma}
\newtheorem{proposition}[theorem]{Proposition}
\newtheorem{remark}[theorem]{Remark}

\newcommand{\1}{\hspace{0.2mm}\text{I}\hspace{0.2mm}}
\newcommand{\II}{\text{I \hspace{-2.8mm} I} }

\newcommand{\noi}{\noindent}
\newcommand{\Z}{\mathbb{Z}}
\newcommand{\R}{\mathbb{R}}
\newcommand{\T}{\mathbb{T}}

\let\P= \undefined
\newcommand{\P}{\mathbf{P}}

\newcommand{\E}{\mathbb{E}}

\newcommand{\F}{\mathcal{F}}

\newcommand{\al}{\alpha}
\newcommand{\be}{\beta}
\newcommand{\dl}{\delta}

\newcommand{\nb}{\nabla}

\newcommand{\Dl}{\Delta}
\newcommand{\eps}{\varepsilon}
\newcommand{\kk}{\kappa}
\newcommand{\g}{\gamma}
\newcommand{\G}{\Gamma}
\newcommand{\ld}{\lambda}

\newcommand{\s}{\sigma}
\newcommand{\Si}{\Sigma}
\newcommand{\ft}{\widehat}
\newcommand{\Ft}{{\mathcal{F}}}
\newcommand{\wt}{\widetilde}
\newcommand{\cj}{\overline}
\newcommand{\dx}{\partial_x}
\newcommand{\dt}{\partial_t}

\newcommand{\ta}{\theta}

\renewcommand{\l}{\ell}
\renewcommand{\o}{\omega}
\renewcommand{\O}{\Omega}

\newcommand{\les}{\lesssim}
\newcommand{\ges}{\gtrsim}

\newcommand{\jb}[1]
{\langle #1 \rangle}

\newcommand{\ind}{\mathbf 1}

\newcommand{\N}{\mathbb{N}}
\newcommand{\NN}{\mathcal{N}}

\renewcommand{\H}{\mathcal{H}}

\newtheorem*{ackno}{Acknowledgements}

\newcommand{\I}{\mathcal{I}}

\newcommand{\RR}{\mathcal{R}}
\newcommand{\B}{\mathcal{B}}

\newcommand{\D}{\mathcal{D}}

\newcommand{\C}{\mathcal{C}}
\numberwithin{equation}{section}
\numberwithin{theorem}{section}

\begin{document}
\baselineskip = 14pt

\title[Comparing stochastic NLW and NLH]
{Comparing the stochastic nonlinear wave\\
and heat equations: a case study}

\author[T.~Oh and M.~Okamoto]
{Tadahiro Oh and Mamoru Okamoto}
\address{
Tadahiro Oh, School of Mathematics\\
The University of Edinburgh\\
and The Maxwell Institute for the Mathematical Sciences\\
James Clerk Maxwell Building\\
The King's Buildings\\
Peter Guthrie Tait Road\\
Edinburgh\\ 
EH9 3FD\\
 United Kingdom}

\email{hiro.oh@ed.ac.uk}

\address{
Mamoru Okamoto\\
Division of Mathematics and Physics\\
Faculty of Engineering\\
Shinshu University\\
4-17-1 Wakasato\\
Nagano City 380-8553\\
Japan}


\curraddr{Department of Mathematics\\
 Graduate School of Science\\ Osaka University\\
Toyonaka\\ Osaka\\ 560-0043\\ Japan}

\email{okamoto@math.sci.osaka-u.ac.jp}



\begin{abstract}
We study  the two-dimensional stochastic nonlinear wave equation (SNLW) 
and stochastic nonlinear heat equation (SNLH) with a quadratic nonlinearity, 
forced by a fractional derivative (of order $\al > 0$) of a space-time white noise.
In particular, we show that the well-posedness theory
breaks at $\al = \frac 12$ for SNLW
and at $\al = 1$ for SNLH.
This provides a first example showing that SNLW behaves less favorably than SNLH.
(i)~As for SNLW, 
Deya (2020) essentially proved its local well-posedness 
for $0 < \al < \frac 12$.
We first revisit this argument
and establish  multilinear
smoothing of order $\frac 14$ on the second order stochastic term in the spirit of a recent work by Gubinelli, Koch, and Oh (2018).
This allows us to simplify the local well-posedness
argument 
for some range of~$\al$.
On the other hand, 
when $\al \geq \frac 12$,
we show that 
 SNLW is ill-posed in the sense that 
the second order stochastic term is not a continuous function of time
with values in spatial distributions.
This  shows that a standard method such as the Da Prato-Debussche trick
or its variant,  based on a higher order expansion, breaks down for $\al \ge \frac 12$.
(ii)~As for SNLH, we establish analogous results with a threshold given by $\al = 1$.

These examples show that 
in the case of rough noises,  
 the existing well-posedness theory 
 for singular stochastic PDEs
 breaks down 
 before reaching the critical values
($\al = \frac 34$ in the wave case
and 
 $\al = 2$ in the heat case)
 predicted by 
 the scaling analysis 
 (due to Deng, Nahmod, and Yue (2019) in the wave case
 and due to Hairer (2014) in the heat case).

\end{abstract}


\maketitle

\tableofcontents

\section{Introduction}
\label{SEC:1}

\subsection{Singular stochastic PDEs}

In this paper, we study 
 the following  stochastic nonlinear wave equation (SNLW) on $\T^2 = (\R/2\pi\Z)^2$:
\begin{align}
\begin{cases}
\dt^2 u  + (1 -  \Dl)  u +  u^2  = \jb{\nb}^{\al} \xi\\
(u, \dt u) |_{t = 0} = (u_0, u_1)
\end{cases}
\qquad (x, t) \in \T^2\times \R_+,
\label{SNLW1}
\end{align}

\noi
and the stochastic nonlinear heat equation (SNLH) on $\T^2$:
\begin{align}
\begin{cases}
\dt u + (1 - \Dl) u +  u^2  = \jb{\nb}^\al \xi\\
u|_{t = 0} = u_0
\end{cases}
\qquad (x, t) \in \T^2\times \R_+,
\label{SNLH1}
\end{align}

\noi
where  
$\jb{\nb} = \sqrt{1-\Dl}$ and $\al > 0$.
Namely, both equations are endowed with a quadratic nonlinearity
and forced by an $\al$-derivative of a (Gaussian)  space-time white noise $\xi$
on $\T^2\times \R_+$.

Over the last decade, we have seen a tremendous development
in the study of singular stochastic PDEs, in particular in the parabolic setting
\cite{Hairer0, Hairer, GIP, CC, Kupi, MW1, 
CW,
chandra_analytic_2016, bruned_renormalising_2017,  bruned_algebraic_2016}.
Over the last few years, we have also witnessed a rapid progress
in the theoretical understanding of
 nonlinear wave equations with singular stochastic forcing
 and/or rough random initial data \cite{OTh2, GKO, GKO2, GKOT,  OPTz,ORTz, OOR, OOTz, 
 ORSW, ORW, ORSW2, OOTol, Bring}.
 While the regularity theory in the parabolic setting is well understood, 
 the understanding of the solution theory in the hyperbolic/dispersive setting
 has  been rather poor.  
 This is due to the intricate nature of hyperbolic/dispersive problems, 
 where case-by-case analysis is often necessary
 (for example, to show multilinear smoothing
 as in Proposition~\ref{PROP:sto1} below).
Let us  compare the hyperbolic and parabolic  $\Phi^3_3$-models
 on the three-dimensional torus $\T^3$ as an example.
In the parabolic setting  \cite{EJS}, 
  the standard  Da Prato-Debussche trick suffices
  for local well-posedness, 
  while in the wave setting,  the situation is much more complicated.
  In \cite{GKO2}, 
Gubinelli, Koch, and the first author studied the hyperbolic $\Phi^3_3$-model
by adapting the paracontrolled calculus~\cite{GIP} to the hyperbolic/dispersive setting.
In particular, it was essential to  exploit  multilinear
 smoothing in the construction of stochastic objects and 
 also to introduce paracontrolled operators.
While this comparison on the hyperbolic and parabolic $\Phi^3_3$-model shows that 
it may require more effort to study SNLW than SNLH, 
the resulting outcomes (local well-posedness on $\T^3$ with a quadratic nonlinearity forced by a space-time white noise) are essentially the same.

   The main purpose of this paper is to 
   investigate further  the behavior of solutions 
 to SNLW  and SNLH 
 and study the following question: 
{\it  Does   the solution theory for SNLW  match up with  that for SNLH?}
For this purpose, we study these equations
 in a simpler setting of 
 a quadratic nonlinearity on the two-dimensional torus $\T^2$
 but with noises more singular than  a space-time white noise
 (i.e.~$\al > 0$).
In this setting, we indeed provide a negative answer to the question above.

When $\al = 0$, the equations \eqref{SNLW1} and \eqref{SNLH1}
correspond to the so-called hyperbolic $\Phi^3_2$-model and parabolic $\Phi^3_2$-model,
respectively,\footnote
{Strictly speaking, the hyperbolic $\Phi^3_2$-model  
would require a damping term $\dt u$ in \eqref{SNLW1}.
Since this modification does not change
local well-posedness properties of the equation, 
we simply consider the undamped wave equation \eqref{SNLW1}.}
whose local well-posedness can be obtained by 
the standard Da Prato-Debussche trick; see \cite{DPD2, GKO}.
In this paper, we compare the behavior of solutions to  these equations 
for more singular noises, i.e.~$\al > 0$.
We now state a ``meta''-theorem.

\begin{metatheorem}\label{THM:1}
\textup{(i)} Let $0 <  \al < \frac 12$.
Then, the quadratic SNLW \eqref{SNLW1} is locally well-posed.
When $\al \geq \frac 12$, 
the quadratic SNLW~\eqref{SNLW1} is ill-posed in the sense the 
standard solution theory such as the Da Prato-Debussche trick
or its variant based on a higher order expansion does not work.

\smallskip 
\noi
\textup{(ii)} Let $0 <  \al < 1$.
Then, the quadratic  SNLH \eqref{SNLH1} is locally well-posed.
When $\al \geq 1$, 
the quadratic  SNLH~\eqref{SNLH1} is ill-posed in the sense described above.

\end{metatheorem}

For precise statements, see 
Theorem \ref{THM:WP}, Proposition \ref{PROP:stodiv}, 
Theorem \ref{THM:WP2}, and Proposition~\ref{PROP:heat}.
Let  $\al_* = \frac 12$ for SNLW \eqref{SNLW1} and $ \al_* = 1$
for SNLH \eqref{SNLH1}.
Then, for $0 < \al < \al_*$, 
we prove local well-posedness of the equation via the second order expansion:\footnote{As we see below, 
for small values of $\al$, the first order expansion \eqref{X2} suffices.}
\begin{align}
u = \<1> - \<20> + v. 
\label{X1}
\end{align}

\noi
Here,  we adopt Hairer's convention to denote the stochastic terms by trees;
the vertex ``\,$\<dot>$\,'' in $\<1>$ corresponds to 
the random noise $\jb{\nb}^\al \xi$,
while the edge denotes the Duhamel integral operator:\footnote{In the case
of SNLW, this corresponds to the forward fundamental solution to the linear Klein-Gordon equation:
$\dt^2 u +(1-\Dl) u = 0$.}
 \begin{align*}
 \I = \big(\dt^2 + (1 - \Dl)\big)^{-1}
 \text{ for SNLW} 
 \qquad \text{and}
 \qquad  
\I = \big(\dt + (1 - \Dl)\big)^{-1}
 \text{ for SNLH.}
 \end{align*}
 
 \noi
With this notation, 
the stochastic convolution $\<1>$ 
and the second order stochastic term $\<20>$
can be expressed as
\begin{align}
 \<1> = \I(\jb{\nb}^\al \xi)
\qquad \text{and}
\qquad 
 \<20> = \I (\<2>), 
\label{X1a}
\end{align}

\noi
where  $\<2>$ denotes a renormalized version of $\<1>^2$.
See \eqref{stoc1} and \eqref{stoc1H}
for precise definitions of the stochastic convolutions.
In particular, we impose $\<1>(0) = 0$ in the wave case
and $\<1>(-\infty) = 0$ in the heat case.
We then solve the fixed point problem for the residual term $v
= u - \<1> + \<20>$.
See \eqref{SNLW4} and \eqref{SNLH4}.

On the other hand, for $\al \geq \al_*$
we show that the second order term $\<20>$ does {\it not} belong to 
$C([0, T]; \D'(\T^2))$ for any $T>0$, almost surely
(see Propositions \ref{PROP:stodiv} and \ref{PROP:heat} below).
This implies\footnote
{In some extreme cases, it may be  possible to have $ u \in C([0, T]; \D'(\T^2))$
even if $\<20> \notin C([0, T]; \D'(\T^2))$, namely when the singularities of $\<20>$
and $v$ in \eqref{X1} cancel each other.
We, however, ignore such a ``rare'' case since it is not within the scope of the  standard solution theory,
(where we postulate that $v$ is ``nice'').} 
that a solution $u$ would not belong 
to $C([0, T]; \D'(\T^2))$
if we were to solve the equation via the second (or higher) order expansion \eqref{X1}
or the first order expansion (= the Da Prato-Debussche trick):
\begin{align}
u = \<1> + v
\label{X2}
\end{align}

\noi
since the second order term $\<20>$ appears in 
case-by-case analysis of the nonlinear contribution for the residual term $v = u - \<1>$.

In Subsection \ref{SUBSEC:NLW1},  
we go over details  for SNLW \eqref{SNLW1}.
In Subsection \ref{SUBSEC:heat}, 
we discuss the case of SNLH \eqref{SNLH1}.

\begin{remark}\rm

Our main goal in this paper is to study to what extent
the existing solution theory\footnote{In this paper, we restrict our attention
to the  solution theory based on  the Da Prato-Debussche trick
or its higher order variants.} extends to handle rough noises
in the context of SNLW and SNLH.
For this purpose, we consider the simplest kind
of nonlinearity (i.e.~the quadratic nonlinearity) in \eqref{SNLW1} and~\eqref{SNLH1}.

There are several reasons for considering the ``fractional'' noise $\jb{\nb}^\al \xi$ in \eqref{SNLW1}
and \eqref{SNLH1}.
In studying stochastic PDEs, we often consider a noise
of the form $\Phi \xi$, where $\Phi$ is a bounded operator on $L^2(\T^2)$.
Furthermore, we often assume that  $\Phi$ is Hilbert-Schmidt\footnote{In the Banach space setting, 
we often assume that $\Phi$ is a $\g$-radonifying operator
from $L^2(\T^2)$ to some Banach space $B$.} from $L^2(\T^2)$
to $H^s(\T^2)$. 
See \cite{DPZ14, DD1, OPW}.
It is also common to make  a further assumption
that a noise is spatially homogeneous.
Namely, $\Phi$ is given by a convolution operator.
The Bessel potential $\jb{\nb}^\al$ is one of the simplest 
operator of this kind, which also allows us to tune the (spatial) regularity of the noise.

Since the work \cite{MV}, 
fractional noises have been considered as very natural stochastic perturbation models.
Stochastic PDEs with fractional noises (including $\jb{\nb}^\al \xi$) 
have been  studied by many researchers
(see, for example, \cite{PZ, Dal, TTV, BaTu, DQ, BJQ, Hos, HHLNT, Deya1, Deya2} and the references therein).
In stochastic PDEs, 
the first examples studied in this direction 
are those given by white-in-time fractional-in-space (or colored-in-space) noises
\cite{Walsh, DPZ14, PZ, Dal}.
In view of the close relation of the 
Fourier series representation 
of the noise $\jb{\nb}^\al \xi$ 
and  the fractional-in-space noise
(see Subsection 5.2 in \cite{OST}), 
the models \eqref{SNLW1} and \eqref{SNLH1} provide
good substitutes for 
white-in-time fractional-in-space  noises, 
enabling us to make an essential point
without being bogged down with technical difficulties
related to fractional noises.
See Remark \ref{REM:fBM}
for the case of fractional-in-time (and general fractional) noises.

\end{remark}

\subsection{Stochastic nonlinear wave equation}
\label{SUBSEC:NLW1}

Stochastic nonlinear wave equations 
 have been studied extensively
in various settings; 
see \cite[Chapter 13]{DPZ14} for the references therein.
 In~\cite{GKO},  Gubinelli, Koch, 
and  the first author
considered SNLW  on $\T^2$ with an additive space-time white noise:
\begin{equation}
\label{SNLW1a}
 \dt^2 u + (1- \Dl)  u  + u^k =  \xi, 
\end{equation}

\noi
where $k\geq 2$ is an integer.
The main difficulty of this problem 
comes from the roughness of the space-time white noise.
In particular, 
 the stochastic convolution $\<1>$, 
solving the linear stochastic wave equation:
\begin{equation*}
 \dt^2 \<1> + (1 - \Dl)\<1>   =  \xi, 
\end{equation*}

\noi
is not a classical function but is merely a  distribution
for the spatial dimension $d \geq 2$.
This raises an issue in  making sense of powers $\<1>^k$ and a fortiori of the full nonlinearity 
 $u^k$ in \eqref{SNLW1a}. 
In~\cite{GKO}, by introducing an appropriate time-dependent
renormalization, 
the authors proved local well-posedness of (a renormalized version of) \eqref{SNLW1a} on $\T^2$.
See \cite{GKO2, GKOT, ORTz, OOR, ORSW, ORW, ORSW2}
for further work on SNLW with singular stochastic forcing.
We also  mention the work \cite{Deya1, Deya2} by Deya 
on SNLW with 
more singular (both in space and time) noises on bounded domains in $\R^d$
and the work  \cite{Tolo} on global well-posedness of the cubic SNLW on $\R^2$.

We first state a  local well-posedness result of 
the quadratic SNLW \eqref{SNLW1} on $\T^2$.
Given $N \in \N$, we define the (spatial) frequency projector $\pi_N$  by 
\begin{align}
\pi_N u :=
\sum_{ |n| \leq N}  \ft u (n)  \, e_n,
\label{pi}
\end{align}

\noi
where $\ft u (n)$ denotes the Fourier coefficient of $u$
and $e_n(x) = \frac 1{2\pi}e^{in \cdot x}$ as in \eqref{exp}.
We also set 
 \begin{align}
\H^s(\T^2) = H^s(\T^2)\times H^{s-1}(\T^2).
\label{H1}
 \end{align}

\begin{theorem}\label{THM:WP}
Let  $0<\al<\frac 12$
 and $s>\al$.
 Then, the quadratic SNLW \eqref{SNLW1} on $\T^2$
 is locally well-posed in $\H^s(\T^2)$.
More precisely,  
 there exists a sequence of time-dependent constants $\{\s_N(t)\}_{N\in \N}$
 tending to $\infty$ \textup{(}see \eqref{sigma1} below\textup{)} such that, 
 given 
any  $(u_0, u_1) \in \H^{s}(\T^2)$,  
there exists 
 an almost surely positive stopping time  $T = T(\o)$ such that 
the solution $u_N$ to the following renormalized SNLW with a regularized noise:
\begin{align}
\begin{cases}
\dt^2 u_N + (1-  \Dl)  u_N  +    u_N^2  - \s_N =  \jb{\nb}^\al \pi_N \xi\\
(u_N, \dt u_N)|_{t = 0} = (u_0, u_1)
\end{cases}
\label{SNLW10}
\end{align}

\noi
converges almost surely  to some limiting  process $u \in C([0, T]; H^{-\al -\eps} (\T^2))$
for any $\eps > 0$.

\end{theorem}

In \cite{Deya2}, Deya 
 proved
Theorem \ref{THM:WP}  on bounded domains on $\R^2$
but the same proof essentially applies on $\T^2$.\footnote{One may  invoke the finite speed of propagation
and directly apply the result in \cite{Deya2} to $\T^2$.
We also point out that the paper \cite{Deya2} handles noises with rougher temporal regularity than the space-time white noise and Theorem~\ref{THM:WP} is a subcase of the main result in \cite{Deya2}.}
For $0 < \al < \frac 13$, 
the standard Da Prato-Debussche argument suffices to prove Theorem \ref{THM:WP}.
Indeed, with the first order expansion \eqref{X2}, 
the residual term $v = u - \<1>$ satisfies
\begin{align}
\begin{split}
\dt^2 v  + (1 -  \Dl) v
&=  -(v+\<1>)^2 \\
&= -v^2 - 2 v \<1> - \<2>.
\end{split}
\label{SNLW2}
\end{align}

\noi
At the second equality,  we performed the Wick renormalization: $\<1>^2 \rightsquigarrow \<2>$.
It is easy to see that  $\<1>$ and $\<2>$ have regularities\footnote{In the following, we restrict our attention to spatial regularities.
Moreover,  we use $a-$ 
(and $a+$) to denote $a- \eps$ (and $a+ \eps$, respectively)
for arbitrarily small $\eps > 0$.
If this notation appears in an estimate, 
then  an implicit constant 
is allowed to depend on $\eps> 0$ (and it usually diverges as $\eps \to 0$).}
$-\al -$
and $- 2\al - $, respectively (see Lemma~\ref{LEM:stoconv} below). 
Then, thanks to one degree of smoothing from the wave Duhamel integral operator, 
we expect that $v$ has regularity $1 - 2\al -$.
The restriction $\al < \frac 13$ 
appears from $(1 - 2\al-) + (-\al-) > 0$ in  making sense of the product 
$v \<1>$ in \eqref{SNLW2}.\footnote{Recall that a product
of two functions is defined in general if the sum of the regularities is positive.}
Then, by viewing
\[(u_0, u_1, \<1>, \<2>)\] 

\noi
as a given {\it enhanced data set},\footnote{Namely,  
once we have the pathwise regularity property of the stochastic terms $\<1>$ and $\<2>$, 
we can build a continuous solution map:
$(u_0, u_1, \<1>, \<2>) \mapsto v$ in the deterministic manner.}
one can easily prove local well-posedness of \eqref{SNLW2}.

For $\frac 13 \leq \al < \frac 12$, 
the argument in \cite{Deya2} is based on the second order expansion \eqref{X1}.
In this case, 
the residual term $v = u - \<1> + \<20>$ satisfies
\begin{align}
\begin{split}
\dt^2 v +(1 -  \Dl)  v 
&  =  - (v + \<1> - \<20>)^2 + \<2> \\
& = -  (v  - \<20>)^2 -2 v \<1> + 2 \<20>\<1>.
\end{split}
\label{SNLW4}
\end{align}

\noi
If we proceed with a ``parabolic thinking'',\footnote{Namely, 
if we only count the regularity of each of $\<1>$ in $\<2>$
and put them together with one degree of smoothing
from the wave Duhamel integral operator  {\it without} taking into account the product structure
and the oscillatory nature of the linear wave propagator.}
then we expect that  $\<20>$ has
regularity 
\begin{align*}
1- 2\al - = 2(-\al -) + 1, 
\end{align*}

\noi
where we gain one derivative from the wave Duhamel integral operator;
see \eqref{stoc2}.
With this parabolic thinking, we see that 
the last product 
$ \<20>\<1>$ in \eqref{SNLW4} makes sense (in a deterministic manner)
only for $\al < \frac 13$ so that $(1 - 2\al - ) + (-\al-) > 0$.
Nonetheless, for $\frac 13 < \al < \frac 12$, 
one can use stochastic analysis
to give a meaning to 
$\<21> := \<20> \cdot \<1>$ 
as a random distribution of regularity $- \al -$
(inheriting the bad regularity of $\<1>$).
Using the equation~\eqref{SNLW4}, 
we expect that $v$ has regularity $1 - \al - $ and, with this regularity of $v$, 
 all the terms on the right-hand side of \eqref{SNLW4} make sense.
Then, by viewing
\begin{align}
 \big(u_0, u_1, \<1>, \<20>,  \<21>\big)
\label{data1}
\end{align}

\noi
as a given enhanced data set, 
a standard contraction argument with the energy estimate 
(Lemma \ref{LEM:energy}) yields  local well-posedness of \eqref{SNLW4}.

In view of ``Theorem'' \ref{THM:1}, 
 the restriction $\al < \frac 12$ in Theorem \ref{THM:WP} is sharp.
See Proposition~\ref{PROP:stodiv} below.
There is, however, 
one point that we would like to investigate in this well-posedness part.
In the discussion above, 
we simply used a ``parabolic thinking''
to conclude that  $\<20>$ has regularity (at least) $1 - 2\al-$.
In fact, by exploiting the explicit product structure
and multilinear dispersion, we show that there is an extra smoothing for $\<20>$.

Given $N \in \N$, 
let $\<20>_N$ to denote the second order term, emanating from the truncated noise $\pi_N \jb{\nb}^\al \xi$.
See \eqref{stoc3} for a precise definition.
We then have the following proposition.

\begin{proposition} \label{PROP:sto1}
Let $0<\al < \frac 12$ and 
 $s \in \R$ satisfy
\begin{align}
s < s_\al
:= 1-2\al + \min \big( \al,\tfrac 14 \big)
=
\begin{cases}
1-\al, & \text{if } \al \le \frac 14, \\
\frac 54 -2\al, & \text{if } \al>\frac 14. 
\end{cases}
\label{exreg}
\end{align}

\noi
Then, for any $T>0$,  $\{ \<20>_N \}_{N \in \N}$ is
 a Cauchy sequence
in 
$C([0,T];W^{s,\infty} (\T^2))$ almost surely.
In particular,
denoting the limit by $\<20>$,
we have
  \[\<20> \in C([0,T];W^{s_\al - \eps,\infty} (\T^2))
  \]
  
  \noi
  for any $\eps > 0$, 
  almost surely.
\end{proposition}

See also Proposition \ref{PROP:sto2} below
for another instance of multilinear smoothing.
In \cite{GKO2}, such an extra smoothing property 
on stochastic terms via multilinear dispersion effect 
played an essential role in the study of the quadratic SNLW on the three-dimensional torus $\T^3$.
We believe that the multilinear smoothing  in
Proposition \ref{PROP:sto1}
is itself of interest since such a multilinear smoothing
in the stochastic context for the wave equation
 is not well understood.
See also Remark \ref{REM:smooth} below.

In our current setting, 
this extra smoothing does {\it not} improve  the range of $\al$ in Theorem~\ref{THM:WP}
since, as we will show below, the range $\al < \frac 12$ is sharp.
Proposition~\ref{PROP:sto1}, however, allows
us to 
simplify the local well-posedness argument for the range $\frac 13\leq \al < \frac 5{12}$. 
While the discussion above showed the  Da Prato-Debussche
argument to study \eqref{SNLW2} breaks down at $\al = \frac 13$, 
the extra smoothing  in Proposition \ref{PROP:sto1}
allows us to study \eqref{SNLW2} 
at the level of the Duhamel formulation:
\begin{align}
\begin{split}
 v
&= S(t)(u_0, u_1) -\I(v^2 + 2 v \<1>)  - \I(\<2>)\\
&= S(t)(u_0, u_1) -\I(v^2 + 2 v \<1>)  - \<20>, 
\end{split}
\label{SNLW5}
\end{align}

\noi
where $S(t)$ denotes the linear wave propagator
defined in \eqref{LW3a}.
Thanks to Proposition~\ref{PROP:sto1}, 
we expect that $v$ has regularity $\frac 54 - 2\al - $,
thus allowing us to make
sense of the product 
$ v \<1>$ as long as $\frac 13 \leq \al < \frac 5{12}$, 
i.e.~$(\frac 54 - 2\al -) + (-\al-) > 0$.
In this refined Da Prato-Debussche argument, 
the relevant enhanced data set is given by 
\begin{align}
 \big(u_0, u_1, \<1>, \<20>\big).
\label{data2}
\end{align}

\noi
See Theorem~\ref{THM:LWP} for a precise statement.

Alternatively, we may work with the second order expansion \eqref{X1}
and study the equation~\eqref{SNLW4}.
In this case, Proposition \ref{PROP:sto1}
allows us to 
make sense of the product $\<20>\<1>$
in the {\it deterministic} manner for $\al < \frac 5{12}$.
This in particular shows that for the range $\frac 13\leq \al < \frac 5{12}$, 
we can solve \eqref{SNLW4} for $v = u - \<1> + \<20>$
with a smaller enhanced data set  in \eqref{data2}.
Namely, 
when $\al < \frac 5 {12}$, 
there is no need to a priori prescribe the last term $\<21>$ in \eqref{data1}.
See Theorem~\ref{THM:LWPv}\,(i) for a precise statement.

For the range of $\al$ under consideration, i.e.~$\al \geq   \frac 13$, 
the extra gain of regularity in Proposition~\ref{PROP:sto1}
is  $\frac 14$, regardless of the value of $\al$.
When $\frac 5{12} \leq \al < \frac 12$, this extra smoothing is unfortunately not sufficient to 
make sense of the product $\<20>\<1>$ in the deterministic manner.
Recalling the paraproduct decomposition (see \eqref{para1} below), 
we see that the resonant product $\<21p> := \<20>\pe \<1>$
is the only issue here.
Thus, for $\frac 5{12} \leq \al < \frac 12$, 
we solve \eqref{SNLW4} 
with an enhanced data set:
\begin{align*}
 \big(u_0, u_1, \<1>, \<20>, \<21p>\big),
\end{align*}

\noi
where
 we use stochastic analysis
to give a meaning to the problematic resonant product $\<21p>$;
see Proposition \ref{PROP:sto2}.

\begin{remark}\label{REM:smooth}\rm

Note that the extra smoothing is at most $\frac 14$ in Proposition \ref{PROP:sto1}, while a $\frac 12$-smoothing was shown on $\T^3$ in \cite{GKO2}.
This $\frac 14$-difference in two- and three-dimensions 
seems to come from the effect of  Lorentz transformations along null directions.
The same situation appears in bilinear estimates for solutions to the linear wave equation;
see,  for example, Subsection 3.6 in \cite{DFS}.
See also Remark \ref{REM:exsmooth} for a further discussion, 
where (i) we show that our computation on $\T^2$ is essentially sharp
and (ii) we compute the maximum possible gain of regularity on $\T^d$, $d \geq 3$.
Lastly, we point out that 
 Proposition \ref{PROP:sto1} states that the extra smoothing vanishes as $\al \to 0$.
\end{remark}

Next, let us consider the situation for $\al \geq \frac 12$.
In \cite[Proposition 1.4]{Deya2}, 
Deya  showed that 
$\E \big[ \|\<2>_N(t)\|_{H^s}^2 \big]$ diverges for any $s \in \R$, 
when $\al \ge \frac 12$.
This can be used to show that 
the Wick power $\<2>$ is not a distribution-valued function of time
when $\al \geq \frac 12$.
The following proposition shows that the same result holds for 
$\<20>$.

\begin{proposition} \label{PROP:stodiv}
Let $\al \ge \frac 12$.
Then,  given any $T> 0$,  $\{ \<20>_N \}_{N \in \N}$
forms a divergent sequence in $C([0,T]; \mathcal{D}'(\T^2))$ almost surely.
\end{proposition}

We point that Proposition \ref{PROP:stodiv} is by no means to be expected
from the bad behavior of $\<2>$ for $\al \geq \frac 12$.
For example,
in the parabolic $\Phi^4_3$-model, 
it is well known that the cubic Wick power $\<3>$ does not make sense as a distribution-valued function of time but that $\<30> = (\dt - \Dl)^{-1} \<3>$ belongs
to $C(\R; \C^{\frac 12-}(\T^3))$; see \cite{EJS, MWX}.\footnote{A more fundamental example
of this kind may be  the space-time white noise $\xi$ which does not make sense as a distribution-valued function of time,
while we can define the stochastic convolution $\I (\xi)$ as a distribution-valued
function by a limiting procedure.}
Furthermore, in Proposition \ref{PROP:heat} below, 
we prove that, for the quadratic SNLH \eqref{SNLH1},  
(i)~the Wick power $\<2>$ is not a distribution-valued function for $\al \geq \frac 12$
but (ii)~$\<20>$ in the heat case makes sense as a distribution-valued  function for $\al < 1$.
Therefore, we find it rather intriguing
that for the wave equation, 
both $\<2>$ and $\<20>$
have the same threshold $\al = \frac 12$.

In the proof of Proposition \ref{PROP:stodiv}, 
we  show that each Fourier coefficient $\ft{\<20>}_N(n,t)$  diverges almost surely for $\al \ge \frac 12$.
See Remark~\ref{REM:divcov}.
This divergence comes from the high-to-low energy transfer.
Namely, the divergence comes from 
the nonlinear interaction of two incoming high-frequency waves
resulting in a low-frequency wave.\footnote{In \eqref{div2x}, 
this corresponds to the interaction of two functions 
of spatial frequencies $k$ and $n-k$
giving rise to an output function 
of spatial frequency $n$ with 
$|n| \ll |k| \sim |n-k| $.}
Such high-to-low energy transfer was exploited in 
proving  ill-posedness of the deterministic nonlinear wave equations in negative Sobolev spaces; see \cite{CCT2b, OOTz, FOk}.

\begin{remark}\rm
(i) The proof of  Proposition \ref{PROP:stodiv} also applies to  $\T^d$.
See Remark~\ref{REM:divcov}\,(ii) for details.
In particular,
$\{ \<20>_N \}_{N \in \N}$
forms a divergent sequence  in $C([0,T]; \mathcal{D}'(\T^d))$ 
almost surely for  $\al \ge 1-\frac d4$.

\smallskip

\noi
(ii) It is interesting 
to note that we can prove local well-posedness
of SNLW \eqref{SNLW1}
for the entire range $0 < \al < \frac 12$ {\it without} using the paracontrolled approach 
as in the three-dimensional case~\cite{GKO2}.

\end{remark}

\begin{remark}\label{REM:scaling1}\rm

In a recent preprint \cite{DNY}, 
Deng, Nahmod, and Yue  introduced the notion of probabilistic scaling 
and the associated critical regularity.
This is based on the observation
that the Picard second iterate\footnote{More precisely, 
the Picard second iterate minus the linear solution
(= the Picard first iterate).
For the sake of simplicity,  however, we  refer to this as the Picard second iterate.
For example, a Picard iteration scheme for SNLW \eqref{SNLW1}
with the zero initial data
yields the $j$th Picard iterates $P_j$, $j = 1, 2$,  given by 
$P_1 = \<1>$
and $P_2 =\<1> - \I(\<2>) = \<1> - \<20>$. 
For simplicity, we refer to 
$ \<20>$ as the Picard second iterate in the following
(where we also dropped the insignificant $-$ sign).

In the random data well-posedness theory
for the quadratic NLW:
$ \dt^2 u  + (1 -  \Dl)  u +  u^2  = 0$
with random initial data $(u_0^\o, u_1^\o)$, 
the first two Picard iterates are given by 
$P_1 = S(t) (u_0^\o, u_1^\o)$
and $P_2 = P_1 - \I(P_1^2)$, 
where a proper renormalization is applied to $P_1^2$.
Once again, for the sake of simplicity, 
we refer to $P_2 - P_1 
=\I(P_1^2) =  \I\big((S(t) (u_0^\o, u_1^\o))^2\big)$
as the Picard second iterate in this discussion.} should be (at least) as smooth
as a stochastic convolution (or a random linear solution in the context
of the random data well-posedness theory).
In their terminology, 
 the quadratic SNLW \eqref{SNLW1} on $\T^2$
 is critical when $\al_* = \frac 34$.
 Proposition~\ref{PROP:stodiv}, however, 
 shows that  the Picard second iterate $\<20>$ is not well defined for $\al \ge \frac 12$
in the sense that 
 each Fourier coefficient 
$\ft{\<20>}_N(n, t)$  diverges as $N \to \infty$.\footnote{While
 we work on the quadratic
 nonlinear wave equation (NLW) with a stochastic forcing, 
 the same divergence result  also holds for the quadratic NLW
 with random initial data considered in \cite{DNY}.}
This in particular  implies that 
 the existing solution theory such as the Da Prato-Debussche trick
or its higher order variants \cite{BOP3, OPTz}\footnote{This includes the paracontrolled approach
used in \cite{GKO2}.} 
breaks down at $\al = \frac 12$ {\it before} reaching the critical value $\al_* = \frac 34$.
See also Remark~\ref{REM:divcov}
for the general $d$-dimensional case.

We now make several remarks.
(i) The discrepancy between the critical value $\al_* = \frac 34$
predicted by the probabilistic scaling and the actual value $\al = \frac 12$
for the non-existence of the Picard second iterate $\<20>$ (in the limiting sense)
stems from  the fact that, as discussed in \cite{DNY}, 
the probabilistic scaling only takes into account
several simple\footnote{A ``critical'' value should be something which 
can be computed in advance without too much difficulty.
In this sense, the simplification made in \cite{DNY} in capturing 
main interactions seems appropriate.} interactions (high-to-high and high-to-low)
in computing a critical value.
In the proof of Proposition 
 \ref{PROP:stodiv}, we make a more precise computation
 in proving the divergence of the Picard second iterate.
 (ii) As we see in the next subsection, 
 an analogous phenomenon occurs for the quadratic SNLH~\eqref{SNLH1} on $\T^2$.
More precisely,  while the critical value predicted by the scaling analysis
for \eqref{SNLH1}
 is $\al = 2$, the Picard second iterate $\<20>$
 fails to exist already at $\al = 1$ in the heat case.
See  Remark~\ref{REM:scaling2} below.
In both the wave and heat cases, 
this pathological behavior 
(i.e.~the divergence of the Picard second iterate and thus
the breakdown of the existing solution theory) 
before reaching the predicted critical values
seems to be closely related to the fact that we are dealing with {\it very} rough noises
(rougher than the space-time white noise).
This is in particular relevant 
in studying a stochastic PDE (or a deterministic PDE with random initial data)
with a nonlinearity of low degree
(and also in low dimensions).
For example, 
we may expect  a similar discrepancy 
for the  nonlinear Schr\"odinger equation (NLS)
with a quadratic nonlinearity:
\begin{align*}
i \dt u   -  \Dl  u +  \NN(u, \cj u)  = 0
\end{align*}

\noi
with rough random initial data, 
where $\NN(u, \cj u)  = u^2$, $\cj u^2$, or $|u|^2$
(with a proper renormalization).

\end{remark}

\subsection{Stochastic nonlinear heat equation}
\label{SUBSEC:heat}

In this subsection, we go over the corresponding results for the quadratic SNLH \eqref{SNLH1} on $\T^2$.
With $\I = \big(\dt + (1 - \Dl)\big)^{-1}$, 
let $\<1>$ and $\<20>$ be as in~\eqref{X1a}
and  $\<2>$ be  the Wick renormalization of $\<1>^2$.
We first state the crucial regularity result for the stochastic terms.

\begin{proposition} \label{PROP:heat}
{\rm (i)} 
For  $0 < \al < \frac 12$ and $\eps>0$,
$\{ \<2>_N \}_{N \in \N}$ is a Cauchy sequence
in $C(\R_+;\C^{-2\al- \eps} (\T^2))$
almost surely.
 In particular, 
denoting the limit by $\<2>$,
we have
  \[\<2> \in C(\R_+;\C^{-2\al- \eps}(\T^2))
  \]
almost surely.
On the other hand, 
for $\al \geq \frac 12$,
$\{ \<2>_N \}_{N \in \N}$
forms a divergent sequence 
in $C([0,T]; \mathcal{D}'(\T^2))$ for any $T>0$, almost surely.
Here, $\C^{s}(\T^2)$ denotes the H\"older-Besov space defined in \eqref{besov1}.

\smallskip
\noi
{\rm (ii)}
For  $0 < \al <  1$ and $\eps>0$,
$\{ \<20>_N \}_{N \in \N}$ is
 a Cauchy sequence
in $C(\R_+;\C^{2-2\al- \eps} (\T^2))$
almost surely.
 In particular, 
denoting the limit by $\<20>$,
we have
  \[\<20> \in C(\R_+;\C^{2-2\al- \eps}(\T^2))
  \]
almost surely.
On the other hand, 
for $\al \geq 1$,
$\{ \<20>_N \}_{N \in \N}$
forms a divergent sequence 
in $C([0,T]; \mathcal{D}'(\T^2))$ for any $T>0$, almost surely.

\end{proposition}

In short, Proposition \ref{PROP:heat} states that 
$\<2>$ is a distribution-valued function
if and only if $\al < \frac 12$, 
while 
$\<20>$ is a distribution-valued function
if and only if $\al < 1$.
Hence, for the range $\frac 12 \leq \al < 1$, 
while $\<2>(t)$ does not make sense as a spatial distribution, 
 $\<2> = (\dt + (1-\Dl)) \<20>$ makes sense as a space-time distribution. 
As mentioned above, such a phenomenon
is already known for the parabolic $\Phi^4_3$-model; see \cite{EJS, MWX}.
Proposition \ref{PROP:heat}
exhibits sharp contrast with the situation
for SNLW discussed earlier (Proposition \ref{PROP:stodiv} above), 
where the threshold $\al = \frac 12$ applies
to both $\<2>$ and $\<20>$.

We now state a sharp local well-posedness result for the quadratic SNLH \eqref{SNLH1}.

\begin{theorem}\label{THM:WP2}
Let  $0<\al<1$
 and $s> - \al - \eps $
 for sufficiently small $\eps > 0$.
 Then, the quadratic SNLH \eqref{SNLH1} on $\T^2$
 is locally well-posed in $\C^s(\T^2)$.
More precisely,  
 there exists a sequence of constants $\{\kk_N\}_{N\in \N}$
 tending to $\infty$ \textup{(}see \eqref{kap} below\textup{)} such that, 
 given 
any  $u_0 \in \C^{s}(\T^2)$,  
there exists 
 an almost surely positive stopping time  $T = T(\o)$ such that 
the solution $u_N$ to the following renormalized SNLH:
\begin{align*}
\begin{cases}
\dt u_N + (1-  \Dl)  u_N  +    u_N^2  - \kk_N =  \jb{\nb}^\al \pi_N \xi\\
u_N|_{t = 0} = u_0
\end{cases}
\end{align*}

\noi
converges almost surely to some limiting  process $u \in C([0, T]; \C^{-\al -\eps} (\T^2))$
for any $\eps > 0$.

\end{theorem}

In \cite{DPD2}, Da Prato and Debussche proved Theorem \ref{THM:WP2}
for $\al = 0$.
The same proof based on the Da Prato-Debussche trick also applies
for  $0 < \al < \frac 23$.
In this case, with the first order expansion~\eqref{X2}, 
the residual term $v = u - \<1>$ satisfies
\begin{align}
\begin{split}
\dt v  + (1 -  \Dl) v
&= -v^2 - 2 v \<1> - \<2>, 
\end{split}
\label{SNLH3}
\end{align}

\noi
where   $\<1>$ and $\<2>$ have regularities
$-\al -$
and $- 2\al - $, respectively.
Then, 
by repeating the analysis in the previous subsection
with two degrees of smoothing coming from the heat Duhamel integral operator, 
 $v$ has expected regularity $2 - 2\al - $
and thus the restriction $\al < \frac 23$ 
appears from $(2 - 2\al-) + (-\al-) > 0$ in  making sense of the product 
$v \<1>$ in \eqref{SNLH3}.
Then,    local well-posedness of \eqref{SNLH3}
easily follows
with an enhanced data set
$(u_0,  \<1>, \<20>)$.

For $\frac 23 \leq \al < 1$, 
the proof of Theorem \ref{THM:WP2} is based on  the second order expansion \eqref{X1}
and proceeds exactly as in the wave case (but without any multilinear smoothing).\footnote{Since there is no multilinear smoothing for the heat equation, ``parabolic thinking'' provides a correct insight.}
In this case, 
the residual term $v = u - \<1> + \<20>$ satisfies
\begin{align}
\begin{split}
\dt v +(1 -  \Dl)  v 
& = -  (v  - \<20>)^2 -2 v \<1> + 2 \<20>\<1>.
\end{split}
\label{SNLH4}
\end{align}

\noi
When $\al \geq \frac 23$, 
we can not make sense of the last product 
$\<20>\<1>$ in the deterministic manner. 
Using stochastic analysis, we can give a meaning to 
$\<20>\<1>$ 
as a distribution of regularity $-\al -$
for $\frac 23 \leq  \al < 1$.
See Lemma \ref{LEM:heat2}.
In this case, $v$ has expected regularity of $2 - \al - $
and thus the restriction $\al < 1$
also appears in making sense of the product $v\<1>$,
namely from $(2 - \al - ) + (-\al - ) > 0$.
Then, by applying
the standard Schauder estimate, 
we can easily  prove local well-posedness of \eqref{SNLH4}
with an enhanced data set:
\begin{align*}
 \big(u_0,  \<1>, \<20>,  \<21p>\big).
\end{align*}

\begin{remark}\rm
Let us compare the situations for SNLW \eqref{SNLW1}
and SNLH \eqref{SNLH1}.
In this discussion, we disregard initial data.
For the quadratic SNLH \eqref{SNLH1}, 
the required enhanced data set
consists  of 
$\<1>$ and $\<20>$
when $0 \leq  \al < \frac 23$. Namely, it involves only (the powers of)
the first order process $\<1>$.
When $\frac 23 \leq \al < 1$, 
it also involves the second order and the third order processes $\<20>$ and  $\<21p>$.
It is interesting to note that 
for the quadratic SNLW \eqref{SNLW1}, 
thanks to the multilinear smoothing effect (Proposition \ref{PROP:sto1}), 
there is now an intermediate regime $\frac 13\leq \al < \frac 5{12}$, 
where the required enhanced data set in \eqref{data2}
involves only the first and second order processes
(but not the third order process).
Furthermore, 
in this range, 
while the usual Da Prato-Debussche argument with \eqref{SNLW2}
fails, 
the refined 
Da Prato-Debussche argument \eqref{SNLW5} at the level of the Duhamel formulation works
thanks to the multilinear smoothing in Proposition \ref{PROP:sto1}.

\end{remark}

\begin{remark}\label{REM:scaling2} \rm

Consider the following scaling-invariant model
for the quadratic SNLH~\eqref{SNLH1}:
\begin{align*}
\dt u   - \Dl u +  u^2  = |\nb|^\al \xi.
\end{align*}

\noi
As in \cite{MW1},  
we now apply a scaling argument to find  a critical value of $\al $.
By applying the following parabolic scaling
(and the associated white noise scaling for $\xi$): 
\[
\wt{u}(x,t) = \ld^{\al} u(\ld x, \ld^2 t)
\qquad \text{and} \qquad
\wt{\xi}(x,t) = \ld^2 \xi (\ld x, \ld^2 t)
\]
for $\ld>0$, 
we obtain 
\[
\dt \wt{u} - \Dl \wt{u} +  \ld^{2-\al} \wt{u}^2 = |\nb|^{\al} \wt{\xi}.
\]

\noi
Then, by taking $\ld \to 0$, the nonlinearity formally vanishes when $\al < 2$.
This provides the critical value of $\al_* = 2$, 
(which agrees with the notion of local subcriticality introduced in~\cite{Hairer}).
It is very intriguing that, for the quadratic SNLH \eqref{SNLH1}, 
 the solution theory 
based on  the Da Prato-Debussche trick
or its higher order variants 
breaks
down at $\al = 1$ {\it before} reaching the critical value $\al_* = 2$.
See \cite{Hos} for a similar phenomenon in the context
of the KPZ equation with a fractional noise.
For dispersive equations including the quadratic SNLW, the scaling analysis 
as above does not seem to provide any useful insight,\footnote{For example, 
applying the hyperbolic scaling $(x, t) \mapsto (\ld x, \ld t)$, 
the scaling invariant version of SNLW~\eqref{SNLW1}
yields a critical value of $\al = \frac 52$, even higher than the heat case
but  the well-posedness theory for SNLW breaks down at $\al = \frac 12$.}
unless appropriate integrability conditions are  incorporated.  See,
 for example,~\cite{FOW} for a discussion in the case of the stochastic nonlinear Schr\"odinger equation.\footnote{In a recent preprint \cite{DNY}, a notion of  probabilistic scaling was
 introduced.
 While the criticality associated with this notion seems to 
 provide a good intuition for many problems, 
 it does not provide a good prediction  for the quadratic SNLW \eqref{SNLW1}.
 See  Remark~\ref{REM:scaling1}.
 }

\end{remark}

\begin{remark}\label{REM:fBM} \rm
Lastly, we state a remark on SNLW and SNLH with a fractional-in-time noise.
The space-time white noise $\xi$ in \eqref{SNLW1} and \eqref{SNLH1}
is given by a distributional time derivative of 
the $L^2$-cylindrical Wiener process $W$ (see \eqref{Wpro} below).
We may instead consider
a noise  $\xi^{H}= \dt W^H$ induced by 
a (spatially white) fractional-in-time Brownian motion $W^{H}$
with the Hurst parameter $0<H<1$.
When $H = \frac 12$, 
the noise $\xi^H$ reduces to the usual space-time white noise $\xi$.

We recall that  the stochastic convolution $\I (\xi^H) = \I_\text{heat} (\xi^H)$, 
emanating from the fractional-in-time noise $\xi^H$, 
has 
(spatial) regularity $2H-1-$.
See, for example, Theorem~4 in \cite{TTV}.
Namely, 
SNLH \eqref{SNLH1}
with 
the noise  $\jb{\nb}^\al \xi$ 
formally corresponds to the quadratic SNLH with 
the fractional-in-time noise $\xi^H$
with the Hurst parameter $H= \frac{1-\al}2$
and the well-posedness result in Theorem~\ref{THM:WP2} for $\al < 1$
seems to carry to the fractional-in-time noise case with $0 < H< \frac 12$.
Note that the threshold value $\al = 1$ in Theorem \ref{THM:WP2}
(and Proposition~\ref{PROP:heat}\,(ii)) corresponds
to the $H= 0$ case, which we do not discuss here.  
See, for example, \cite{FKS} for the study of the fractional Brownian motion
with $H = 0$ (which is a Gaussian process with stationary increments and logarithmic increment structure).

In the case of the wave equation, 
  the stochastic convolution $\I (\xi^H) = \I_\text{wave} (\xi^H)$, 
emanating from the fractional-in-time noise $\xi^H$, 
has 
(spatial) regularity $H-\frac 12-$.
See Proposition 1.2 in~\cite{Deya1}.
Thus, 
SNLW \eqref{SNLW1}
formally corresponds to the quadratic SNLW with 
the fractional-in-time noise $\xi^H$
with the Hurst parameter $H= \frac{1}2 - \al$.
In this case, 
local  well-posedness of the quadratic SNLW
with the fractional-in-time noise $\xi^H$ 
is known to hold for  $0 < H< \frac 12$
(corresponding to the range $0 < \al < \frac 12$
in  Theorem~\ref{THM:WP}).
See \cite{Deya1, Deya2}.
Note that the threshold value $\al = \frac 12$ in the wave case
also  corresponds
to the $H= 0$ case.
We also point out that in this fractional-in-time noise case, 
the regularities of  the stochastic convolutions $\I_\text{wave} (\xi^H)$
and $\I_\text{heat} (\xi^H)$ for the wave and heat equations
agree only when $H = \frac 12$. 

It is also possible to consider 
a noise  $\xi^{\vec H}= \dt W^{\vec H}$ coming from 
a  space-time fractional Brownian motion $W^{\vec H}$
with the Hurst parameter $\vec H = (H_0, H_1, H_2)$, 
$0 < H_j < 1$, where $H_0$ corresponds to the temporal direction 
and $H_1$ and $H_2$ correspond to the two spatial directions.
See,  for example, 
\cite{Deya1, Deya2} in the wave case.
In this setting, the threshold value $\al = \frac 12$
for SNLW~\eqref{SNLW1}
corresponds to $H_0 + H_1 + H_2 = 1$
and in this case, we expect the divergence of $\<20>_N$.\footnote{In 
the fractional noise case, a subscript $N$ signifies
that it is a stochastic process, coming from a certain  approximation
$W^{\vec H}_N$
of $W^{\vec H}$.  See \cite{Deya1, Deya2}.

In \cite{Deya2}, 
the divergence of $\<2>_N$ is established 
for $H_0 + H_1 + H_2 \le 1$.
In the dispersive setting, however, it is more important to study
the property of $\<20>_N$, i.e.~$\<2>_N$ under the Duhamel integral operator
since a common practice in dispersive PDEs
is to make sense of a product under the Duhamel integral operator, exploiting multilinear 
dispersion. For example, if we consider the stochastic cubic NLS on $\T$, 
forced by a space-time white noise: $i \dt u - \dx^2 u + |u|^2 u = \xi$
 (with a proper renormalization), 
then the renormalized product $\<3> = \, :\!|\I(\xi)|^2 \I(\xi)\!:$ of the three copies
of the stochastic convolution $\I(\xi) = \I_\text{Schr\"odinger}(\xi)$ does not make sense
as a distribution-valued function.
On the other hand, it is not difficult to see that the Picard second iterate
$\<30> = \I( :\!|\I(\xi)|^2 \I(\xi)\!:)$ is a well defined distribution-valued function of time.
The strength of Proposition \ref{PROP:stodiv} lies in showing that 
$\<20>_N$ indeed diverges at the same threshold as $\<2>_N$
(which is not something we expect commonly in the study of dispersive PDEs).}
In the heat case, 
 the stochastic convolution $\I (\xi^{\vec H}) = \I_\text{heat} (\xi^{\vec H})$, 
emanating from the space-time fractional noise $\xi^{\vec H}$, 
has 
(spatial) regularity $2H_0 + H_1 + H_2 -2-$.
In this case, the threshold value $\al = 1$
for SNLH~\eqref{SNLH1}
corresponds to $2H_0 + H_1 + H_2 = 1$, 
at which we expect an analogous 
divergence of the second order process $\<20>$
(in the limiting sense). 
We do not pursue this direction in this paper.

\end{remark}

This paper is organized as follows.
In Section \ref{SEC:2}, we introduce some notations and recall useful lemmas.
In 
Section \ref{SEC:3}, assuming the regularity properties
of the stochastic objects, 
we prove local well-posedness of SNLW \eqref{SNLW1}
(Theorem \ref{THM:WP}).
We then present details of the construction
of the stochastic objects in Section \ref{SEC:sto1}.
In particular, we prove 
the multilinear smoothing for $\<20>$
(Proposition \ref{PROP:sto1}) and 
divergence of $\<20>$ (Proposition \ref{PROP:stodiv}).
Finally, in Section~\ref{SEC:heat}, 
we present  proofs of Proposition \ref{PROP:heat} and Theorem \ref{THM:WP2}.

\section{Basic  lemmas}
\label{SEC:2}

In this section, we introduce some notations and go over basic lemmas.

\subsection{Notations}
We set
\begin{align}
e_n(x) := \frac1{2\pi}e^{in\cdot x}, 
\qquad  n\in\Z^2, 
\label{exp}
\end{align}
 
\noi 
for the orthonormal Fourier basis in $L^2(\T^2)$. 
Given $s \in \R$, we  define the  Sobolev space  $H^s (\T^2)$ by  the norm:
\[
\|f\|_{H^s(\T^2)} = \|\jb{n}^s\ft{f}(n)\|_{\l^2(\Z^2)},
\]

\noi
where $\ft{f}(n)$ is the Fourier coefficient of $f$ and  $\jb{\,\cdot\,} = (1+|\cdot|^2)^\frac{1}{2}$. 
We then set 
$\H^s(\T^2) = H^s(\T^2)\times H^{s-1}(\T^2)$
as in \eqref{H1}.
Similarly, given $s \in \R$ and $p \geq 1$, 
we define
 the  $L^p$-based Sobolev space (Bessel potential space) 
 $W^{s, p}(\T^2)$
 by the norm:
\[\| f\|_{W^{s, p}} = \| \jb{\nb}^s f \|_{L^p} = \big\| \F^{-1}( \jb{n}^s \ft f(n))\big\|_{L^p}.\]

\noi
When $p = 2$, we have $H^s(\T^2) = W^{s, 2}(\T^2)$. 
When we work with space-time function spaces, we use short-hand notations such as
 $C_TH^s_x = C([0, T]; H^s(\T^2))$.

For $A, B > 0$, we use $A\lesssim B$ to mean that
there exists $C>0$ such that $A \leq CB$.
By $A\sim B$, we mean that $A\lesssim B$ and $B \lesssim A$.
We also use  a subscript to denote dependence
on an external parameter; for example,
 $A\les_{\al} B$ means $A\le C(\al) B$,
  where the constant $C(\al) > 0$ depends on a parameter $\al$.

\subsection{Besov spaces and paraproduct estimates}

Given $j \in \N_0 := \N\cup\{0\}$,  let $\P_{j}$
 be the (non-homogeneous) Littlewood-Paley projector
 onto the (spatial) frequencies $\{n \in \Z^2: |n|\sim 2^j\}$
 such that 
 \[ f = \sum_{j = 0}^\infty \P_jf.\]

\noi
We then define 
 the Besov spaces $B^s_{p, q}(\T^2)$
 by the norm:
\begin{equation*}
\| f \|_{B^s_{p,q}} = \Big\| 2^{s j} \| \P_{j} f \|_{L^p_x} \Big\|_{\l^q_j(\N_0)}.
\end{equation*}

\noi
Note that  $H^s(\T^2) = B^s_{2,2}(\T^2)$.
We also define the H\"older-Besov space by setting
\begin{align}
\C^s (\T^2)= \B^s_{\infty,\infty}(\T^2).
\label{besov1}
\end{align}

Next, we recall the following paraproduct decomposition due to 
 Bony~\cite{Bony}.
See \cite{BCD, GIP} for further details.
 Given two functions $f$ and $g$ on $\T^2$
of regularities $s_1$ and $s_2$, respectively,
we write the product $fg$ as
\begin{align}
\begin{split}
fg 
& 
= f\pl g + f \pe g + f \pg g \\
 : \! & = \sum_{j < k-2} \P_{j} f \, \P_k g
+ \sum_{|j - k|  \leq 2} \P_{j} f\,  \P_k g
+ \sum_{k < j-2} \P_{j} f\,  \P_k g.
\end{split}
\label{para1}
\end{align}

\noi
The first term 
$f\pl g$ (and the third term $f\pg g$) is called the paraproduct of $g$ by $f$
(the paraproduct of $f$ by $g$, respectively)
and it is always well defined as a distribution
of regularity $\min(s_2, s_1+ s_2)$.
On the other hand, 
the resonant product $f \pe g$ is well defined in general 
only if $s_1 + s_2 > 0$.

We have the following product estimates.
See \cite{BCD, MW2} for details of the proofs in the non-periodic case
(which can be easily extended to the current periodic setting).

\begin{lemma}\label{LEM:para}
\textup{(i) (paraproduct and resonant product estimates)}
Let $s_1, s_2 \in \R$ and $1 \leq p, p_1, p_2, q \leq \infty$ such that 
$\frac{1}{p} = \frac 1{p_1} + \frac 1{p_2}$.
Then, we have 
\begin{align}
\| f\pl g \|_{B^{s_2}_{p, q}} \les 
\|f \|_{L^{p_1}} 
\|  g \|_{B^{s_2}_{p_2, q}}.  
\notag
\end{align}

\noi
When $s_1 < 0$, we have
\begin{align}
\| f\pl g \|_{B^{s_1 + s_2}_{p, q}} \les 
\|f \|_{B^{s_1 }_{p_1, q}} 
\|  g \|_{B^{s_2}_{p_2, q}}.  
\notag
\end{align}

\noi
When $s_1 + s_2 > 0$, we have
\begin{align}
\| f\pe g \|_{B^{s_1 + s_2}_{p, q}} \les 
\|f \|_{B^{s_1 }_{p_1, q}} \|  g \|_{B^{s_2}_{p_2, q}}.  
\notag
\end{align}

\noi
\textup{(ii)}
Let $s_1 <  s_2 $ and $1\leq p, q \leq \infty$.
Then, we have 
\begin{align} 
\| u \|_{B^{s_1}_{p,q}} 
&\les \| u \|_{W^{s_2, p}}.
\notag
\end{align}

\end{lemma}

\subsection{Product estimates
and discrete convolutions}

Next, we recall the following product estimates.
See \cite{GKO} for the proof.

\begin{lemma}\label{LEM:bilin}
 Let $0\le \al \le 1$.

\smallskip

\noi
\textup{(i)} Suppose that 
 $1<p_j,q_j,r < \infty$, $\frac1{p_j} + \frac1{q_j}= \frac1r$, $j = 1, 2$. 
 Then, we have  
\begin{equation*}  
\| \jb{\nb}^\al (fg) \|_{L^r(\T^d)} 
\les \Big( \| f \|_{L^{p_1}(\T^d)} 
\| \jb{\nb}^\al g \|_{L^{q_1}(\T^d)} + \| \jb{\nb}^\al f \|_{L^{p_2}(\T^d)} 
\|  g \|_{L^{q_2}(\T^d)}\Big).
\end{equation*}

\smallskip

\noi
\textup{(ii)} 
Suppose that  
 $1<p,q,r < \infty$ satisfy the scaling condition:
$\frac1p+\frac1q\leq \frac1r + \frac{\al}d $.
Then, we have
\begin{align*}
\| \jb{\nb}^{-\al} (fg) \|_{L^r(\T^d)} \les \| \jb{\nb}^{-\al} f \|_{L^p(\T^d) } 
\| \jb{\nb}^\al g \|_{L^q(\T^d)}.  
\end{align*}

\end{lemma}

Note that
while  Lemma \ref{LEM:bilin} (ii) 
was shown only for 
$\frac1p+\frac1q= \frac1r + \frac{\al}d $
in \cite{GKO}, 
the general case
$\frac1p+\frac1q\leq \frac1r + \frac{\al}d $
follows from a straightforward modification
of the proof.

We also recall the following basic lemma on a discrete convolution.

\begin{lemma}\label{LEM:SUM}
\textup{(i)}
Let $d \geq 1$ and $\al, \be \in \R$ satisfy
\[ \al+ \be > d  \qquad \text{and}\qquad \al, \be < d.\]
\noi
Then, we have
\[
 \sum_{n = n_1 + n_2} \frac{1}{\jb{n_1}^\al \jb{n_2}^\be}
\les \jb{n}^{d - \al - \be}\]

\noi
for any $n \in \Z^d$.

\smallskip

\noi
\textup{(ii)}
Let $d \geq 1$ and $\al, \be \in \R$ satisfy $\al+ \be > d$.
\noi
Then, we have
\[
 \sum_{\substack{n = n_1 + n_2\\|n_1|\sim|n_2|}} \frac{1}{\jb{n_1}^\al \jb{n_2}^\be}
\les \jb{n}^{d - \al - \be}\]

\noi
for any $n \in \Z^d$.

\end{lemma}

Note that we do not have the restriction $\al, \be < d$ in the resonant case (ii).
Lemma \ref{LEM:SUM} follows
from elementary  computations.
See, for example,  Lemmas 4.1 and 4.2 in \cite{MWX}.

\subsection{Linear estimates}

In this subsection, we recall linear estimates for the wave and heat equations. 
First, we state the energy estimate for solutions to the nonhomogeneous linear wave equation
$\T^d$:
\begin{align}
\begin{cases}
\dt^2 u + (1-\Dl) u = F\\
(u, \dt u) |_{t = 0} = (u_0, u_1).
\end{cases}
\label{LW}
\end{align}

\noi
By writing \eqref{LW} in the Duhamel formulation, we have
\begin{align}
u (t) = S(t) (u_0, u_1) 
+ \I (F)(t), 
\label{LW3}
\end{align}

\noi
where 
the linear wave propagator $S(t)$ is defined by 
\begin{align}
S(t) (u_0, u_1) =   \cos (t \jb{\nb}) u_0 + \frac{\sin (t \jb{\nb})}{\jb{\nb}} u_1 
\label{LW3a}
\end{align}

\noi
and
the wave Duhamel integral operator $\I$ is defined by 
\begin{align}
\I(F) (t) 
&=  \int_0^t \frac{\sin ((t-t')\jb{\nb})}{\jb{\nb}}F(t') dt'.
\label{LW4}
\end{align}

\noi
Then, the following energy estimate follows from 
 \eqref{LW3}, \eqref{LW4}, 
and 
the unitarity of the linear wave propagator $S(t)$ in $\H^s(\T^d)$.

\begin{lemma}\label{LEM:energy}
Let $s \in \R$.
Then, the solution $u$ to \eqref{LW} satisfies
\begin{align*}
\| u \|_{L_T^\infty H_x^s}
\les  \| (u_0,u_1) \|_{\H^s} + \| F \|_{L_T^1 H_x^{s-1}}
\end{align*}

\noi
for any $T>0$.
\end{lemma}

In \cite{GKO, ORSW}, the authors used the Strichartz estimates 
to study local well-posedness of the stochastic nonlinear wave equations.
Note, however, that 
the Strichartz estimates are not needed 
for proving local well-posedness of  the quadratic NLW  
 in two dimensions.
More precisely, the energy estimate (Lemma \ref{LEM:energy}), Sobolev's inequality, and 
a standard 
contraction argument yield local well-posedness of the quadratic NLW in $H^s(\T^2)$ for $s>0$.

Next, we recall the Schauder estimate for the heat equation.
Let $P (t)  = e^{-t (1-\Dl)}$ denote the linear heat propagator defined 
as a Fourier multiplier operator:
\begin{align}
P(t) f = \sum_{n \in \Z^2} e^{-t \jb{n}^2} \ft f(n) e_n
\label{heat1}
\end{align}

\noi
for $t \geq 0$.
Then, we have the following Schauder estimate
on $\T^d$.

\begin{lemma} \label{LEM:Schauder}
Let $ -\infty< s_1 \leq  s_2 < \infty$.
Then, we have 
\begin{align}
\| P(t) f \|_{\C^{s_2}}
\les t^{\frac{s_1 -s_2}{2}} \| f \|_{\C^{s_1}}
\label{regheat}
\end{align}

\noi
for any $t > 0$. 
\end{lemma}

The bound \eqref{regheat} on $\T^d$ 
follows from the decay estimate for the heat kernel on $\R^d$ 
(see Lemma~2.4 in \cite{BCD}) and the Poisson summation formula
to pass such a decay estimate to $\T^d$.

\subsection{Tools from stochastic analysis}

Lastly, we recall useful lemmas from stochastic analysis.
Let $\{ g_n \}_{n \in \N}$ be a sequence of independent standard Gaussian random variables defined on a probability space $(\O, \F, P)$, where $\mathcal{F}$ is the $\s$-algebra generated by this sequence. 
Given $k \in \N_0$, 
we define the homogeneous Wiener chaoses $\mathcal{H}_k$ 
to be the closure (under $L^2(\O)$) of the span of  Fourier-Hermite polynomials $\prod_{n = 1}^\infty H_{k_n} (g_n)$, 
where
$H_j$ is the Hermite polynomial of degree $j$ and $k = \sum_{n = 1}^\infty k_n$.
We also set
\begin{align*}
 \H_{\leq k} = \bigoplus_{j = 0}^k \H_j
 \end{align*}

\noi
 for $k \in \N$.

We say that a stochastic process $X:\R_+ \to \mathcal{D}'(\T^d)$
is spatially homogeneous  if  $\{X(\cdot, t)\}_{t\in \R_+}$
and $\{X(x_0 +\cdot\,, t)\}_{t\in \R_+}$ have the same law for any $x_0 \in \T^d$.
Given $h \in \R$, we define the difference operator $\dl_h$ by setting
\begin{align*}
\dl_h X(t) = X(t+h) - X(t).
\end{align*}

\noi
The following lemma will be used in studying regularities of stochastic objects.
For the proof, see Proposition 3.6 in \cite{MWX} and Appendix in \cite{OOTz}.
In the following, we state the result
in terms of the Sobolev space $W^{s, \infty}(\T^d)$ but the same result
holds for the H\"older-Besov space $\C^s(\T^d)$.

\begin{lemma}\label{LEM:reg}
Let $\{ X_N \}_{N \in \N}$ and $X_0$ be spatially homogeneous stochastic processes
$:\R_+ \to \mathcal{D}'(\T^d)$.
Suppose that there exists $k \in \N$ such that 
  $X_N(t)$ and $X_0(t)$ belong to $\H_{\leq k}$ for each $t \in \R_+$.

\smallskip
\noi\textup{(i)}
Let $t \in \R_+$.
If there exists $s_0 \in \R$ such that 
\begin{align}
\E\big[ |\ft X_0(n, t)|^2\big]\les \jb{n}^{ - d - 2s_0}
\label{reg1}
\end{align}

\noi
for any $n \in \Z^d$, then  
we have
$X_0(t) \in W^{s, \infty}(\T^d)$, $s < s_0$, 
almost surely.
Furthermore, if there exists $\g > 0$ such that 
\begin{align}
\E\big[ |\ft X_N(n, t) - \ft X_M(n, t)|^2\big]\les N^{-\g} \jb{n}^{ - d - 2s_0}
\label{reg2}
\end{align}

\noi
for any $n \in \Z^d$ and $M\geq N \geq 1$, 
then 
$\{X_N(t)\}_{N\in\N }$ is a Cauchy sequence in  $W^{s, \infty}(\T^d)$, $s < s_0$, 
almost surely, 
thus converging to some limit $X(t)$ in 
 in  $W^{s, \infty}(\T^d)$.

\noi

\smallskip
\noi\textup{(ii)}
Let $T > 0$ and suppose that \textup{(i)} holds on $[0, T]$.
If there exists $\theta \in (0, 1)$ such that 
\begin{align}
 \E\big[ |\dl_h \ft X_0(n, t)|^2\big]
 \les \jb{n}^{ - d - 2s_0+ \theta}
|h|^\theta, 
\label{reg3}
\end{align}

\noi
for any  $n \in \Z^d$, $t \in [0, T]$, and $h \in [-1, 1]$,\footnote{We impose $h \geq - t$ such that $t + h \geq 0$.}
then we have 
$X_0 \in C([0, T]; W^{s, \infty}(\T^d))$, 
$s < s_0 - \frac \theta2$,  almost surely.
Furthermore, 
if there exists $\g > 0$ such that 
\begin{align}
 \E\big[ |\dl_h \ft X_N(n, t) - \dl_h \ft X_M(n, t)|^2\big]
 \les N^{-\g}\jb{n}^{ - d - 2s_0+ \theta}
|h|^\theta, 
\label{reg4}
\end{align}

\noi
for any  $n \in \Z^d$, $t \in [0, T]$,  $h \in [-1, 1]$, and $M \geq N \geq 1$, 
then 
$\{X_N\}_{N\in \N}$ 
is a Cauchy sequence  in $C([0, T]; W^{s, \infty}(\T^d))$, $s < s_0 - \frac{\theta}{2}$,
almost surely, 
thus converging to some process $X$
in $C([0, T]; W^{s, \infty}(\T^d))$.
\end{lemma}

Lastly, we recall the following Wick's theorem.
See Proposition I.2 in \cite{Simon}.

\begin{lemma}\label{LEM:Wick}	
Let $g_1, \dots, g_{2n}$ be \textup{(}not necessarily distinct\textup{)}
 real-valued jointly Gaussian random variables.
Then, we have
\[ \E\big[ g_1 \cdots g_{2n}\big]
= \sum  \prod_{k = 1}^n \E\big[g_{i_k} g_{j_k} \big], 
\]

\noi
where the sum is over all partitions of $\{1, \dots, 2 n\}$
into disjoint pairs $(i_k, j_k)$.
\end{lemma}

\section{Stochastic nonlinear wave equation with rough noise}
\label{SEC:3}

In this section, we consider SNLW \eqref{SNLW1}.
We first state the regularity properties
of the relevant stochastic terms 
and  reformulate the problem in terms of the residual term 
$v = u - \<1>$ or 
$v = u - \<1> + \<20>$.
We then 
 present a proof of Theorem \ref{THM:WP}.
The analysis of the stochastic terms will be presented in Section \ref{SEC:sto1}.

\subsection{Reformulation of SNLW}
\label{SUBSEC:reno}

Let  $W$ denote a cylindrical Wiener process on $L^2(\T^2)$:
\begin{align}
 W(t) & =  \sum_{n\in \Z^2}
\be_n (t) e_n,
\label{Wpro}
\end{align}

\noi
where $\{ \be_n\}_{n \in \Z^2} $ is a family of  mutually independent 
 complex-valued Brownian
motions
on a fixed probability space $(\O, \F, P)$
conditioned
 so that\footnote{In particular, we take $\be_0$ to be real-valued.}  $\be_{-n} = \cj{\be_n}$, $n \in \Z^2$. 
By convention, we normalize $\be_n$ such that 
$\text{Var}(\be_n(t)) = t$.
Then, 
 the stochastic convolution $\<1> = \I(\jb{\nb}^\al \xi) $ in the wave case
 can be formally written as 
\begin{align}
\<1>  
 = \int_{0}^t \frac{\sin ((t-t')\jb{\nb})}{\jb{\nb}^{1-\al}}dW(t')
 = \sum_{n \in \mathbb{Z}^2} e_n 
  \int_0^t \frac{\sin ((t - t') \jb{ n })}{\jb{ n }^{1-\al}} d \beta_n (t').
\label{stoc1}
\end{align}

\noi
We indeed construct the stochastic convolution $\<1>$ in \eqref{stoc1}
as the  limit of the truncated stochastic convolution  $\<1>_N$ defined by 
\begin{align}
\<1>_N
:= \sum_{\substack{n \in \mathbb{Z}^2 \\ |n| \le N}} e_n 
  \int_0^t \frac{\sin ((t - t') \jb{ n })}{\jb{ n }^{1-\al}} d \beta_n (t').
\label{so4a}
\end{align}

\noi
See Lemma \ref{LEM:stoconv} below.
We then define  the Wick power~$\<2>_N$
 by 
\begin{align}
\<2>_N   := (\<1>_N)^2 - \s_N, 
\label{so4b}
\end{align}

\noi
where $\s_N$ is given by 
\begin{align}
\s_N(t) & = \E\big[ (\<1>_N(x, t))^2\big]
=  \frac{1}{4\pi^2} \sum_{|n|\leq N}
\int_0^t \bigg[\frac{\sin((t - t')\jb{n})}{\jb{n}^{1-\al}} \bigg]^2 dt' \notag\\
& = \frac{1}{8\pi^2} \sum_{|n|\leq N} \bigg\{\frac{t}{\jb{n}^{2(1-\al)}} - \frac{\sin(2t \jb{n})}{2\jb{n}^{3-2\al}}\bigg\}
\sim t N^{2\al}
\label{sigma1}
\end{align}

\noi
for $\al>0$.
We have the following regularity
and convergence properties of $\<1>_N$ and $\<2>_N$
whose proofs are presented in Section \ref{SEC:sto1}.

\begin{lemma}\label{LEM:stoconv}
Let  $T >0$.

\smallskip

\noi
\textup{(i)}
For any $\al \in \R$ and $s<-\al$,
$\{ \<1>_N \}_{N \in \N}$ defined in \eqref{so4a} is
 a Cauchy sequence
in 
$C([0,T];W^{s,\infty}(\T^2))$ almost surely.
In particular,
denoting the limit by $\<1>$,
we have
  \[\<1> \in C([0,T];W^{-\al - \eps,\infty}(\T^2))
  \]
  
  \noi
  for any $\eps > 0$, 
  almost surely.

\smallskip

\noi
\textup{(ii)}
For any $0<\al<\frac{1}{2}$ and $s<-2\al$, 
$\{ \<2>_N \}_{N \in \N}$ defined in \eqref{so4b} is
 a Cauchy sequence
in 
$C([0,T];W^{s,\infty}(\T^2))$
almost surely.
In particular,
denoting the limit by $\<2>$,
we have
 \[\<2> \in C([0,T];W^{-2\al - \eps,\infty}(\T^2))\]
 
 \noi
 for any $\eps > 0$, 
 almost surely.
\end{lemma}

Next, we define the second order stochastic term $\<20>$ by
\begin{align}
  \<20> (t)
  :=  \I(\<2>) (t) = \int_{0}^t \frac{\sin ((t-t')\jb{\nb})}{\jb{\nb}}\<2>(t') dt',
\label{stoc2}
\end{align}

\noi
Then, 
Proposition \ref{PROP:sto1}
shows that 
 $ \<20>$ is a well-defined distribution
 and is a limit of the truncated version: 
\begin{align}
\<20>_N\ = \I(\<2>_N), 
\label{stoc3}
\end{align}

\noi
provided that $0 < \al < \frac 12$.

Next, we give a meaning to the third order process 
``\,$\<21> =  \<20>\<1>$\,''.
As mentioned in Section~\ref{SEC:1}, 
we need to use stochastic analysis for this purpose
when $\frac 5{12} \leq \al < \frac 12$.
Formally write the product 
 $\<20>\<1>$ as
 \[  \<20>\<1> =  \<20> \pl \<1> + \<20>\pe \<1> + \<20>\pg \<1>. \]

\noi
The paraproducts
$\<20> \pl \<1>$ and $\<20>\pg \<1>$
are always well defined as long as each of 
$\<20>$ and $ \<1>$ is well defined.
Thus, we need 
 stochastic analysis 
 only to  give a meaning to 
the resonant product $\<20> \pe \<1>$.

\begin{proposition}\label{PROP:sto2}
Let $0<\al<\frac 12$ and  
$s<s_\al-\al$, where $s_\al$ is as in \eqref{exreg}.
Set
$\<21p>_N := \<20>_N\pe \<1>_N$.
Then, given $T >0$, 
$\{ \<21p>_N \}_{N \in \N}$ is
 a Cauchy sequence
in 
$C([0,T];W^{s,\infty}(\T^2))$
almost surely.
In particular,
denoting the limit by $\<21p>$,
we have
  \[\<21p> \in C([0,T];W^{s_\al - \al - \eps, \infty}(\T^2))\]
  
  \noi
  for any $\eps > 0$, 
  almost surely.
\end{proposition}

Recall from Section \ref{SEC:1} that 
the standard Da Prato-Debussche trick yields local well-posedness
of SNLW \eqref{SNLW1}
for  $0 < \al  < \frac 13$.
When $\frac 13 \leq \al <  \frac{5}{12}$, 
we 
use the first order expansion~\eqref{X2}
and study  the Duhamel formulation \eqref{SNLW5}.

\begin{theorem}\label{THM:LWP}
Let $\frac 13 \le \al < \frac 5{12}$ and $s > \al$.
Then, the equation \eqref{SNLW5} 
 is locally well-posed
in $\H^s(\T^2)$.
More precisely, 
given any $(u_0, u_1) \in \H^s(\T^2)$, 
there exists an almost surely positive stopping time $T =T(\o)$
such that there exists  a unique solution $v \in C([0, T];H^\s(\T^2))$ to~\eqref{SNLW5}, 
where  $\s > \al$ is sufficiently close to  $\al$.
Furthermore, 
the solution $v$ 
depends  continuously 
on the enhanced data set\textup{:}
\begin{align}
\Xi = \big(u_0, u_1, \<1>, \<20>\big)
\notag
\end{align}

\noi
almost surely belonging to 
 the class\textup{:}
\begin{align}
\mathcal{X}^{s, \eps}_T
& = \H^s(\T^2) \times 
C([0,T]; W^{-\al - \eps, \infty}(\T^2)) 
\times 
C([0,T]; W^{s_\al - \eps, \infty}(\T^2))
\label{Edata1}
\end{align}

\noi
for some small $\eps = \eps(\al, s) > 0$.
Here,  $s_\al$ is as in \eqref{exreg}.

\end{theorem}

When $\frac 5{12} \leq \al <  \frac{1}{2}$, 
we use 
 the second order expansion \eqref{X1}
and study 
the equation~\eqref{SNLW4}
satisfied by 
the residual term $ v =  u - \<1> + \<20>$. 
With the paraproduct decomposition \eqref{para1}, 
we write \eqref{SNLW4} as 
\begin{equation} \label{SNLW7}
\begin{cases}
\dt^2v + (1-\Dl)v =
-v^2 - 2v (\<1>-\<20>) - \<20>^2 + 2 (\<20> \pl \<1> + \<21p> + \<20> \pg \<1>)   \\
(v, \dt v) |_{t = 0} = (u_0, u_1).
\end{cases}
\end{equation}

\noi
We now state  local well-posedness of the perturbed SNLW \eqref{SNLW7}
for the entire range $0 < \al < \frac 12$.

\begin{theorem}\label{THM:LWPv}
Let $0<\al<\frac 12$ and $s> \al$.
Then,
the Cauchy problem \eqref{SNLW7} is locally well-posed
in $\H^s(\T^2)$.
More precisely, 
given any $(u_0, u_1) \in \H^s(\T^2)$, 
there exists an almost surely positive stopping time $T =T(\o)$
such that there exists  a unique solution $v \in C([0, T];H^\s(\T^2))$ to \eqref{SNLW7}, 
where $\s \leq s$ and $\al < \s <1-\al$.
Furthermore, we have the following continuous dependence statements
for some small $\eps = \eps(\al, s) > 0$.

\begin{itemize}
\item[\textup{(i)}]
For $0 < \al < \frac 5{12}$, 
the solution $v$ 
depends  continuously 
on the enhanced data set\textup{:}
\begin{align}
\Xi = \big(u_0, u_1, \<1>, \<20>\big)
\notag
\end{align}

\noi
almost surely belonging to 
 the class $\mathcal{X}^{s, \eps}_T$ defined \eqref{Edata1}.

\smallskip
\item[\textup{(ii)}]
For $\frac 5{12}\leq  \al < \frac 12$, 
the solution $v$ 
depends  continuously 
on the enhanced data set\textup{:}
\begin{align}
\Xi = \big(u_0, u_1, \<1>, \<20>,  \<21p>\big)
\label{data4}
\end{align}

\noi
almost surely belonging to 
 the class\textup{:}
\begin{align*}
\mathcal{Y}^{s, \eps}_T
& = \H^s(\T^2) \times 
C([0,T]; W^{-\al - \eps, \infty}(\T^2)) \notag\\
& \hphantom{X}
\times 
C([0,T]; W^{s_\al - \eps, \infty}(\T^2)) 
\times
C([0,T]; W^{s_\al-\al - \eps, \infty}(\T^2)).
\end{align*}

\end{itemize}

\end{theorem}

In Subsection \ref{SUBSEC:LWPv}, 
we present a proof of Theorem \ref{THM:LWPv}.
In view of the pathwise regularities of the relevant stochastic terms, 
we simply build a continuous map, 
sending the enhanced data set $\Xi$
to a solution $v$ in the deterministic manner.

\begin{remark}\rm
If we take $\Xi = \big(u_0, u_1, \<1>, \<20>,  \<21p>\big)$ as the enhanced data set, 
then in order to prove  Theorem \ref{THM:LWPv}, 
it is enough to 
have $\<21p> \in C([0,T];W^{-\al-,\infty}(\T^2))$ almost surely.
Namely, we do not need to exploit 
the extra multilinear smoothing for $\<21p>$.
See Remark \ref{REM:21reg} for a further discussion.
While it is possible to replace 
$\<21p>$ in \eqref{data4} by 
$\<21> : =  \<20> \pl \<1> + \<21p> + \<20>\pg \<1>$, 
we chose not to do so in order to emphasize the fact that 
the resonant product $\<21p>$ is the only term
which needs to be defined a priori.
(As mentioned above, 
given $\<20>$ and $ \<1>$, 
the paraproducts
$\<20> \pl \<1>$ and $\<20>\pg \<1>$
are well-defined distributions.)


We point out, however, that, in Theorem~\ref{THM:LWPv}\,(i),  the extra smoothing on $\<20>$
plays an essential role
in making sense of the product $\<20> \<1>$ in the deterministic manner in the range
$0 < \al < \frac 5{12}$.

\end{remark}

We conclude this subsection by 
presenting a proof of Theorem \ref{THM:WP}.

\begin{proof}[Proof of Theorem \ref{THM:WP}]
We only consider the case $\frac 5{12}\leq \al<\frac 12$ and $s>\al$.
For $(u_0,u_1) \in \H^s(\T^2)$ and $N \in \N$, we set 
\[\Xi_N:= \big(u_0, u_1, \<1>_N, \<20>_N,  \<21p>_N\big).\]

\noi
By Theorem \ref{THM:LWPv}\,(ii), there exists a unique local-in-time solution $v_N \in C([0,T]; H^{\s}(\T^2))$ to~\eqref{SNLW7} with the enhanced data set $\Xi_N$, 
where 
 $\al < \s <1-\al$.
Then, we see that 
\[
u_N =v_N+ \<1>_N-\<20>_N
\]
satisfies the renormalized SNLW \eqref{SNLW10}.

It follows 
from Lemma \ref{LEM:stoconv} and Propositions \ref{PROP:sto1} and \ref{PROP:sto2}
 that $\Xi_N$ converges almost surely to 
 $\Xi$ in \eqref{data4} 
 with respect to the  $\mathcal{Y}^{s, \eps}_T$-topology.
In particular, 
 Theorem \ref{THM:LWPv}\,(ii) implies that the sequence $v_N = u_N - \<1>_N+\<20>_N$ converges to $v \in C([0,T]; H^{\s}(\T^2))$.
Then, we conclude that $u_N$ converges to
\[
u= v+\<1>-\<20> 
\]

\noi
in $C([0,T]; H^{-\al-\eps}(\T^2))$.
\end{proof}

\subsection{Proof of Theorem \ref{THM:LWP}}
\label{SUBSEC:LWP}

In the following, we study the Duhamel formulation \eqref{SNLW5}.
Let $\frac 13 \le \al < \frac 5{12}$ and
 $0 < T\leq 1$ and fix $\eps > 0$ sufficiently small.
Define a map $\G$ by 
\begin{align*}
\G (v) (t) &= S(t) (u_0, u_1) 
 - \I \big( v^2 + 2v \<1>)(t) - \<20>(t)
\end{align*}

\noi
where $S(t)$ and $\I$ are as in  \eqref{LW3a} and  \eqref{LW4}.
In the following, we take $\al < \s\leq  s$.

By the energy estimate (Lemma \ref{LEM:energy}), we have 
\begin{align}
\begin{split}
\| \G (v) \|_{L_T^\infty H_x^\s}
&\les \| (u_0,u_1) \|_{\H^s} + T \Big( \| v^2 \|_{L_T^\infty H_x^{\s-1}} 
+ \| v \<1> \|_{L_T^\infty H_x^{\s-1}}\Big)  
+ \| \<20> \|_{L_T^\infty H_x^{\s}}.
\end{split}
\label{xLWP0}
\end{align}

\noi
By Sobolev's inequality, we have
\begin{align}
\| v^2 \|_{L_T^\infty H_x^{\s-1}}
\les \| v^2 \|_{L_T^\infty L_x^{\frac{2}{2-\s}}}
= \| v \|_{L_T^\infty L_x^{\frac{4}{2-\s}}}^2
\les \| v \|_{L_T^\infty H_x^\s}^2
\label{xLWP1}
\end{align}

\noi
for  $0<\s<1$.
From 
 Lemmas~\ref{LEM:bilin} and~\ref{LEM:stoconv},  
we have
\begin{align}
\begin{split}
\| v\<1> \|_{L_T^\infty H_x^{\s-1}}
&\les \| \jb{\nb}^{-\al-\eps} (v\<1>) \|_{L_T^\infty L_x^2} \\
&\les \| \jb{\nb}^{\al+\eps} v \|_{L_T^{\infty} L_x^2} \| \jb{\nb}^{-\al-\eps} \<1> \|_{L_T^\infty L_x^{\infty}}  \\
&\les C_\o \| v \|_{L_T^\infty H_x^\s} 
\end{split}
\label{xLWP2}
\end{align}

\noi
for some almost surely finite constant $C_\o > 0$,
provided that  $\al<\s<1-\al$.

Then, from 
\eqref{xLWP0}, \eqref{xLWP1}, \eqref{xLWP2}, and Proposition \ref{PROP:sto1}, we obtain
\begin{align}
\| \G (v) \|_{L_T^\infty H_x^\s}
&\les \| (u_0,u_1) \|_{\H^s} 
 + T \Big( \| v \|_{L_T^\infty H_x^\s} + C_\o \Big)^2 + C_\o
\label{xLWP5}
\end{align}

\noi
Note that we need $\s \leq  \frac 54 - 2\al$ in estimating the last term in \eqref{xLWP0}.
This can be guaranteed by taking $\s > \al$ sufficiently close to $\al$
as long as $\al < \frac 5{12}$.
Similarly, we have
\begin{align}
\| \G (v_1) - \G (v_2) \|_{L_T^\infty H_x^\s}
&\les T \Big( \| v_1 \|_{L_T^\infty H_x^\s} + \| v_2 \|_{L_T^\infty H_x^\s} 
+ C_\o \Big) \| v_1-v_2 \|_{L_T^\infty H_x^\s}.
\label{xLWP6}
\end{align}

\noi
Therefore, we conclude from 
 \eqref{xLWP5} and \eqref{xLWP6}
 that a standard contraction argument  yields local  well-posedness of \eqref{SNLW5}.
Moreover, an analogous  computation 
shows  that the solution $v \in C([0,T];H^\s(\T^2))$ depends continuously on
the enhanced data set $\Xi = \big(u_0, u_1, \<1>, \<20>\big)$.
This completes  the proof of Theorem \ref{THM:LWP}.

\subsection{Proof of Theorem \ref{THM:LWPv}}
\label{SUBSEC:LWPv}

Next, we study the perturbed SNLW \eqref{SNLW7}.
Let $\al < \s\leq  s$ and  $0 < T\leq 1$ and fix $\eps > 0$ sufficiently small.
Define a map $\wt \G$ by 
\begin{align*}
\wt \G (v) (t) &= S(t) (u_0, u_1) 
 - \I \big( v^2 + 2v (\<1>-\<20>) + \<20>^2 - 2 \<21> \big)(t), 
\end{align*}

\noi
where $\<21>$ stands for 
\begin{align}
\<21> =  \<20> \pl \<1> +  \<21p> 
+  \<20>\pg \<1>. 
\label{Z1}
\end{align}

\noi
For $0 < \al < \frac 5{12}$, 
we see from \eqref{exreg} that $s_\al - \al > 0$.
Hence, 
Proposition \ref{PROP:sto1} and Lemma~\ref{LEM:stoconv}
with 
Lemma \ref{LEM:para} imply that 
$\<21>  = \<20> \<1>$
belongs to 
$C([0,T]; W^{-\al - \eps, \infty}(\T^2))$, almost surely.
On the other hand, 
for $\frac 5{12} \leq \al < \frac 12$, 
Proposition \ref{PROP:sto2}
implies 
\begin{align}
\<21> \in 
C([0,T]; W^{-\al - \eps, \infty}(\T^2))
\label{Z2}
\end{align}

\noi
almost surely.

By the energy estimate (Lemma \ref{LEM:energy}), we have 
\begin{align}
\begin{split}
\| \wt \G (v) \|_{L_T^\infty H_x^\s}
&\les \| (u_0,u_1) \|_{\H^s} + T \Big( \| v^2 \|_{L_T^\infty H_x^{\s-1}} + \| v (\<1>-\<20>) \|_{L_T^\infty H_x^{\s-1}} \\
&\qquad
+ \| \<20>^2 \|_{L_T^\infty H_x^{\s-1}}
+ \| \<21> \|_{L_T^\infty H_x^{\s-1}}
\Big).
\end{split}
\label{LWP0}
\end{align}

\noi
Proceeding as in \eqref{xLWP2}
with Proposition \ref{PROP:sto1},  
we have
\begin{align}
\begin{split}
\| v(\<1>-\<20>) \|_{L_T^\infty H_x^{\s-1}}
&\les \| \jb{\nb}^{-\al-\eps} (v(\<1>-\<20>)) \|_{L_T^\infty L_x^2} \\
&\les \| \jb{\nb}^{\al+\eps} v \|_{L_T^{\infty} L_x^2} \| \jb{\nb}^{-\al-\eps} (\<1>-\<20>) \|_{L_T^\infty L_x^{\infty}}  \\
&\les C_\o \| v \|_{L_T^\infty H_x^\s} 
\end{split}
\label{LWP2}
\end{align}

\noi
for some almost surely finite constant $C_\o > 0$,
provided that  $\al<\s<1-\al$.
From Proposition \ref{PROP:sto1}, we also have\begin{align}
&\| \<20>^2 \|_{L_T^\infty H_x^{\s-1}} 
\le \| \<20> \|_{L_T^\infty L_x^\infty}^2  \leq C_\o.
\label{LWP3}
\end{align}

Putting together 
\eqref{LWP0},  \eqref{LWP2}, and \eqref{LWP3} with  \eqref{xLWP1} and \eqref{Z2},  we obtain
\begin{align*}
\| \wt \G (v) \|_{L_T^\infty H_x^\s}
&\les \| (u_0,u_1) \|_{\H^s} 
 + T \Big( \| v \|_{L_T^\infty H_x^\s} + C_\o \Big)^2.
\end{align*}

\noi
Similarly, we have
\begin{align*}
\| \wt \G (v_1) - \wt \G (v_2) \|_{L_T^\infty H_x^\s}
&\les T \Big( \| v_1 \|_{L_T^\infty H_x^\s} + \| v_2 \|_{L_T^\infty H_x^\s} 
+ C_\o \Big) \| v_1-v_2 \|_{L_T^\infty H_x^\s}.
\end{align*}

\noi
The rest follows as in the previous subsection.
This completes  the proof of Theorem \ref{THM:LWPv}.

\section{On the construction of the relevant stochastic objects}
\label{SEC:sto1}

In this section, we 
go over the construction 
of the stochastic terms for SNLW \eqref{SNLW1}.
As in~\cite{GKO2}, our strategy is to estimate
the second moment of the Fourier coefficient
and apply Lemma \ref{LEM:reg}.
In Subsection \ref{SUBSEC:31},  we briefly 
discuss the regularity and convergence properties
of  $\<1>$ and $\<2>$
(Lemma~\ref{LEM:stoconv}).
By exploiting multilinear dispersive smoothing
for  $\<20>$, 
we then present a proof of Proposition~\ref{PROP:sto1}
in Subsection~\ref{SUBSEC:32}.
In Subsection \ref{SUBSEC:33}, 
we establish  analogous multilinear smoothing for  $\<21p>$ (Proposition \ref{PROP:sto2}).
Lastly, in Subsection \ref{SUBSEC:34}, 
we show that,
when $\al \geq \frac 12$, 
the second order stochastic term  $\<20> (t)$ is not a spatial distribution almost surely 
for any $t >0$ (Proposition \ref{PROP:stodiv}).

Let  $\<1> = \I(\jb{\nb}^\al \xi)$ 
be the stochastic convolution  defined in \eqref{stoc1}.
Given $n \in \Z^2$ and  $0 \leq t_2\leq t_1$, 
we define $\s_{n}(t_1, t_2 )$ by 
\begin{align}
\begin{split}
\s_{n}(t_1, t_2 )  
 : \! & =
\E  \big[  \ft{\<1>}(n, t_1)  \,  \ft{\<1>}(-n, t_2) \big]
  =
   \int_0^{t_2} \frac{\sin ((t_1 - t')\jb{n} )}{\jb{n}^{1-\al}}
   \frac{\sin ((t_2 - t') \jb{n})}{\jb{n}^{1-\al}} d t' \\
&  =    \frac{\cos((t_1 - t_2)\jb{n}) }{2 \jb{n}^{2(1-\al)}} t_2
+  \frac{\sin ((t_1-t_2)  \jb{n})}{4 \jb{n}^{3-2\al} }
-  \frac{\sin ((t_1 +  t_2) \jb{n})}{4 \jb{n}^{3-2\al} }.
\end{split}
\label{sigma2}
\end{align}

\noi
Then, from 
\eqref{sigma1} and \eqref{sigma2}, we have
\begin{align}
\s_N(t) = \frac 1{4\pi^2} \sum_{|n| \le N} \s_n(t,t).
\label{sigma3}
\end{align}

\noi
Moreover, from Wick's theorem (Lemma \ref{LEM:Wick}), we have
\begin{align}
\E \Big[\big(  |\ft{\<1>}(n_1, t_1)|^2 -  \s_{n_1}(t_1, t_1)\big)\big( |\ft{\<1>}(n_2, t_2)|^2 -  \s_{n_2}(t_2, t_2)\big)\Big] 
 =  \ind_{n_1=\pm n_2} \cdot  \s_{n_1}^2(t_1,t_2).
\label{X3}
\end{align}

\noi
In the following, we fix $T>0$.

\subsection{Proof of Lemma \ref{LEM:stoconv}}
\label{SUBSEC:31}

(i) 
From \eqref{sigma2}, we have
\begin{align}
\E\big[|\ft{\<1>}_N(n, t)|^2\big] 
= \s_n (t,t)
\les_T \jb{n}^{-2 + 2\al}
\label{sconv3}
\end{align}

\noi
for any $n \in \Z^2$ and   $0 \le t \le T$, uniformly in $N \in \N$.
Also, by the mean value theorem and an interpolation argument as in \cite{GKO2}, 
we have 
\[
\E\big[ |\ft{\<1>}_N(n, t_1) - \ft{\<1>}_N(n, t_2)|^2 \big]
\les_T \jb{n}^{-2(1-\al)+\theta} |t_1-t_2|^\theta
\]
for any $\theta \in [0, 1]$, $n \in \Z^2$,  and $0 \le t_2 \le t_1 \le T$ with $t_1-t_2 \le 1$,
uniformly in $N \in \N$.
Hence, from Lemma \ref{LEM:reg}, 
we conclude that 
 $\<1>_N \in C([0, T]; W^{-\al - \eps, \infty}(\T^2))$
for any $\eps > 0$, 
 almost surely.
Moreover, a slight modification
of the argument  yields
that $\{ \<1>_N \}_{N \in \N}$ is almost surely
 a Cauchy sequence in $C([0, T]; W^{-\al - \eps, \infty}(\T^2))$,
  thus converging to some limit $\<1>$.
Since the required modification is exactly the same as in  \cite{GKO2}, 
we omit the details here.

In the remaining part of this section, 
we only establish 
the estimate \eqref{reg1} in Lemma~\ref{LEM:reg}
for each of $\<2>_N$, $\<20>_N$, and $\<21p>_N$, 
uniformly in $N \in \N$.
The time difference estimate \eqref{reg3}
and the convergence claim 
follow from a straightforward modification
as in \cite{GKO2}.

\smallskip

\noi
(ii)
Next, we study the Wick power $\<2>_N$.
In view of Lemma \ref{LEM:reg} and the comment above, it suffices to prove
\begin{align}
\E\big[|\ft{\<2>}_N(n, t)|^2\big]
\les_T \jb{n}^{-2+4\al}
\label{2var}
\end{align}
for $n \in \Z^2$ and $0 \le t \le T$,
uniformly in $N \in \N$.
From \eqref{so4b} and~\eqref{sigma3}, we have
\begin{align*}
\ft{\<2>}_N(n, t)
&  = \ft{\<1>_N^2}(n, t) - \ind_{n = 0} \cdot 2\pi \s_N(t)\\
 & = \frac{1}{2\pi} \sum_{\substack{n = n_1 + n_2\\ |n_1|, |n_2| \le N}} 
 \Big(\ft{\<1>}(n_1, t)\ft{\<1>}(n_2, t)
- \ind_{n = 0} \cdot \s_{n_1}(t,t) \Big) 
\end{align*}

\noi
and thus we have
\begin{align}
\begin{split}
\E\big[|\ft{\<2>}_N(n, t)|^2\big] 
& = \frac{1}{4\pi^2} \sum_{\substack{n = n_1 + n_2\\ |n_1|, |n_2| \le N}} 
\sum_{\substack{n = n_1' + n_2' \\ |n_1'|, |n_2'| \le N}}
\E\bigg[ \Big(\ft{\<1>}(n_1, t)\ft{\<1>}(n_2, t)
- \ind_{n = 0} \cdot \s_{n_1}(t,t) \Big) \\
& \hphantom{XXXXXXXXXXX}
\times
\cj{\Big(\ft{\<1>}(n_1', t)\ft{\<1>}(n_2', t)
- \ind_{n = 0} \cdot \s_{n_1'}(t,t) \Big)}\bigg].
\end{split}
\label{V1}
\end{align}

\noi
In order to have non-zero contribution in \eqref{V1}, 
we must have~$n_1 = n_1'$ and $n_2 = n_2'$ up to permutation.

By Wick's theorem (Lemma \ref{LEM:Wick}), we have 
\begin{equation} 
\E \big[ |\ft{\<1>}(n, t)|^4 \big] = 2 \s_{n}^2(t,t).
\label{1mon4}
\end{equation}

\noi
Then, 
for $n=0$, it follows from \eqref{V1}, \eqref{X3}, and \eqref{1mon4} that
\begin{align}
\begin{split}
\E\big[|\ft{\<2>}_N(0, t)|^2\big]
&\les
\sum_{\substack{k \in \Z^2}}
\E\bigg[ \Big(|\ft{\<1>}(k, t)|^2 - \s_k (t,t) \Big)^2 \bigg]
= \sum_{\substack{k \in \Z^2}} \bigg( \E \big[ |\ft{\<1>}(k, t)|^4 \big] - \s_{k}^2(t,t) \bigg) \\
&=\sum_{\substack{k \in \Z^2}} \s_{k}^2(t,t)
\les_T \sum_{k\in \Z^2} \frac{1}{\jb{k}^{4(1-\al)}}<\infty, 
\end{split}
\label{2var1}
\end{align}

\noi
provided that $\al < \frac 12$.
Similarly, for $n \neq 0$, we have
\begin{align}
\begin{split}
\E\big[|\ft{\<2>}_N(n, t)|^2\big] 
&=
\frac 1{2\pi^2}\sum_{\substack{n = n_1 + n_2 \\ n_1 \neq \pm n_2}}
\E\Big[ |\ft{\<1>}_N(n_1, t)|^2 |\ft{\<1>}_N(n_2, t)|^2 \Big]
+\frac 1{4\pi^2}\cdot \ind_{n \in 2 \Z^2\setminus\{0\}} 
\E\Big[ \big|\ft{\<1>}_N \big( \tfrac{n}{2}, t \big) \big|^4 \Big]\\
&= \frac 1{2\pi^2}\sum_{\substack{n = n_1 + n_2 \\ n_1+n_2 \neq 0 \\ |n_1|, |n_2| \le N}} \s_{n_1}(t,t) \s_{n_2}(t,t)\\
& \les_T \sum_{n = n_1 + n_2} \frac{1}{\jb{n_1}^{2(1-\al)} \jb{n_2}^{2(1-\al)}}
\les \jb{n}^{-2+4\al},
\end{split}
\label{2var2}
\end{align}

\noi
provided that $0 < \al < \frac 12$.
In the last inequality, 
we used  Lemma \ref{LEM:SUM}.
This proves \eqref{2var}.

\subsection{Proof of Proposition \ref{PROP:sto1}}
\label{SUBSEC:32}

Let $0<\al<\frac 12$ and let $s_\al$ be as in \eqref{exreg}.
In view of Lemma \ref{LEM:reg}, it suffices to show 
\begin{align}
\E \big[ |\ft{\<20>}_N(n, t)|^2\big] & \les_T \jb{n}^{-2-2s_\al+}
\label{S3a}
\end{align}

\noi
for any $n \in \Z^2$ and  $0 \leq t \leq T$,
uniformly in $N \in \N$.
Our argument follows closely to that in the proof of Proposition 1.6 in \cite{GKO2}
up to Case 2 below, where our argument diverges.
We, however, present details for readers' convenience.
See also Remark \ref{REM:exsmooth} below.

By the definition \eqref{stoc2}, we have
\begin{align}
\E \big[ |\ft{\<20>}_N(n, t)|^2\big]
& = 
 \int_0^t \frac{\sin((t - t_1) \jb{n})}{\jb{n}} 
 \int_0^t \frac{\sin((t - t_2) \jb{n})}{\jb{n}} 
\E\Big[ \ft{\<2>}_N(n, t_1) \cj{\ft{\<2>}_N(n, t_2)} \Big]  dt_2dt_1.
\label{S4}
\end{align}

\noi
Let us first consider the case $n = 0$.
 It follows from 
\eqref{S4} and \eqref{V1} that
\begin{align*}
\E \big[ |\ft{\<20>}_N(0, t)|^2\big]
& = \frac 1{4\pi^2}
 \int_0^t \sin (t - t_1) 
 \int_0^t \sin (t-t_2) \notag \\
& \hphantom{X}
\times
\sum_{\substack{k_1, k_2 \in \Z^2 \\ |k_1|, |k_2| \le N}}
\E\Big[\big( |\ft{\<1>}(k_1, t_1)|^2 -  \s_{k_1}(t_1, t_1)\big)\big( |\ft{\<1>}(k_2, t_2)|^2 -  \s_{k_2}(t_2, t_2)\big)\Big] 
 dt_2dt_1.
\end{align*}

\noi
By symmetry, \eqref{X3}, and \eqref{sigma2}, we obtain
\begin{align}
\E \big[ |\ft{\<20>}_N(0, t)|^2\big]
&\les
\int_0^t 
 \int_0^{t_1} 
\sum_{k \in \Z^2} \s_k^2(t_1,t_2)
 dt_2dt_1 
 \les_T \sum_{k \in \Z^2}\frac{1}{\jb{k}^{4(1-\al)}}
< \infty, 
 \notag
\end{align}
provided that $\al<\frac 12$.
This proves \eqref{S3a} when $n = 0$.

In the following, we consider the case  $n \ne 0$.
With \eqref{V1} and proceeding as in \eqref{2var2}, 
we have
\begin{align}
\E \big[ |\ft{\<20>}_N(n, t)|^2\big]
& = \frac 1{\pi^2}  \sum_{\substack{n = n_1 + n_2\\n_1 \ne \pm n_2 \\ |n_1|, |n_2| \le N}} \int_0^{t} \frac{\sin((t - t_1) \jb{n})}{\jb{n}} 
\notag \\
& \hphantom{XXXXXXX}
 \times 
 \int_0^{t_1} \frac{\sin((t - t_2) \jb{n})}{\jb{n}} 
\s_{n_1}(t_1, t_2) \s_{n_2}(t_1, t_2)  dt_2dt_1\notag\\
& \hphantom{X}
+ \frac 1{2\pi^2}\cdot \ind_{n \in 2 \Z^2\setminus\{0\}} 
 \int_0^t \frac{\sin((t - t_1) \jb{n})}{\jb{n}} 
 \int_0^{t_1} \frac{\sin((t - t_2) \jb{n})}{\jb{n}} 
\notag\\
& \hphantom{XXXXXXX}
\times \E\Big[ \ft{\<1>}_N \big(\tfrac{n}{2}, t_1\big)^2 \, \cj{\ft{\<1>}_N \big(\tfrac{n}{2}, t_2\big)^2 } \Big]  dt_2dt_1
\notag\\
& =: 
\1 (n, t)+ \II(n, t), 
\label{S5}
\end{align}

\noi
where $\II(n, t)$ denotes the contribution from $n_1 = n_2 = n_1' = n_2' = \frac n2$.

We first estimate the second term $\II(n, t)$ in \eqref{S5}.
By Wick's theorem (Lemma \ref{LEM:Wick}) with~\eqref{sigma2}, 
we have
\begin{align*}
 \bigg|\E\Big[ \ft{\<1>}\big(\tfrac{n}{2}, t_1\big)^2 & \, \cj{\ft{\<1>}\big(\tfrac{n}{2}, t_2\big)^2 } \Big] \bigg|
\les_T \jb{n}^{-4(1-\al)}
\end{align*}

\noi
under $0 \le t_2 \le t_1 \leq t \leq T$.
Hence, from  \eqref{S5}
with \eqref{exreg},  we conclude that 
\begin{align}
|\II(n, t)|\les_T \jb{n}^{-6+4\al}
\le \jb{n}^{-2-2s_\al},
\label{X3a}
\end{align}

\noi
verifying \eqref{S3a}.

Next, we estimate $\1(n, t)$ in \eqref{S5}.
As in \cite{GKO2}, we have
\begin{align}
\begin{split}
\1 (n, t) & = 
 - \frac 1{4\pi^2} \sum_{k_1, k_2 \in \{1, 2\}} \sum_{\eps_1, \eps_2 \in \{-1, 1\}}
\frac{\eps_1\eps_2e^{i (\eps_1+\eps_2) t\jb{n}}}{\jb{n}^2}
\sum_{\substack{n = n_1 + n_2\\n_1 \ne \pm n_2 \\ |n_1|, |n_2| \le N}}
\int_0^{t} 
e^{-i \eps_1t_1\jb{n}} \\
& \hphantom{XX}
\times  \int_0^{t_1} e^{-i \eps_2t_2\jb{n}}
\prod_{j = 1}^2\s_{n_j}^{(k_j)}(t_1, t_2)
\,   dt_2dt_1 =: \sum_{k_1, k_2 \in \{1, 2\}} \1^{(k_1,k_2)} (n, t),
\end{split}
\label{S6a}
\end{align}

\noi
where 
$\s_{n}^{(1)}(t_1, t_2)$
and $\s_{n}^{(2)}(t_1, t_2)$ are defined by 
\begin{align}
\s_{n}^{(1)}(t_1, t_2) &  :=  
 \frac{\cos((t_1 - t_2)\jb{n}) }{2 \jb{n}^{2(1-\al)}} t_2, 
 \label{S61}
\\ 
\s_{n}^{(2)}(t_1, t_2)  & := 
  \frac{\sin ((t_1-t_2)  \jb{n})}{4\jb{n}^{3-2\al} }
-  \frac{\sin ((t_1 +  t_2) \jb{n})}{4\jb{n}^{3-2\al} }
\label{S62}
\end{align}

\noi
such that 
$ \s_{n}(t_1, t_2) =  \s_{n}^{(1)}(t_1, t_2) + \s_{n}^{(2)}(t_1, t_2)$.
%

By Lemma \ref{LEM:SUM}, 
the contribution to $\1(n, t)$ in~\eqref{S6a}
from  $(k_1, k_2) \ne (1, 1)$ can be estimated by 
\[
\frac{1}{\jb{n}^2}\sum_{n = n_1 + n_2} \frac{1}{\jb{n_1}^{2(1-\al)} \jb{n_2}^{3-2\al}}
\les \jb{n}^{ - 2 -2(1-\al)+}
\]

\noi
for $0<\al<\frac 12$, verifying \eqref{S3a}.
Hence, we focus on estimating 
 $\1(n, t)$ coming from 
 $(k_1, k_2) = (1, 1)$:
\begin{align}
\begin{split}
\1^{(1,1)} (n, t)
& : = -\frac 1{64\pi^2} \sum_{\eps_1, \eps_2, \eps_3, \eps_4 \in \{-1, 1\}}
\sum_{\substack{n = n_1 + n_2\\n_1 \ne \pm n_2 \\ |n_1|, |n_2| \le N}}
\frac{\eps_1\eps_2e^{i (\eps_1+\eps_2) t\jb{n}}}{\jb{n}^2\jb{n_1}^{2(1-\al)} \jb{n_2}^{2(1-\al)}}\\
& \hphantom{XXXX}\times
\int_0^{t} 
e^{-i t_1 \kk_1( \bar n)}
 \int_0^{t_1} 
t_2^2   e^{-i t_2\kk_2(\bar n)}
 \, dt_2dt_1, 
\end{split}
\label{X4}
\end{align}

\noi
where $\kk_1( \bar n)$ and $\kk_2( \bar n)$ are defined by 
\begin{align}
\kk_1(\bar n) & =
 \eps_1 \jb{n} -\eps_3 \jb{n_1} - \eps_4\jb{n_2} 
 \qquad \text{and}\qquad 
\kk_2(\bar n)  = 
 \eps_2 \jb{n} +\eps_3 \jb{n_1} + \eps_4\jb{n_2}.
 \label{kk1}
\end{align}

\noi
When $|n| \les 1$, 
it follows from Lemma \ref{LEM:SUM}
that $|\1^{(1,1)} (n, t)|\les_T 1$  for $0<\al<\frac 12$.
Hence, we assume $|n|\gg 1$.
As in \cite{GKO2}, 
we must carefully estimate $\1^{(1,1)} (n, t)$, 
depending on   $\bar \eps = (\eps_1, \eps_2, \eps_3, \eps_4)$, 
by  exploiting either (i) the dispersion (= oscillation) 
or (ii) smallness of the measure of the relevant frequency set.

Fix our choice of $\bar \eps = (\eps_1, \eps_2, \eps_3, \eps_4)$ 
and denote by $\1^{(1,1)}_{\bar \eps}(n,t)$ the associated contribution to $\1^{(1,1)}(n,t)$.
By switching the order of integration and first integrating in $t_1$, we have 
\begin{align}
\begin{split}
\bigg|\int_0^{t} 
& e^{-i t_1 \kk_1( \bar n)}
 \int_0^{t_1} 
t_2^2   e^{-i t_2\kk_2(\bar n)}
 \, dt_2dt_1\bigg| \\
& = \bigg| \int_0^{t} 
t_2^2   e^{-i t_2\kk_2(\bar n)}
\frac{e^{-i t \kk_1( \bar n)} - e^{-i t_2 \kk_1( \bar n)}}{-i \kk_1(\bar n)}
 \, dt_2\bigg| \les_T (1+ |\kk_1(\bar n)|)^{-1}.
\end{split}
\label{X5}
\end{align}

\noi
From \eqref{X4} and \eqref{X5},  we have 
\begin{align}
|\1^{(1,1)}_{\bar \eps} (n, t)|
& \les_T   
\sum_{n = n_1 + n_2} 
\frac{1}{\jb{n}^2\jb{n_1}^{2(1-\al)} \jb{n_2}^{2(1-\al)}(1+ |\kk_1(\bar n)|)}.
\label{SS1}
\end{align}

\noi
In the following, 
we assume $|n_1|\geq |n_2|$
without loss of generality.
Under $n = n_1 + n_2$
 we then have 
\begin{align}
\jb{n_1} \sim \jb{n}+ \jb{n_2}.
\label{SS1a}
\end{align}

When $(\eps_1, \eps_3, \eps_4) = (\pm 1, \mp1, \mp1)$ or
$(\pm1, \mp1, \pm1)$, 
we have
$|\kk_1(\bar n)| \geq \jb{n}$.
Then, the desired bound \eqref{S3a} follows from \eqref{SS1}
and Lemma \ref{LEM:SUM}.

Next, 
we  consider the case $(\eps_1, \eps_3, \eps_4) = (\pm 1, \pm1, \mp1)$.
In this case, we have
$|\kk_1(\bar n)|=  \jb{n} + \jb{n_2} -\jb{n_1}$.
By \eqref{SS1},
the contribution to $\1^{(1,1)}_{\bar \eps} (n, t)$ from $n_2=0$ is estimated by
\begin{align*}
\frac 1{\jb{n}^2 \jb{n}^{2(1-\al)}}
= \jb{n}^{-4+2\al}
\le \jb{n}^{-2-2s_\al},
\end{align*}

\noi
satisfying \eqref{S3a}.
Hence, 
we assume  $n_2 \neq 0$
in the following.
By viewing  $n_1$ as a vector based at $n_2$, 
we see that 
three vectors $n$, $n_1$, and $n_2$ form a triangle.
Hence, it follows from 
the law of cosines that
\begin{align}
|n|^2 + |n_2|^2 - |n_1|^2 = 2 |n| |n_2| \cos \big( \angle(n, n_2)\big),  
\label{SS1b}
\end{align}

\noi
where
$\angle(n, n_2)$ denotes the angle between $n$ and $n_2$.
Then, from  \eqref{SS1a} and \eqref{SS1b}, 
we have
\begin{align}
\begin{split}
|\kk_1(\bar n)| 
& =  \frac{(\jb{n} + \jb{n_2})^2 -\jb{n_1}^2}
{\jb{n} + \jb{n_2} +\jb{n_1}}
  =  \frac{2\jb{n}\jb{n_2} + |n|^2 + |n_2|^2 - |n_1|^2 + 1}
{\jb{n} + \jb{n_2} +\jb{n_1}} \\
& 
\ges
 \frac{1+ |n| |n_2| (1 - \cos \theta)}
{\jb{n_1}} , 
\end{split}
\label{SS2}
\end{align}

\noi
where $\ta = \angle(n_2, -n) \in [0, \pi]$ is the angle between $n_2$ and $-n$.

Using \eqref{SS1}, 
the contribution to $\1^{(1,1)}_{\bar \eps} (n, t)$ from $n_2\ne0$ is estimated by
\begin{align}
&\les 
\sum_{\substack{n = n_1 + n_2 \\ 1\le |n_2| \le |n_1| \\ 1 - \cos \ta \ges 1}} 
\frac{1}{\jb{n}^2\jb{n_1}^{2(1-\al)} \jb{n_2}^{2(1-\al)}(1+ |\kk_1(\bar n)|)} \notag\\
&\quad+\sum_{\substack{n = n_1 + n_2 \\ 1\le |n_2| \le |n_1| \\1 - \cos \ta \ll 1}} 
\frac{1}{\jb{n}^2\jb{n_1}^{2(1-\al)} \jb{n_2}^{2(1-\al)}(1+ |\kk_1(\bar n)|)}
\notag\\
&=: \1^{(1,1)}_{\bar \eps, 1} (n, t)  + \1^{(1,1)}_{\bar \eps, 2} (n, t).
\label{SS1_0}
\end{align}

\noi
In the following, we separately estimate 
$\1^{(1,1)}_{\bar \eps, 1} (n, t)$ and $\1^{(1,1)}_{\bar \eps, 2} (n, t)$.

\smallskip

\noi
$\bullet$ {\bf Case 1:}
 $1 - \cos \ta \ges 1$.
\quad In this case, from \eqref{SS1_0}, \eqref{SS2},  and  Lemma \ref{LEM:SUM}, 
we have 
\begin{align*}
\1^{(1,1)}_{\bar \eps, 1} (n, t)
& \les_T   
\sum_{n = n_1 + n_2} 
\frac{1}{\jb{n}^{3}\jb{n_1}^{1-2\al} \jb{n_2}^{3-2\al}} \notag\\
& \les   
\jb{n}^{-4+2\al+}
\les \jb{n}^{-2-2s_\al+}, 
\end{align*}

\noi
provided that  $0<\al < \frac 12$.
This verifies  \eqref{S3a}.

\smallskip

\noi
$\bullet$ {\bf Case 2:}
 $1 - \cos \ta \ll1$.
\quad In this case, we have $0\leq \ta \ll 1$, namely, 
$n$ and $n_2$ point in almost opposite directions.
In particular,  we have $1 - \cos \ta \sim \ta^2 \ll 1$.
By dyadically decomposing $n_2$ into $|n_2| \sim N_2$ for dyadic  $N_2 \geq 1$,
we have
\begin{align}
\1^{(1,1)}_{\bar \eps, 2} (n, t)
& \les_T   
\sum_{\substack{N_2 \geq 1\\\text{dyadic}} }
\sum_{\substack{n = n_1 + n_2 \\ \ta^2 \ll 1 \\ |n_2| \sim N_2}} 
\frac{1}{\jb{n}^2\jb{n_1}^{2(1-\al)} \jb{n_2}^{2(1-\al)}(1+ |\kk_1(\bar n)|)}.
\label{SS1d}
\end{align}

\noi
We see that 
for fixed $n \in \Z^2$, 
 the range of possible $n_2$ with $|n_2| \sim N_2$
is constrained to an axially symmetric trapezoid $\RR$
whose height is $\sim N_2 \cos \ta \sim N_2$
and the top and bottom widths $\sim N_2 \sin \ta \sim N_2 \ta$
with the axis of symmetry  given by $-n$.
See Figure \ref{FIG:2}.
Hence, we have
\begin{align}
\sum_{\substack{n_2 \in \Z^2 \\ \ta^2 \ll 1 \\ |n_2| \sim N_2}} 1
\les
1 + \text{vol}(\RR) \sim 1 + N_2^2 \ta.
\label{SS1e}
\end{align}

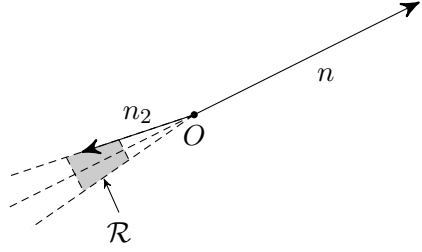
\begin{figure}[h]
\begin{tikzpicture}

\path[fill=gray!40]
(-1.7, -0.566) -- (-1.47, -1.02)
--  (-0.866, -0.601) -- (-1, -0.333) -- cycle;

\draw[ decoration={markings,mark=at position 1 with {\arrow[scale=2]{>}}},
    postaction={decorate},
    shorten >=0.4pt] (0,0) node[below]{$O$}-- (3, 1.5) 
node[pos=0.5, below right] {$n$}; 

\draw (0,0)   node[ddot]{}; 
\draw[densely dashed](0,0) --(-2.5, -1.25) ;

\draw[ decoration={markings,mark=at position 1 with {\arrow[scale=2]{>}}},
    postaction={decorate},
    shorten >=0.4pt] (0,0) node[dot]{}  --(-1.5, -0.5) 
node[pos=.5,  above] {$n_2$};

\draw[densely dashed](0,0) --(-2.5, -0.83) ;

\draw[densely dashed](0, 0 ) -- (2.5 - 14/3, 0.83 -7/3); 

\draw[densely dashed](-1, -0.333) -- (-0.866, -0.601); 

\draw[densely dashed](-1.7, -0.566) -- (-1.47, -1.02);

\draw[->] (-1, -1.3) -- (-1.2, -0.85)
node[pos=0,  below ] {$\mathcal{R}$};
 ;

\end{tikzpicture}

\caption{A typical configuration in Case 2}
\label{FIG:2}
\end{figure}

We now use the following seemingly crude bound: 
\begin{align}
1+|\kk_1(\bar n)| \ges |\kk_1(\bar n)|^{\frac 12}.
\label{X6}
\end{align}

\noi
Then, from 
\eqref{SS2} and 
\eqref{X6}
with $|n_j| \sim N_j$, $j = 1, 2$, we have
\begin{align}
\Theta_1 : = \frac{1 + N_2^2 \ta}{1+|\kk_1(\bar n)|}
\les N_1^\frac{1}{2}
\frac{1 + N_2^2 \ta}
{1 + \jb{n}^\frac 12  N_2^\frac 12 \ta}
\les \frac{N_1^\frac{1}{2} N_2^\frac{3}{2}}{\jb{n}^\frac{1}{2}}, 
\label{X6a}
\end{align}

\noi
provided that $N_2^3 \ges \jb{n}$.
The bound \eqref{X6a} follows
from separately considering the cases:
$\jb{n} \les N_2$
and $N_2 \ll \jb{n} \les N_2^3$, using the condition $N_2^3 \ges \jb{n}$.
When $N_2^3 \ll  \jb{n}$, 
the bound~\eqref{X6a} also holds
under an additional assumption $N_2^2 \ta \ges 1$
(which also implies $\jb{n}^\frac 12  N_2^\frac 12 \ta \gg 1$).
When $N_2^3 \ll  \jb{n}$
and 
 $N_2^2 \ta \ll 1$, 
 we simply use the lower bound:
 $1+|\kk_1(\bar n)| \ges 1$.

Hence,  from \eqref{SS1d}, \eqref{SS1e},   \eqref{X6}, and \eqref{X6a}
with $N_1 \sim |n_1| \sim \max(|n|, |n_2|)$,
we have 
\begin{align}
\begin{split}
\1^{(1,1)}_{\bar \eps, 2} (n, t)
& \les_T   
\sum_{\substack{N_2 \geq 1\\\text{dyadic}} }
\frac{1}{\jb{n}^{2} N_1^{2(1-\al)} N_2^{2(1-\al)}} 
\Theta_1
\cdot \big(
 \ind_{ N_2^3 \ges  \jb{n}}
+ 
 \ind_{ N_2^3 \ll \jb{n}}\cdot 
 \ind_{N_2^2 \ta \ges 1}
 \big) 
\\
& \quad + 
\sum_{\substack{ 1\ll N_2 \ll  \jb{n}^\frac{1}{3} \\ \text{dyadic}} }
\frac{1}{\jb{n}^{2} N_1^{2(1-\al)} N_2^{2(1-\al)}} 
\cdot \ind_{ N_2^2 \ta \ll 1}\\
& \les
\sum_{\substack{N_2 \geq 1\\\text{dyadic}} }
\frac{1}{\jb{n}^{\frac 52} \max(\jb{n}, N_2)^{\frac 32 -2\al} N_2^{\frac 12 -2\al} } \\
& =
\sum_{\substack{1 \le N_2 < \jb{n} \\\text{dyadic}} }
\frac{1}{\jb{n}^{4-2\al} N_2^{\frac 12 -2\al} } 
+
\sum_{\substack{N_2 \geq \jb{n} \\\text{dyadic}} }
\frac{1}{\jb{n}^{\frac 52} N_2^{2 -4\al} }.
\end{split}
\label{SS4}
\end{align}

\noi
The first term on the right-hand side of \eqref{SS4} is bounded by
\begin{align*}
\begin{split}
\sum_{\substack{1 \le N_2 < \jb{n} \\\text{dyadic}} }
\frac{1}{\jb{n}^{4-2\al} N_2^{\frac 12 -2\al} } 
&\les
\begin{cases} \jb{n}^{-4+2\al+}, & \text{if } \al  \le \frac 14, \\  
\jb{n}^{-\frac 92 +4\al}, & \text{if } \al>\frac 14
\end{cases}
\\
&\les \jb{n}^{-2-2s_\al+}, 
\end{split}
\end{align*}

\noi
where $s_\al$ is as in \eqref{exreg}.
As for the second term on the right-hand side of \eqref{SS4}, we have
\begin{align*}
\sum_{\substack{N_2 \geq \jb{n} \\\text{dyadic}} }
\frac{1}{\jb{n}^{\frac 52} N_2^{2 -4\al} }
\les
\jb{n}^{-\frac 92 +4\al}
\les \jb{n}^{-2-2s_\al}
\end{align*}
for $\al < \frac 12$.

Lastly, when $(\eps_1, \eps_3, \eps_4) = (\pm 1, \pm1, \pm1)$, 
we can essentially reduce the analysis to Cases 1 and 2 above.
See  Case 3 in the proof of Proposition 1.6 in \cite{GKO2}.
This completes the proof of Proposition \ref{PROP:sto1}.

\begin{remark} \label{REM:exsmooth} \rm

As mentioned in Section \ref{SEC:1},  the extra smoothing 
for $\<20>$
on $\T^2$
is at most $\frac 14$, while $\frac 12$-extra smoothing on $\T^3$ 
was shown in \cite{GKO2}.
This difference comes from Case 2
in the proof of Proposition \ref{PROP:sto1} above, 
where we applied~\eqref{X6}.
We point out that the bound \eqref{X6} is used to cancel the powers
of $\ta $ in~\eqref{SS4}.
Furthermore, 
 we can show that the estimate shown above is essentially sharp.
More precisely,
we have
the following the lower bound of $\1^{(1,1)} (n, t)$ in \eqref{X4}:
\begin{align}
\1^{(1,1)} (n,t)
\ges t^4 \jb{n}^{-2-2s_\al}
\label{D-3}
\end{align}
for $0<\al<\frac 12$ and $\jb{n}^{- \frac 12} \ll t \ll 1$,
where $s_\al$ is as in \eqref{exreg}.

For simplicity, we drop the truncation $|n_1|, |n_2| \le N$ in \eqref{X4}
with the understanding that a rigorous computation is to be done
with the truncation $|n_1|, |n_2| \le N$ in \eqref{X4}
and then by taking $N \to \infty$.
Namely,
we consider that 
$\1^{(1,1)} (n, t)$ in \eqref{X4} is written as follows:
\begin{align}
\begin{split}
\1^{(1,1)} (n, t)
&= \frac 1{4\pi^2} 
\sum_{\substack{n = n_1 + n_2\\n_1 \ne \pm n_2}} \int_0^{t} 
\int_0^{t_1} \frac{\sin((t - t_1) \jb{n})}{\jb{n}} 
\frac{\sin((t - t_2) \jb{n})}{\jb{n}} \\
& \hphantom{XXXXXXX}
 \times
\frac{\cos ((t_1-t_2) \jb{n_1}) }{\jb{n_1}^{2(1-\al)}}
\frac{\cos ((t_1-t_2) \jb{n_2}) }{\jb{n_2}^{2(1-\al)}} t_2^2 dt_2dt_1.
\end{split}
\label{S5a}
\end{align}
A direct calculation shows that
\begin{align}
\sin & ((t - t_1) \jb{n}) \sin((t - t_2) \jb{n}) \cos ((t_1-t_2) \jb{n_1})  \cos ((t_1-t_2) \jb{n_2}) \notag\\
&=\frac 14
\Big(-\cos ((2t-t_1-t_2) \jb{n}) + \cos ((t_1-t_2) \jb{n})\Big)
\notag\\
& \hphantom{XXX}
 \times
\Big(\cos ((t_1-t_2) ( \jb{n_1}+ \jb{n_2})) + \cos ((t_1-t_2)(\jb{n_1}-\jb{n_2}))\Big) \notag\\
&=
\frac 14 \Big(
-\cos ((2t-t_1-t_2) \jb{n}) \cos ((t_1-t_2) ( \jb{n_1}+ \jb{n_2}))
\notag\\
& \hphantom{XXX}
-\cos ((2t-t_1-t_2) \jb{n}) \cos ((t_1-t_2) ( \jb{n_1}- \jb{n_2}))
\notag\\
& \hphantom{XXX}
+ \cos ((t_1-t_2) \jb{n}) \cos ((t_1-t_2) ( \jb{n_1}+ \jb{n_2})) 
\notag\\
& \hphantom{XXX}
+ \cos ((t_1-t_2) \jb{n}) \cos ((t_1-t_2) ( \jb{n_1}- \jb{n_2})) 
\Big)
\notag\\
&=: \sum_{j=1}^4 A_{n,n_1,n_2}^{(j)}(t,t_1,t_2).
\label{D-2}
\end{align}

\noi
We denote the contribution to \eqref{S5a} from $A_{n,n_1,n_2}^{(j)}(t,t_1,t_2)$ by $\1_j (n, t)$:
\begin{align}
\1_j(n,t) :=
\frac 1{4\pi^2} 
\sum_{\substack{n = n_1 + n_2\\n_1 \ne \pm n_2}}
\frac 1{\jb{n}^2 \jb{n_1}^{2(1-\al)} \jb{n_2}^{2(1-\al)}}
\int_0^{t}
\int_0^{t_1}
A_{n,n_1,n_2}^{(j)}(t,t_1,t_2)
t_2^2 dt_2dt_1.
\label{D-1}
\end{align}

\noi
As we see below, the main contribution comes
from $\1_4(n, t)$.

First, we show that $\1_1$ and $\1_2$ satisfy
\begin{align}
|\1_1(n,t)| + |\1_2(n,t)|
\les
t^3 \jb{n}^{-\frac 52-2s_\al}
\label{D0}
\end{align}
for $0<\al<\frac 12$, $0 \le t \le 1$, and $n \in \Z^2$.
In the following, we only estimate  $\1_1$, 
since $\1_2$ can be handled in an analogous manner.
By applying a change of  the variable $\tau_1 = \frac{t_1-t_2}2$, $\tau_2 = \frac{t_1+t_2}2$
to \eqref{D-1}, 
we have
\begin{align}
\begin{split}
|\1_1 (n,t)|
&\les
\sum_{\substack{n = n_1 + n_2\\n_1 \ne \pm n_2}}
\frac{1}{\jb{n}^2 \jb{n_1}^{2(1-\al)} \jb{n_2}^{2(1-\al)}}
\\
& \quad
\times
\bigg| \int_0^{\frac t2} \int_{\tau_1}^{t-\tau_1} \cos (2(t-\tau_2)\jb{n}) \cos (2 \tau_1 (\jb{n_1}+\jb{n_2})) (\tau_1-\tau_2)^2 d\tau_2 d\tau_1 \bigg|.
\end{split}
\label{D1}
\end{align}

\noi
By integration by parts, we can bound the inner integral by 
\begin{align}
\begin{split}
 \bigg|  \int_{\tau_1}^{t-\tau_1}  & \cos (2(t-\tau_2)\jb{n}) (\tau_1-\tau_2)^2 d\tau_2 \bigg| \\
&=\bigg| -\frac{\sin (2\tau_1\jb{n})}{2 \jb{n}} (2\tau_1-t)^2 \\
&\hphantom{XXX}
- \frac 1{\jb{n}} \int_{\tau_1}^{t-\tau_1} \sin (2(t-\tau_2)\jb{n}) (\tau_1-\tau_2) d\tau_2 \bigg| \\
&\les \frac{t^2}{\jb{n}}
\end{split}
\label{D2}
\end{align}
for $0 \le \tau_1 \le \frac t2 \le \frac 12$.
It follows from \eqref{D1}, \eqref{D2}, and Lemma \ref{LEM:SUM} with $0<\al<\frac 12$ that
\begin{align*}
|\1_1(n,t)|
&\les
\sum_{\substack{n = n_1 + n_2\\n_1 \ne \pm n_2}}
\frac {t^3}{\jb{n}^3 \jb{n_1}^{2(1-\al)} \jb{n_2}^{2(1-\al)}} \\
&\les
t^3 \jb{n}^{-5+4\al}
\les
t^3 \jb{n}^{-\frac 52-2s_\al}
\end{align*}
for $0 \le t \le 1$ and $n \in \Z^2$.
This proves  \eqref{D0}.

Next, we prove the lower bound \eqref{D-3} on 
from $\1^{(1, 1)}(n, t)$.
From \eqref{D-2}
and the product formula, we have 
\begin{align*}
A_{n,n_1,n_2}^{(3)}(t,t_1,t_2)
&=
\frac 18 \Big(
\cos ((t_1-t_2) (\jb{n}+\jb{n_1}+\jb{n_2}))
\\
& \hphantom{XXXXX}
+\cos ((t_1-t_2) (\jb{n}-\jb{n_1}-\jb{n_2})) \Big), \\
A_{n,n_1,n_2}^{(4)}(t,t_1,t_2)
&=
\frac 18
\Big( \cos ((t_1-t_2) (\jb{n}+\jb{n_1}-\jb{n_2}))
\\
& \hphantom{XXXXX}
+\cos ((t_1-t_2) (\jb{n}-\jb{n_1}+\jb{n_2})) \Big).
\end{align*}

\noi
Moreover, we have
\begin{align}
\begin{split}
\int _0^t \int_0^{t_1} \cos ((t_1-t_2) a) t_2^2 dt_2 dt_1
&= \int_0^t \frac{2}{a^2}\Big( t_1 - \frac{\sin (t_1 a)}a \Big) dt_1 \\
&= \frac{t^2}{a^2} + 2 \frac{\cos (ta) - 1}{a^4} \\
&= \frac{2}{a^4} \Big( \cos (ta) -1 + \frac{t^2 a^2}2 \Big)
\ge 0
\label{D1a}
\end{split}
\end{align}
for any $a \in \R \setminus \{0\}$ and $t \ge 0$.
When $a = 0$, then the left-hand side of \eqref{D1a}
is obviously non-negative.
Hence, 
from \eqref{D-1} and \eqref{D1a}, we 
see that  $\1_3 (n,t), \1_4(n,t) \ge 0$.
Hence, 
from  \eqref{D-1} and \eqref{D0},
we obtain
\begin{align}
\begin{split}
\1^{(1,1)} (n,t)
&\ge
\1_4 (n,t) - C t^3 \jb{n}^{-\frac 52-2s_\al} \\
&\ge
J (n,t)
- C t^3 \jb{n}^{-\frac 52-2s_\al}
\end{split}
\label{D3}
\end{align}

\noi
for $0<t \ll 1$, 
where $J (n,t)$ is defined by 
\begin{align}
\begin{split}
J (n,t)
&:=
\frac 1{16\pi^2} 
\sum_{\substack{n = n_1 + n_2\\n_1 \ne \pm n_2}}
\ind_{n_2 \ne 0}\cdot 
\frac{1}{\jb{n}^2 \jb{n_1}^{2(1-\al)} \jb{n_2}^{2(1-\al)}}
\frac{1}{(\jb{n}-\jb{n_1}+\jb{n_2})^4}
\\
& \hphantom{XXXXX}
\times
\bigg(
\cos (t (\jb{n}-\jb{n_1}+\jb{n_2}))
-1 + \frac{ t^2 (\jb{n}-\jb{n_1}+\jb{n_2})^2}{2}
\bigg)\\
& \hphantom{XX} 
+ \frac 1{768\pi^2} 
\frac{t^4}{\jb{n}^{4-2\al} }.
\end{split}
\label{D3a}
\end{align}

\noi
Once we have
\begin{align}
J (n,t)
\ges
t^4 \jb{n}^{-2-2s_\al}
\label{D4}
\end{align}
for $0<t \ll 1$ and $n \in \Z^2$,
\eqref{D-3} follows from \eqref{D3} and \eqref{D4}:
\[
\1^{(1,1)} (n,t)
\ges t^3 \big( t - \jb{n}^{-\frac 12} \big) \jb{n}^{-2-2s_\al}
\ges t^4 \jb{n}^{-2-2s_\al}
\]

\noi
for $\jb{n}^{-\frac 12} \ll t \ll 1$.

Hence, it remains to 
prove \eqref{D4}.
First, consider the case
$0<\al \le \frac 14$.
In this case, from the second term on the right-hand side
of \eqref{D3a}, we have 
\begin{align}
J(n,t)
\ges t^4 \jb{n}^{-4+2\al}
= t^4 \jb{n}^{-2 -2(1-\al)}
\label{ex2}
\end{align}
for $0<t \ll 1$.
In view of \eqref{exreg}, 
this proves 
 \eqref{D4}  in this case.

Next, we consider the case $\frac 14 < \al < \frac 12$.
Given  a dyadic number $M$, we set $n=(M,0) \in \Z^2$ and
\begin{align*}
K_M = \Big\{ (a,b) \in \Z^2 : 2M \le a \le 4M, \, |b| \le M^{\frac 12} \Big\}.
\end{align*}

\noi
With  $n_1 =(a,b) \in K_M$, we have
\begin{align}
\begin{aligned}
\jb{n} - M &= \frac{1}{\jb{n}+M} \les M^{-1}, \\
\jb{n_1} - a &= \frac{1+b^2}{\jb{n_1}+a} \les 1, \\
\jb{n-n_1} - (a-M) &= \frac{1+b^2}{\jb{n-n_1}+a-M} \les 1.
\end{aligned}
\label{X7}
\end{align}

\noi
Then, it follows from \eqref{X7} that
\begin{align}
\jb{n}-\jb{n_1}+\jb{n-n_1}
= M - a + (a-M) + O(1)
\les 1.
\label{D5}
\end{align}

\noi
Hence,
from \eqref{D5}
and 
 the Taylor remainder theorem, 
we obtain
\begin{align}
\begin{split}
J(n,t)
&\ges
\sum_{n_1 \in K_M} 
\frac{1}{\jb{n}^2\jb{n_1}^{2(1-\al)} \jb{n-n_1}^{2(1-\al)}}
\frac 1{(\jb{n}-\jb{n_1}+\jb{n-n_1})^4}
\\
& \hphantom{XXX}
\times
\Big(
\cos (t (\jb{n}-\jb{n_1}+\jb{n-n_1}))
-1 + \frac{ t^2 (\jb{n}-\jb{n_1}+\jb{n-n_1})^2}{2}
\Big) \\
&\sim t^4 M^{-\frac 92 +4\al}
\sim t^4 \jb{n}^{-2-2(\frac 54- 2\al)}
\end{split}
\label{ex3}
\end{align}
for $0<t \ll 1$.
In view of \eqref{exreg}, 
this proves 
 \eqref{D4} when $\al > \frac 14$.

We also point out that 
 the calculation above can easily be extended 
 to the higher dimensional case.
More precisely, 
the right-hand side of  \eqref{ex2} is unchanged on $\T^d$
since we did not perform any summation.
By setting
\[
K_M = \Big\{ (a,b) \in \Z \times \Z^{d-1} : 2M \le a \le 4M, \, |b| \le M^{\frac 12} \Big\}
\]

\noi
and repeating the same computation on $\T^d$, 
the power on the right-hand side of \eqref{ex3}
becomes $-\frac{11}{2}+\frac d2 +4\al$.

By writing
\[ -4+2\al = -d -2 \big( 2-\tfrac d2 -\al \big)
\qquad  \text{and}\qquad 
 -\tfrac{11}{2}+\tfrac d2 +4\al = -d -2 \big( \tfrac{11}4 -\tfrac{3d}{4} - 2\al \big),\]
 
 \noi
 this computation indicates that 
 the regularity of $\<20>$ on $\T^d$ is at best
\[
\min \bigg( 2-\frac d2 -\al, \frac{11}4-\frac{3d}{4} - 2\al \bigg)
= \begin{cases}  2-\frac d2 -\al, & \text{if } 0 < \al \le \frac{3-d}{4}, \\
\frac{11}4-\frac{3d}{4} - 2\al, & \text{if } \al> \frac{3-d}{4}.
\end{cases}
\]

\noi
When $d = 3$ and $\al = 0$, 
this agrees with the $\frac 12$-smoothing shown in \cite{GKO2}.

\end{remark}

\subsection{Proof of Proposition \ref{PROP:sto2}}
\label{SUBSEC:33}

In this subsection, 
we present a proof of Proposition~\ref{PROP:sto2}
on  the resonant product $\<21p>_N =\<20>_N\pe \<1>_N$.
As in the previous subsection, 
we follow the argument in~\cite{GKO2}
but, as we see below,  our argument turns out to be simpler than the proof of Proposition~1.8 in~\cite{GKO2}.

From \eqref{para1} and  \eqref{stoc2}, we have
\begin{align}
\ft{\<21p>}_N (n, t)
& = \frac{1}{4\pi^2}
\sum_{\substack{n = n_1 + n_2+n_3\\|n_1 + n_2| \sim |n_3|\\ n_1 + n_2 \ne 0}} 
\int_0^t  \frac{\sin ((t-t') \jb{ n_1 + n_2 })}{\jb{ n_1 + n_2 }} 
\ft{\<1>}_N(n_1, t')\ft{\<1>}_N(n_2, t') dt' \cdot  \ft{\<1>}_N(n_3, t) \notag\\
& \hphantom{X} 
+ \frac{1}{4\pi^2}
\sum_{\substack{n_1\in \Z^2 \\ |n_1| \le N}} \ind_{|n|\sim 1} \int_0^t \sin (t - t')\cdot \big(  |\ft{\<1>}(n_1, t')|^2  -  \s_{n_1}(t')\big) 
dt' \cdot  \ft{\<1>}_N(n, t) \notag\\
& =: \ft \RR_1(n, t) + \ft \RR_2(n, t),
\label{Y1}
 \end{align}

\noi
where  the conditions $|n_1 + n_2| \sim |n_3|$ in the first term
and $|n|\sim 1$ in the second term signify the resonant product $\pe$.
From ~\eqref{X3} and Lemma \ref{LEM:reg}, 
we easily see that  $\RR_2 \in C(\R_+; C^\infty(\T^2))$
almost surely, provided that 
$\al<\frac 12$.
Therefore, it suffices to show
\begin{align}
\E\big[|\ft \RR_1(n, t)  |^2\big]
&  \les_T \jb{n}^{-2-2(s_\al-\al)+}
\label{Y2}
\end{align}
for $n \in \Z^2$ and $0 \le t \le T$, uniformly in $N \in \N$,
where $s_\al$ is as in \eqref{exreg}.
As in \cite{GKO2},   decompose $\ft \RR_1$ as 
 \begin{align}
\ft \RR_1(n, t)
& = 
\frac 1{4\pi^2}
\sum_{\substack{n = n_1 + n_2+n_3\\|n_1 + n_2| \sim |n_3|\\ (n_1 + n_2)(n_2 + n_3) (n_3 + n_1) \ne 0}} 
\int_0^t  \frac{\sin ((t-t') \jb{ n_1 + n_2 })}{\jb{ n_1 + n_2 }} \notag\\
& \hphantom{XXXXXXXXXlllllXX}
\times 
\ft{\<1>}_N(n_1, t')\ft{\<1>}_N(n_2, t') dt' \cdot  \ft{\<1>}_N(n_3, t) \notag\\
& \hphantom{X} 
+ \frac 1{2\pi^2} 
\int_0^t \ft{\<1>}_N(n, t') 
\bigg[\sum_{\substack{n_2 \in \Z^2\\|n_2| \sim |n+n_2|\ne 0 \\ |n_2| \le N}}  \frac{\sin ((t-t') \jb{ n + n_2 })}{\jb{ n + n_2 }} \notag\\
& \hphantom{XXXXXXXXXlllllXX}
\times 
\Big(\ft{\<1>}(n_2, t')\ft{\<1>}(-n_2, t) - \s_{n_2}(t,  t')\Big)\bigg] dt'    \notag\\
& \hphantom{X} 
+ \frac 1{2\pi^2} 
\int_0^t \ft{\<1>}_N(n, t') 
\bigg[\sum_{\substack{n_2 \in \Z^2\\| n_2| \sim |n+n_2| \ne 0 \\ |n_2| \le N}} \frac{\sin ((t-t') \jb{ n + n_2 })}{\jb{ n + n_2 }} 
 \s_{n_2}(t, t')\bigg] dt'   \notag\\
& \hphantom{X} 
- \frac 1{4\pi^2}\cdot \ind_{n\ne0}
\int_0^t  \frac{\sin ((t-t') \jb{ 2n })}{\jb{ 2n }} 
\, (\ft{\<1>}_N(n, t'))^2dt' \cdot \ft{\<1>}_N(-n, t)\notag\\
&
=: \ft \RR_{11}(n, t)+  \ft \RR_{12}(n, t)+  \ft \RR_{13}(n, t)+  \ft \RR_{14}(n, t), 
\label{Y3}
  \end{align}

\noi
where $\RR_{12}$ and $\RR_{14}$ correspond to the ``renormalized'' contribution
from $n_1 + n_3= 0$ or $n_2 + n_3= 0$ and the contribution
from $n_1 = n_2 = n = - n_3$, respectively.

Proceeding as in \cite{GKO2}
(and noting
that $|n+ n_2| \sim|n_2|$ implies $|n_2| \ges |n|$), 
 we can estimate 
$\ft \RR_{12}$ and $\ft \RR_{14}$
and show that they satisfy \eqref{Y2}.
As for $\ft \RR_{11} $, 
by applying 
Jensen's inequality as in \cite{GKO2}
(see also 
 Section~10 in~\cite{Hairer}
 and the discussion on $\<31p>$ in Section~4 of \cite{MWX})
and then   \eqref{sconv3}, \eqref{S3a}, and Lemma \ref{LEM:SUM} (ii)
(noting that $|m|\sim |n_3|$), 
 we obtain
 \begin{align}
 \begin{split}
\E\big[|\ft \RR_{11}(n, t)  |^2\big]
&  \les 
\sum_{\substack{n = m + n_3\\ |m|\sim|n_3|}}
\E \big[ |\ft{\<20>}(m, t)|^2\big]
\E \big[ |\ft{\<1>}(n_3, t)|^2\big] \\
& \les_T
\sum_{\substack{n = m + n_3\\ |m|\sim|n_3|}} \frac{1}{\jb{m}^{2+2s_\al-} \jb{n_3}^{2-2\al}} \\
&\les \jb{n}^{-2-2(s_\al-\al)+}, 
\end{split}
\label{Y3w} 
  \end{align}

\noi
provided that  $0<\al<\frac 12$.

Lastly, we consider   $\ft\RR_{13}$ 
in   \eqref{Y3}.
Let  $0 \leq t_2 \leq t_1 \leq T$.
Then, from \eqref{Y3} with \eqref{sigma2},  we have
\begin{align*}
\begin{split}
\E\big[|\ft \RR_{13}(n, t)|^2\big]
& =\frac 1{2\pi^4}
\ind_{|n| \le N}
\sum_{k_0, k_1, k_2 \in \{1, 2\}}
\int_0^t 
\int_0^{t_1} 
\s_{n}^{(k_0)}(t_1, t_2)
 \\
& \hphantom{XX}
\times 
\bigg[\sum_{\substack{n_2 \in \Z^2\\| n_2| \sim |n+n_2| \ne0 \\ |n_2| \le N}} \frac{\sin ((t-t_1) \jb{ n + n_2 })}{\jb{ n + n_2 }} 
\s_{n_2}^{(k_1)}(t, t_1)\bigg]  \\
& \hphantom{XX}
\times 
\bigg[\sum_{\substack{n'_2 \in \Z^2\\| n'_2| \sim |n+n'_2| \ne0 \\ |n_2'| \le N}} \frac{\sin ((t-t_2) \jb{ n + n'_2 })}{\jb{ n + n'_2 }} 
\s_{n_2'}^{(k_2)}(t, t_2)\bigg]
   dt_2 dt_1 \\
& =: \sum_{k_0, k_1, k_2 \in \{1, 2\}} \1^{(k_0,k_1,k_2)}(n,t) , 
\end{split}
\end{align*}

\noi
where  $\s_{n}(t, t' )   = \s_{n}^{(1)}(t, t' )  + \s_{n}^{(2)}(t, t' )  $
as in \eqref{S61} and \eqref{S62}.
In the following, we only consider the contribution
from 
  $(k_0, k_1, k_2) = (1, 1, 1)$,
  since, in the other cases, 
  the desired bound \eqref{Y2} trivially follows
  from Lemma \ref{LEM:SUM}\,(ii)
without using any oscillatory behavior.

By a direction computation with  \eqref{S61}, 
we have 
\begin{align}
 & \1^{(1,1,1)} (n,t)
\notag   \\
& \sim 
\ind_{|n| \le N} \sum_{\substack{\eps_j \in \{-1, 1\}\\j = 1, \dots, 5}}
\sum_{\substack{n_2 \in \Z^2\\| n_2| \sim |n+n_2|\ne 0 \\ |n_2| \le N}} 
\sum_{\substack{n'_2 \in \Z^2\\| n'_2| \sim |n+n'_2| \ne0 \\ |n_2'| \le N}}
\frac{
\eps_1\eps_2e^{i  t (\eps_1 \jb{n+n_2} 
+   \eps_2  \jb{n+n_2'}
+ \eps_3  \jb{n_2} +  \eps_4  \jb{n_2'})}
}{\jb{n}^{2(1-\al)} \jb{n+n_2} \jb{n_2}^{2(1-\al)} \jb{n+n_2'} \jb{n_2'}^{2(1-\al)}}
\notag  \\
& \hphantom{XXXX}
\times  \int_0^t 
t_1  e^{ - i t_1  \kk_3(\bar n)}
 \int_0^{t_1} 
t_2^2  e^{- i t_2\kk_4(\bar n')} 
dt_2 dt_1,
\label{Y5}  
\end{align}

\noi
where
$\kk_3(\bar n)$ and $\kk_4(\bar n)$ are defined by 
\begin{align*}
\begin{split}
\kk_3(\bar n) 
& = \eps_1 \jb{n+n_2}  
+ \eps_3  \jb{n_2}
 - \eps_5  \jb{n}, 
 \\
\kk_4(\bar n') 
& = 
 \eps_2  \jb{n+n_2'}
+  \eps_4  \jb{n_2'}
+  \eps_5  \jb{n}.
\end{split}
\end{align*}

\noi
Note that 
for  $\al< \frac 12$, the sums over $n_2$ and $n_2'$ in~\eqref{Y5} are absolutely convergent.
This makes our analysis simpler than the proof of Proposition 1.8 in \cite{GKO2},
where the corresponding sums in $n_2$ and $n_2'$ were not absolutely convergent
and hence, it was crucial to exploit the oscillatory nature of the problem
and also apply some symmetrization argument.

By first  integrating~\eqref{Y5} in $t_1$
when  $|\kk_3(\bar n)|\geq 1$
and simply bounding the integral in~\eqref{Y5} by $C(T)$
when  $|\kk_3(\bar n)|<  1$,
and then applying Lemma \ref{LEM:SUM}\,(ii),
we have
\begin{align}
|\1^{(1,1,1)} (n,t)|
& \les_T \frac{1}{\jb{n}^{2(1-\al)}}
\sum_{\substack{n_2 \in \Z^2\\|n + n_2| \sim |n_2|}} 
\frac{1}
{ \jb{n+n_2} \jb{n_2}^{2(1-\al)} (1+ |\kk_3(\bar n)|)} \notag\\
& \hphantom{XXXX}
\times
\sum_{\substack{n'_2 \in \Z^2\\|n + n'_2| \sim |n'_2|}}
\frac{1}
{ \jb{n+n_2'} \jb{n_2'}^{2(1-\al)}} \notag\\
& \les \frac{1}{\jb{n}^{3-4\al}}
\sum_{\substack{n_2 \in \Z^2\\|n + n_2| \sim |n_2|}} 
\frac{1}
{ \jb{n+n_2} \jb{n_2}^{2(1-\al)} (1+ |\kk_3(\bar n)|)}
\label{Y7}
\end{align}

\noi
for $\al < \frac 12$.
In the following, 
 we only consider the case $(\eps_1, \eps_3, \eps_5) = (\pm 1, \mp 1, \pm1)$,
since the other cases are  handled in an analogous manner.
See also the proof of  Proposition \ref{PROP:sto1}.

In this case, by repeating the argument in Case 2 of the proof of Proposition \ref{PROP:sto1}
(in particular, \eqref{SS2} with $(n, n+n_2, -n_2)$ replacing $(n, n_1, n_2)$), 
we have
\begin{align}
|\kk_3(\bar n)| = |\jb{n+n_2}-\jb{n_2}-\jb{n}|
\ges  \frac{1 + |n| |n_2| (1 - \cos \theta)}
{\jb{n + n_2}},  
\label{Y9a}
\end{align}
where $\theta = \angle (n,n_2)$.
Then, 
as in \eqref{X6a}, 
it follows 
from 
\eqref{Y9a} and 
$1+|\kk_3(\bar n)| \ges |\kk_3(\bar n)|^{\frac 12}$
with $|n_2| \sim N_2$
that 
\begin{align}
\Theta_3 : = \frac{1 + N_2^2 \ta}{1+|\kk_3(\bar n)|}
\les \jb{n+n_2}^\frac{1}{2}
\frac{1 + N_2^2 \ta}
{1 + \jb{n}^\frac 12  N_2^\frac 12 \ta}
\les \frac{ N_2^2}{\jb{n}^\frac{1}{2}}, 
\label{Y9aa}
\end{align}

\noi
since  $N_2^3 \ges \jb{n}$ under the condition $|n + n_2|\sim |n_2|$.

When  $1 - \cos \ta \ges 1$, 
by summing over $n_2$ in \eqref{Y7}
with \eqref{Y9a} and Lemma \ref{LEM:SUM}, 
we obtain
\begin{align}
|\1^{(1,1,1)} (n,t)|
\les_T \jb{n}^{-5+6\al}
\le \jb{n}^{-2-2(s_\al-\al)}
\notag
\end{align}
for $\al <\frac 12$.

Next, consider the case  $1 - \cos \ta \sim \ta^2 \ll 1$.
We see that 
for fixed $n \in \Z^2$, 
 the range of possible $n_2$ with $|n_2| \sim N_2$, dyadic $N_2\geq 1$, 
is constrained to a trapezoid
whose height is $\sim N_2 |\cos \ta| \sim N_2$
and the width $\sim N_2 \sin \ta \les N_2 \ta$. 
Then,  from \eqref{Y7} with a dyadic decomposition as in \eqref{SS1d}
and \eqref{SS1e}, 
 and \eqref{Y9aa}, 
we have 
\begin{align}
|\1^{(1,1,1)} (n,t)|
& \les_T \frac{1}{\jb{n}^{3-4\al}}
\sum_{\substack{N_2 \ges \jb{n} \\\text{dyadic}} }
\frac{1}{N_2^{3-2\al}}
\Theta_3
 \notag\\
&\les \frac{1}{\jb{n}^{\frac 72-4\al}}
\sum_{\substack{N_2 \ges \jb{n} \\\text{dyadic}} }
\frac{1}{N_2^{1-2\al}} \notag\\
&\les
\jb{n}^{- \frac 92 +6\al}
\les \jb{n}^{-2-2(s_\al-\al)}
\notag
\end{align}
for $\al< \frac 12$.
We therefore obtain \eqref{Y2}.

Since the summations above are absolutely convergent,
a slight modification of the argument yields the time difference estimate \eqref{reg2}
and the estimates \eqref{reg3} and \eqref{reg4}
for proving convergence of $\<21p>_N$ to $\<21p>$
by  
 Lemma \ref{LEM:reg}.
This completes the proof of Proposition~\ref{PROP:sto2}.

\begin{remark} \rm \label{REM:21reg}
As pointed above, 
the sums in \eqref{Y5} is absolutely convergent for $\al < \frac 12$.
Therefore, even without exploiting multilinear dispersion, 
we can make sense of $\<21p>$.
In the following, by a crude estimate, we show 
\begin{align}
\<21p> \in C([0,T];W^{1-3\al-,\infty}(\T^2)) 
\label{X7a}
\end{align}

\noi
almost surely.

Note that $\ft \RR_{12}$ and $\ft \RR_{14}$
satisfy \eqref{Y2} without making use of any dispersion.
Thus, we only need to consider
 $\ft \RR_{11}$ and $\ft \RR_{13}$.
Let $\1 (n,t)$ be as in \eqref{S5}.
Then, by applying Lemma~\ref{LEM:SUM}
to  \eqref{S6a} with \eqref{sigma2}, we have
\[
|\1 (n,t)|
\les_T \frac{1}{\jb{n}^2} \sum_{n=n_1+n_2} \frac{1}{\jb{n_1}^{2(1-\al)} \jb{n_2}^{2(1-\al)}}
\les \jb{n}^{-4+4\al}
\]
for $0<\al<\frac 12$.
Together with \eqref{X3a}, 
this implies 
\begin{align}
\E \big[ |\ft{\<20>}(n, t)|^2\big]
\les_T \jb{n}^{-4+4\al}
\label{X8}
\end{align}

\noi
even when we do not exploit multilinear dispersion.
Then, 
using \eqref{sconv3} and \eqref{X8}, 
we can replace
\eqref{Y3w} 
by 
\begin{align}
\begin{split}
\E\big[|\ft \RR_{11}(n, t)  |^2\big]
&  \les 
\sum_{\substack{n = m + n_3\\ |m|\sim|n_3|}}
\E \big[ |\ft{\<20>}(m, t)|^2\big]
\E \big[ |\ft{\<1>}(n_3, t)|^2\big] \\
&\les_T \jb{n}^{-4+6\al}
 =  \jb{n}^{-2-2(1-3\al)}.
\end{split}
\label{X9}
\end{align}

\noi
By ignoring all the oscillatory factors in \eqref{Y5}, 
we obtain
\begin{align}
\E\big[|\ft \RR_{13}(n, t)|^2\big]
& \les_T \frac{1}{\jb{n}^{2(1-\al)}}
\sum_{\substack{n_2 \in \Z^2\\|n + n_2| \sim |n_2|}} 
\frac{1}
{ \jb{n+n_2} \jb{n_2}^{2(1-\al)}}
\sum_{\substack{n'_2 \in \Z^2\\|n + n'_2| \sim |n'_2|}}
\frac{1}
{ \jb{n+n_2'} \jb{n_2'}^{2(1-\al)}} \notag \\
&\les \jb{n}^{-4+6\al}.
\label{X10}
\end{align}

\noi
Therefore, 
\eqref{X7a} follows from 
 Lemma \ref{LEM:reg}
with  \eqref{X9}, \eqref{X10},
and  the trivial bounds  for $\ft \RR_{12}$ and $\ft \RR_{14}$.
\end{remark}

\subsection{Divergence of the stochastic terms}
\label{SUBSEC:34}

In this subsection, we present the proof of Proposition \ref{PROP:stodiv}.
By \eqref{stoc2} and \eqref{so4b} with \eqref{sigma3}, for $n \in \Z^2$ and $t>0$, we can write
\begin{align}
\ft{\<20>}_N (n,t)
= \frac {1}{2\pi} \sum_{\substack{k \in \Z^2\\k \preccurlyeq n - k \\ |k|, |n-k| \le N}} X_k (n,t),
\label{div2a}
\end{align}
where $\preccurlyeq$ denotes the lexicographic ordering of $\Z^2$ and 
\begin{align}
X_k (n,t) :=
( 2 - \ind_{n = 2k})
\int_0^t \frac{\sin ((t-t')\jb{n})}{\jb{n}} \Big( \ft{\<1>}(k,t') \ft{\<1>}(n-k,t') - \ind_{n = 0} \cdot \s_k (t',t') \Big) dt'.
\label{div2x}
\end{align}

\noi
Note that $X_k(n, t)$'s are independent.
We show that the sum in \eqref{div2a} diverges almost surely.
We only consider the case $|k| \sim |n-k| \gg |n|$.
Otherwise, we have either $|k|\sim  |n| \ges |n-k|$
or $|n-k|\sim  |n| \ges |k|$.
In either case, for fixed $n \in \Z^2$, 
the sum in $k$ is  a finite sum and hence is almost surely convergent.
This allows us to focus on 
 the case $|k| \sim |n-k| \gg |n|$.
In particular, we assume $k \neq \frac n2$.

As in \eqref{S5}, we have 
\begin{align}
\begin{split}
\E [ X_k (n,t)] &=0, \\
\E \big[ |X_k(n,t)|^2\big] &= 8 \int_0^t \frac{\sin ((t-t_1)\jb{n})}{\jb{n}} \int_0^{t_1} 
\frac{\sin ((t-t_2)\jb{n})}{\jb{n}}\\
& \hphantom{XXXX}
\times  \s_k (t_1,t_2) \s_{n-k}(t_1,t_2) dt_2 dt_1. 
\end{split}
\label{div2aa}
\end{align}

\noi
When $n = 0$ (which implies $k \ne 0$ under the assumption $k \ne \frac n2$), 
we used \eqref{1mon4}.
From \eqref{sigma2} and $|k| \sim |n-k|$, we have
\begin{align}
\begin{aligned}
\s_{k} (t_1,t_2) \s_{n-k} (t_1,t_2)
&= \frac{\cos ( (t_1-t_2)(\jb{k}-\jb{n-k}) )}{8 \jb{k}^{2(1-\al)} \jb{n-k}^{2(1-\al)}} t_2^2\\
& \hphantom{X}
+ \frac{\cos ( (t_1-t_2)(\jb{k}+\jb{n-k}) )}{8 \jb{k}^{2(1-\al)} \jb{n-k}^{2(1-\al)}} t_2^2\\
&\quad + O \big( \jb{t_2} \jb{k}^{-5+4\al} \big).
\end{aligned}
\label{div20}
\end{align}

\noi
The contribution to \eqref{div2aa} from the first term on the right-hand side of \eqref{div20} is worst.
Indeed, we can use the dispersion to estimate
the contribution to \eqref{div2aa} from  the second term on the right-hand side of \eqref{div20}.
Namely, by integrating in $t_2$ and using 
$|k| \sim |n-k| \gg |n|$, we have
\begin{align}
\begin{aligned}
\bigg| \int_0^t & \frac{\sin ((t - t_1) \jb{n})}{\jb{n}}
 \int_0^{t_1} \frac{\sin ((t-t_2) \jb{n})}{\jb{n}}
  \frac{\cos ( (t_1-t_2)(\jb{k}+\jb{n-k}) )}{\jb{k}^{2(1-\al)} \jb{n-k}^{2(1-\al)}} t_2^2 dt_2 dt_1 \bigg|  \\
&\les  \frac{1}{\jb{n}^2 \jb{k}^{2(1-\al)} \jb{n-k}^{2(1-\al)}} \\
& \hphantom{XXX}
\times
\sum_{\eps_1, \eps_2 \in \{-1, 1\}}
\int_0^t
\bigg|\int_0^{t_1}
e^{-it_2 (\eps_1 \jb{n} + \eps_2(\jb{k} + \jb{n-k}))}
 t_2^2 d t_2 \bigg| dt_1\\
&\les_t \frac{1}{\jb{n}^2 \jb{k}^{5-4\al}}.
\end{aligned}
\label{div2b}
\end{align}

Now, let us estimate the contribution to \eqref{div2aa} from the first term
on the right-hand side of \eqref{div20}.
Given $n \in \Z^2$, 
we choose small $t > 0$
such that  $t \jb{n} \ll 1$, 
which implies 
\begin{align}
\frac{\sin(t\jb{n})}{\jb{n}}
\ges t 
\qquad\text{and}
\qquad
\cos (t\jb{n})
\ges 1.
\label{div2c}
\end{align}

\noi
Noting that 
\[
|\jb{k}-\jb{n-k}| = \frac{||k|^2-|n-k|^2|}{\jb{k}+\jb{n-k}}
\les |n|,
\]

\noi
it follows from \eqref{div2c} that
\begin{align}
\int_0^t & \frac{\sin ((t - t_1) \jb{n})}{\jb{n}}
 \int_0^{t_1} \frac{\sin ((t-t_2) \jb{n})}{\jb{n}} \frac{\cos ( (t_1-t_2)(\jb{k}-\jb{n-k}) )}{\jb{k}^{2(1-\al)} \jb{n-k}^{2(1-\al)}} t_2^2  dt_2 dt_1 \notag\\
&\ges \frac{1}{\jb{k}^{4(1-\al)}}
\int_0^t (t-t_1) \int_0^{t_1} (t-t_2) t_2^2 dt_2 dt_1 \notag\\
&\ges \frac{t^6}{\jb{k}^{4(1-\al)}}.
\label{div2d}
\end{align}

\noi
By \eqref{div2aa}, \eqref{div20}, \eqref{div2b}, and \eqref{div2d}, we obtain
\begin{equation}
\E \big[ |X_k(n,t)|^2\big]
\ges \frac{t^6}{\jb{k}^{4(1-\al)}} - \frac{1}{\jb{n}^2 \jb{k}^{5-4\al}}
\ges \frac{t^6}{\jb{k}^{4(1-\al)}}
\label{div2e}
\end{equation}

\noi
for $|k| \gg t^{-6}\jb{n}^{-2}$ and $t \jb{n} \ll 1$.
This  implies that
\[
\begin{split}
\sum_{\substack{k \in \Z^2 \\k \preccurlyeq n - k \\ |n| \ll |k| \le N}} \E \big[ |X_k(n,t)|^2 \big]
&\ges_{t,n}
\sum_{\substack{k \in \Z^2 \\ t^{-6}\jb{n}^{-2} \ll |k| \le N}}
\frac{1}{\jb{k}^{4(1-\al)}}
\ges_{t,n}
\begin{cases}
\log N, & \text{if } \al =\frac{1}{2} \\
N^{-2+4\al}, & \text{if }\al > \frac{1}{2}
\end{cases} \\
&\longrightarrow \infty
\end{split}
\]
as $N \to \infty$.
Hence, Kolmogorov's three-series theorem (\cite[Theorem 2.5.8]{Durr}) yields that $P \big( | \lim_{N \to \infty} \ft{\<20>}_N(n,t)| < \infty \big) <1$.
Moreover,
recalling  the independence of $\{X_k(n, t)\}_{k \in \Z^2, k \preccurlyeq n - k}$, 
 it follows from Kolmogorov's zero-one law (\cite[Theorem 2.5.3]{Durr}) that 
 \[P \Big( \big| \lim_{N \to \infty} \ft{\<20>}_N(n,t) \big| <\infty \Big) =0.\]
 
 \noi
In particular, we obtain that $\{ \<20>_N \}_{N \in \N}$ 
forms a divergent sequence  in $C([0,T]; \mathcal{D}'(\T^2))$ almost surely for any $T>0$.

\begin{remark}\label{REM:divcov} \rm 
(i) 
From \eqref{S5}, \eqref{X3a}, and \eqref{div2aa} we have
\[
\E \big[ |\ft{\<20>}_N(n,t)|^2 \big]
= \frac 1{4\pi^2}
 \sum_{\substack{k \in \Z^2\\ k \preccurlyeq n - k \\ |k|, |n-k| \le N}} \E \big[ |X_k (n,t)|^2 \big]
+  \ind_{n \in 2 \Z^2\setminus\{0\}} \cdot O\big(\jb{n}^{-6 + 4\al}\big).
\]

\noi
This shows that 
 Proposition \ref{PROP:stodiv} is a consequence of $\lim_{N \to \infty} \E \big[ |\ft{\<20>}_N(n,t)|^2 \big] =\infty$.

\smallskip

\noi
(ii)
Note that the calculations in \eqref{div2b}, \eqref{div2c}, and \eqref{div2d} are independent of dimensions.
In particular,
the lower bound \eqref{div2e} is also valid on $\T^d$.
From this observation, 
we conclude that  $\{ \<20>_N \}_{N \in \N}$ 
forms a divergent sequence  in $C([0,T]; \mathcal{D}'(\T^d))$ almost surely if $\al \ge 1-\frac d4$.
Note that for $d \geq  5$, 
we need to apply smoothing (i.e.~$\al < 0$)
in order to construct the second order process
  $ \<20>$ as a limit of  $\{ \<20>_N \}_{N \in \N}$. 

Since the critical value given by the probabilistic scaling is
$\al_* = \min \big(\frac{5-d}{4}, \frac{5-d}{2}\big)$, 
we see that 
 the existing solution theory such as the Da Prato-Debussche trick
or its higher order variants
breaks down
at $\al  =  1-\frac d4$
before
  reaching the critical value $\al_*$ in 
  dimensions $d = 1, \dots, 5$.

\end{remark}

\section{Stochastic nonlinear heat equation with rough noise}
\label{SEC:heat}

In this section, we consider SNLH \eqref{SNLH1}.
In Subsection \ref{SUBSEC:heat1}, we first state the regularity properties
of the relevant stochastic terms 
and present a proof of Theorem \ref{THM:WP2}
by reformulating the problem in terms of the residual term 
$v = u - \<1> + \<20>$.
We then proceed with the construction
of the stochastic terms in 
the remaining part of this section.
This includes the divergence 
of $\<2>$ (and $\<20>$, respectively)
for $\al \geq \frac 12$ (and $\al \geq 1$, respectively)
stated in Proposition \ref{PROP:heat}.

\subsection{Reformulation of SNLH}
\label{SUBSEC:heat1}

Let $\al > 0$.
We define 
the truncated stochastic convolution $\<1>_N = \I( \jb{\nb}^\al \pi_N \xi)$ 
by 
\begin{align}
\<1>_N
: = \int_{-\infty}^t P(t- t')  \jb{\nb}^{\al} \pi_N dW(t')
= \sum_{\substack{n \in \mathbb{Z}^2 \\ |n| \le N}} e_n 
  \int_{-\infty}^t e^{-(t-t') \jb{n}^2} \jb{n}^{\al} d \beta_n (t')
\label{so4bHa}
\end{align}

\noi
for $t \geq 0$, 
where $\pi_N$, $P(t)$, $e_n$, and $W(t)$   are as in \eqref{pi}, \eqref{heat1},  \eqref{exp}, and  \eqref{Wpro},  respectively.
We denote the limit of $\<1>_N$ by $\<1>$:
\begin{align}
\<1> = \lim_{N \to \infty} \<1>_N
=\int_{-\infty}^t P(t- t')  \jb{\nb}^{\al} dW(t').
\label{stoc1H}
\end{align}
We then define the truncated Wick power  $\<2>_N$  by 
\begin{align}
\<2>_N   := (\<1>_N)^2 - \kk_N.
\label{so4bHb}
\end{align}

\noi
where
$\kk_N$ is defined by 
\begin{align}
\kk_N : \! & = \E\big[ (\<1>_N(x, t))^2\big]
=  \frac 1 {4 \pi^2}\sum_{|n|\leq N}
\int_{-\infty}^t \Big( e^{-(t-t') \jb{n}^2} \jb{n}^\al \Big)^2 dt' \notag\\
& = \frac 1{8\pi^2} \sum_{|n|\leq N} \frac 1{\jb{n}^{2(1-\al)}}
\sim N^{2\al}.
\label{kap}
\end{align}

\noi
Then, by proceeding as in the proof of Lemma \ref{LEM:stoconv}\,(i), 
we have the following regularity and convergence property
of $\<1>_N$.
Since the argument is standard, we omit details.

\begin{lemma}\label{LEM:stoconv2}
Let $T > 0$. Given  $\al \in \R$ and $s<-\al$,
$\{ \<1>_N \}_{N \in \N}$  defined in \eqref{so4bHa} is
 a Cauchy sequence
in $C([0,T];\C^{s}(\T^2))$
almost surely.
 In particular, 
denoting the limit by $\<1>$,
we have
  \[\<1> \in C([0, T];\C^{-\al - \eps}(\T^2))
  \]
  
  \noi
  for any $\eps > 0$, 
  almost surely.

\end{lemma}

We now define the second order stochastic term:
\[
\<20>_N := \I (\<2>_N).
\]

\noi
Then, 
a slight modification of the proof of Proposition \ref{PROP:heat}\,(ii)
presented below
shows that  
 $\<20>_N$ converges to 
\begin{align*}
  \<20>(t) := 
   \I (\<2>)(t)
  =  \int_{-\infty}^t  P(t- t')\<2>(t') dt'
\end{align*}

\noi
in  $C([0,T];\C^{2-2\al-}(\T^2))$ almost surely, 
 provided that $0 < \al < 1$.

From  the regularities $2-2\al-$ and $-\al-$ of $\<20>$ and $\<1>$, 
there is an issue in making sense of the resonant product 
 $\<20>\pe \<1>$
in the deterministic manner when $\al \geq \frac 23$.
For the range $\frac 23 \leq \al < 1$, 
we instead use stochastic analysis to 
define
 the resonant product 
``$\<20>\pe \<1>$''
as a suitable limit of 
$\<21p>_N := \<20>_N\pe \<1>_N.$

\begin{lemma}\label{LEM:heat2}
Let $0<\al<1$.
Given any  $T>0$, 
$\{ \<21p>_N \}_{N \in \N}$ is
a Cauchy sequence
in 
$C([0,T];\C^{2-3\al-\eps}(\T^2))$
for any $\eps > 0$, 
almost surely.
In particular,
denoting the limit by $\<21p>$,
we have
  \[\<21p> \in C([0,T];\C^{2 -3\al-\eps}(\T^2))\]
  
  \noi
almost surely.
\end{lemma}

In the following, we only consider the range $\frac 23 \leq \al < 1$
since the case $0 < \al  < \frac 23$ can be handled by 
the standard Da Prato-Debussche trick
as mentioned in Section \ref{SEC:1}.
As in Subsection \ref{SUBSEC:reno}, we 
proceed with the second order expansion \eqref{X1}.
Then, after a proper renormalization, 
the residual term $v = u - \<1>+ \<20>$
satisfies the following equation:
\begin{equation} 
\begin{cases}
\dt v + (1-\Dl)v =
-v^2 - 2v (\<1>-\<20>) - \<20>^2 + 
2( \<20> \pl \<1> +  \<21p> 
+  \<20>\pg \<1>) \\
v |_{t = 0} = v_0, 
\label{SNLH7}
\end{cases}
\end{equation}

\noi
where $v_0 = v_0^\o$ is given by 
\begin{align}
v_0 = u_0 - \<1>(0) + \<20>(0).
\label{IV5}
\end{align}

Given $s < \s$ and $T>0$, define $X(T) 
\subset C([0, T] ; \C^s(\T^2))\cap 
C((0, T]; \C^\s(\T^2))$  by the norm:
\begin{align*}
\|v\|_{X(T)} = \|v \|_{C_T\C^s_x} + \|v\|_{Y(T)}, 
\end{align*}

\noi
where the $Y(T)$-norm is given by 
\[\|v\|_{Y(T)} = \sup_{0< t \leq T} t ^\frac {\s - s}{2}\|v(t) \|_{\C^\s}.\]

\noi
We point out that the $Y(T)$-norm is needed to handle
rough initial data in $\C^s(T^2)$, which does not belong to $\C^{\s}(\T^2)$.
The use of this type of norm, allowing a blowup at time $t = 0$, 
is standard in the study of the parabolic equations.
See, for example,
\cite{BC, MW1}.
We then have the following local well-posedness of the perturbed SNLH \eqref{SNLH7}.

\begin{theorem}\label{THM:LWPv2}
Let $0<\al<1$ and $s > -\al - \eps$ for sufficiently small $\eps > 0$.
Then,
the Cauchy problem \eqref{SNLH7} is locally well-posed
in $\C^s(\T^2)$.
More precisely, 
given any $u_0 \in \C^s(\T^2)$, 
there exist  an almost surely positive stopping time $T  = T(\o)$ 
and   a unique solution $v$ to~\eqref{SNLH7}
in the class\textup{:}
\[  X(T)
\subset C([0, T] ; \C^s(\T^2))\cap 
C((0, T]; \C^\s(\T^2)),\]

\noi
where  $-s < \s <  s+2$. 
Furthermore, 
the solution $v$
depends  continuously 
on the enhanced data set\textup{:}
\begin{align*}
\Si = \big(u_0, \<1>, \<20>,  \<21p>\big)
\end{align*}

\noi
almost surely belonging to 
 the class\textup{:}
\begin{align}
\mathcal{Z}^{s, \eps}_T
&  = \C^s(\T^2) \times 
C([0,T]; \C^{-\al - \eps}(\T^2)) \notag\\
& \hphantom{X}
\times 
C([0,T]; \C^{2-2\al- \eps}(\T^2)) 
\times
C([0,T]; \C^{2-3\al - \eps, \infty}(\T^2)).
\notag
\end{align}
\end{theorem}

Once we prove Theorem \ref{THM:LWPv2}, 
Theorem \ref{THM:WP2} follows from the same lines
as in  the proof of Theorem \ref{THM:WP}
and thus we omit details.

\begin{proof}[Proof of Theorem \ref{THM:LWPv2}]
Let $0 < T \leq 1$
and fix $\eps > 0$ sufficiently small.
Define a map $\G$ on $X(T)$ by 
\begin{align*}
\G (v) (t) &= P(t) v_0
- \int_0^t P(t - t') 
\big[ v^2 + 2v (\<1>-\<20>) + \<20>^2 - 2 \<21> \big] (t') dt', 
\end{align*}

\noi
where $v_0$ is as in \eqref{IV5} and 
$\<21>$ is as in \eqref{Z1}.
From 
Lemma \ref{LEM:heat2}, 
Proposition \ref{PROP:heat}, and Lemma~\ref{LEM:stoconv2}
with  Lemma \ref{LEM:para}, we see that 
\begin{align}
\<21> \in 
C([0,T]; \C^{-\al - \eps, \infty}(\T^2))
\label{heat3}
\end{align}

\noi
almost surely.

For simplicity, we only consider the case $s = -\al - \eps$.
From Lemmas \ref{LEM:Schauder} and \ref{LEM:para} 
along with 
 Lemma \ref{LEM:stoconv2}, 
Proposition \ref{PROP:heat}, and \eqref{heat3}, 
we have
\begin{align}
\begin{split}
\| \G (v) \|_{L_T^\infty \C_x^s}
&\les \| v_0  \|_{\C^s} + 
\bigg\| \int_0^t \Big( \| v^2  \|_{ \C^{s}} 
+ \| v (\<1>-\<20>) \|_{ \C^{s}} \Big) (t') dt'\bigg\|_{L^\infty_{T}}\\
&\hphantom{X}
+ T\Big( \| \<20>^2 \|_{L_T^\infty \C_x^{s}} + \| \<21> \|_{L_T^\infty \C_x^{s}}
\Big)
 \\
&\les \| u_0 \|_{\C^s} +
\| \<1>(0)\|_{\C^s}  + \|\<20>(0)\|_{\C^s} \\
&\hphantom{X}
+  \int_0^T (t')^{-\frac{\s-s}{2}}  
 dt' \cdot \big( \| v  \|_{ L^\infty_T \C_x^{s}} 
+ \|\<1>-\<20> \|_{L^\infty_T  \C_x^{s}} \big)
\| v  \|_{Y(T)} 
 \\
&\hphantom{X}
+ T\Big( \| \<20> \|^2_{L_T^\infty \C_x^{2-2\al-}} + \| \<21> \|_{L_T^\infty \C_x^{s}}
\Big)
 \\
& \leq 
\| u_0 \|_{\C^s}
+ 
 T^\ta \Big(  \| v \|_{X(T)}^2
+ C_\o  \| v \|_{X(T)}\Big) 
+C_\o  
\end{split}
\label{LWP1H}
\end{align}

\noi
for  some almost surely finite constant $C_\o > 0$
and $\ta> 0$, 
provided that $\al < 1$ and $- s < \s <  s+ 2$.

Next, we estimate the $Y(T)$-norm of $\G(v)$.
Under the condition $\s < s + 2$, 
a change of variable yields
\begin{align}
t^\frac{\s - s}{2} \int_0^t (t - t')^{-\frac{\s-s}{2}}
(t')^{-\frac{\s-s}{2}}dt'
 = t^{1- \frac{\s - s}{2}}B\big(1-\tfrac{\s - s}{2}, 1-\tfrac{\s - s}{2}\big) 
 \les T^{1- \frac{\s - s}{2}}
\label{heat4}
 \end{align}

\noi
for $0 < t \leq T$, provided that $\s < s + 2$.
Here, $B(x, y) = \int_0^1 (1-\tau)^{x-1} \tau^{y-1} d\tau $ denotes the beta function.

Let $\NN(v) = v^2 +  2v (\<1>-\<20>) + \<20>^2 - 2 \<21>$.
Then, 
for 
 $0 < t \leq T$, 
it follows from
 Lemmas~\ref{LEM:Schauder} and~\ref{LEM:para}
 and \eqref{heat4} 
along with 
 Lemma~\ref{LEM:stoconv2}, 
Proposition \ref{PROP:heat}, and \eqref{heat3}
that 
\begin{align}
\begin{split}
t^\frac{\s - s}{2}\| \G (v)(t) \|_{ \C_x^\s}
&\les \| v_0 \|_{\C^s} + 
t^\frac{\s - s}{2} \int_0^t (t - t')^{-\frac{\s-s}{2}}\|\NN(v)(t') \|_{ \C^s}dt'\\
&\les \| u_0 \|_{\C^s} + 
t^{\frac{\s - s}{2} }  \int_0^t (t - t')^{-\frac{\s - s}{2}}
(t')^{- \frac{\s-s}{2}} dt' 
\\
& \hphantom{XXXXXX}
\times
 \big( \|v\|_{L^\infty_T \C^s_x}
+ \|\<1>-\<20> \|_{L^\infty_T  \C_x^{s}} \big)\|v\|_{Y(T)}
+ 
 C_\o\\
&  
\le \| u_0 \|_{\C^s} + 
T^\ta \Big( \| v \|_{X(T)}^2 
+ C_\o \| v \|_{X(T)}\Big) 
+C_\o ,  
\end{split}
\label{heat5}
\end{align}

\noi
 provided that $\al < 1$ and $- s < \s <  s+ 2$.

Taking a supremum of the left-hand side of \eqref{heat5} over $0 < t \leq T$, 
it follows from \eqref{LWP1H}
and~\eqref{heat5} that 
\begin{align}
\begin{split}
\| \G (v) \|_{X(T)}
& \leq 
\| u_0 \|_{\C^s}
+ 
 T^\ta \Big( \| v \|_{X(T)}^2 
+ C_\o \| v \|_{X(T)}\Big) 
+C_\o.  
\end{split}
\label{heat6}
\end{align}

\noi
By a similar computation, we also obtain a difference estimate:
\begin{align}
\| \G (v_1) - \G (v_2) \|_{X(T)}
&\les T^\ta \Big( \| v_1 \|_{X(T)} + \| v_2 \|_{X(T)} 
+ C_\o \Big) \| v_1-v_2 \|_{X(T)}.
\label{LWP2H}
\end{align}

\noi
Therefore, we conclude from \eqref{heat6} and 
 \eqref{LWP2H}
that 
a standard contraction  argument yields local well-posedness of \eqref{SNLH7}.
Moreover, an analogous  computation 
shows that  the solution $v \in X(T)$ depends continuously on 
the enhanced data set $\Si = \big(u_0, \<1>, \<20>,  \<21p>\big)$.
\end{proof}

\subsection{Proof of  Proposition \ref{PROP:heat}\,(i)}
\label{SUBSEC:51}

Given $n \in \Z^2$ and  $0 \leq t_2\leq t_1$,
define $\kk_{n}(t_1, t_2 )$ by 
\begin{align}
\begin{split}
\kk_{n}(t_1, t_2 )   
:\! & =
\E  \big[  \ft{\<1>}(n, t_1)  \,  \ft{\<1>}(-n, t_2) \big]
  =
   \int_{-\infty}^{t_2} e^{-(t_1-t') \jb{n}^2} \jb{n}^{\al} e^{-(t_2-t') \jb{n}^2} \jb{n}^{\al} d t' \\
&  = \frac{e^{-(t_1-t_2) \jb{n}^2}}{2 \jb{n}^{2(1-\al)}}.
\end{split}
\label{sigma2H}
\end{align}

First, we prove that
$\<2>_N \in C(\R_+;\C^{-2\al-} (\T^2))$ 
with a uniform (in $N$) bound,  
almost surely.
In view of \eqref{so4b} and 
\eqref{so4bHb}, 
by repeating the computation 
in the proof of Lemma \ref{LEM:stoconv}\,(ii)
(in particular  \eqref{2var1} and \eqref{2var2})
and applying   Lemma \ref{LEM:SUM}, we have
\begin{align}
\E\big[|\ft{\<2>}_N(n, t)|^2\big]
&\les \sum_{n = n_1 + n_2} \kk_{n_1}(t,t) \kk_{n_2}(t,t) \notag\\
&\sim  \sum_{n = n_1 + n_2} \frac{1}{\jb{n_1}^{2(1-\al)} \jb{n_2}^{2(1-\al)}}
\les \jb{n}^{-2+4\al},
\label{heat2-a}
\end{align}

\noi
provided that $0<\al<\frac 12$.
Since the time difference estimate follows from a slight modification,
Lemma \ref{LEM:reg} implies that $\<2>_N \in C(\R_+;\C^{-2\al-} (\T^2))$ almost surely.
Moreover, a slight modification
of the argument  yields
that $\{ \<2>_N \}_{N \in \N}$ is almost surely a Cauchy sequence in $C(\R_+;\C^{-2\al-} (\T^2))$, 
thus converging to some limit 
$\<2>$.
Since the  required modification is standard, 
we omit the details here.

Next, we show that
when $\al \ge \frac 12$,  
we show that $\{ \<2>_N\}_{N \in \N}$ forms
a divergent sequence  in $C([0,T]; \D'(\T^2))$ for any  $T> 0$ almost surely.
From \eqref{heat2-a}, we have
\begin{align}
\begin{split}
\E\big[|\ft{\<2>}_N(n, t)|^2\big]
&\ges \sum_{\substack{n_1 \in \Z^2 \\ |n| \ll |n_1| \le N}}
\frac{1}{\jb{n_1}^{2(1-\al)} \jb{n-n_1}^{2(1-\al)}}
\ges
\begin{cases}
\log N, & \text{if } \al = \frac 12 \\
N^{-2+4\al}, & \text{if } \al>\frac 12
\end{cases} \\
& \to \infty
\end{split}
\label{Yk0}
\end{align}
as $N \to \infty$.
Then, from Kolmogorov's three-series theorem, Kolmogorov's zero-one law, and 
Remark~\ref{REM:divcov} with \eqref{Yk0}, we have
\[P \Big( \big| \lim_{N \to \infty} \ft{\<2>}_N (n,t) \big| <\infty \Big) =0.\]

\noi
In particular,
we obtain that $\{ \<2>_N \}_{N \in \N}$ 
forms a divergent sequence  in $C([0,T]; \mathcal{D}'(\T^2))$ almost surely for $\al \ge \frac 12$.
This concludes the proof of Proposition \ref{PROP:heat}\,(i).

\subsection{Proof of  Proposition \ref{PROP:heat}\,(ii)}
\label{SUBSEC:52}

First, we prove that
$\<20>_N \in C(\R_+;\C^{2-2\al-} (\T^2))$ 
with a uniform (in $N$) bound
on each bounded time interval $[0, T]$, 
almost surely.
As in Subsection~\ref{SUBSEC:32}, it suffices to show 
\begin{align}
\E\big[|\ft{\<20>}_N(n, t)|^2\big]
\les_T \jb{n}^{-2-2(2-2\al)}
\label{H20a}
\end{align}
for any $n \in \Z^2$ and $0 \le t \le T$, uniformly in $N \in \N$.

We only consider $n \ne 0 $ for simplicity.
Proceeding as in \eqref{S5}, we have 
\begin{align}
&\E \big[ |\ft{\<20>}_N(n, t)|^2\big] \notag \\
& = \frac{1}{\pi^2}  \sum_{\substack{n = n_1 + n_2\\n_1 \ne \pm n_2 \\ |n_1|, |n_2| \le N}}
\int_0^{t} e^{-(t - t_1) \jb{n}^2}
 \int_0^{t_1} e^{-(t - t_2) \jb{n}^2} 
\kk_{n_1}(t_1, t_2) \kk_{n_2}(t_1, t_2)  dt_2dt_1\notag\\
& \hphantom{X}
+ \frac{1}{2\pi^2} \cdot \ind_{n \in 2 \Z^2\setminus\{0\}} 
 \int_0^t e^{-(t - t_1) \jb{n}^2}
 \int_0^{t_1} e^{-(t - t_2) \jb{n}^2} 
\E\Big[ \ft{\<1>}_N\big(\tfrac{n}{2}, t_1\big)^2 \, \cj{\ft{\<1>}_N\big(\tfrac{n}{2}, t_2\big)^2 } \Big]  dt_2dt_1
\notag\\
& =: 
\1 (n, t)+ \II(n, t),
\label{S5H}
\end{align}

\noi
where $\kk_{n_j}(t_1, t_2)$ is as in \eqref{sigma2H}.
From \eqref{sigma2H}, we see that 
the contribution from $\II(n, t)$ satisfies \eqref{H20a}.
Hence,  we focus on $\1 (n,t)$.
By \eqref{S5H} and \eqref{sigma2H},
we have
\begin{align}
\begin{split}
\1 (n, t)
&\sim \sum_{\substack{n = n_1 + n_2\\n_1 \ne \pm n_2 \\ |n_1|, |n_2| \le N}} \frac{1}{\jb{n_1}^{2(1-\al)} \jb{n_2}^{2(1-\al)}}
\\
& \hphantom{X} \times
e^{-2t\jb{n}^2} 
\int_0^{t} e^{t_1 (\jb{n}^2-\jb{n_1}^2-\jb{n_2}^2)}
 \int_0^{t_1} e^{t_2 (\jb{n}^2+\jb{n_1}^2+\jb{n_2}^2)}
dt_2dt_1\\
&= \sum_{\substack{n = n_1 + n_2\\n_1 \ne \pm n_2 \\ |n_1|, |n_2| \le N}} \frac{1}{\jb{n_1}^{2(1-\al)} \jb{n_2}^{2(1-\al)}}
\\
& \hphantom{X} \times
\frac{1}{\jb{n}^2+\jb{n_1}^2+\jb{n_2}^2}
\bigg(
\frac{1-e^{-2t\jb{n}^2}}{\jb{n}^2}
- e^{-2t\jb{n}^2} \frac{1-e^{t(\jb{n}^2-\jb{n_1}^2-\jb{n_2}^2)}}{\jb{n_1}^2+\jb{n_2}^2-\jb{n}^2} \bigg)
\\
&\les\frac{1}{\jb{n}^2} \sum_{\substack{n = n_1 + n_2\\n_1 \ne \pm n_2}} \frac{1}{\jb{n_1}^{2(1-\al)} \jb{n_2}^{2(1-\al)}}
\frac{1}{\jb{n}^2+\jb{n_1}^2+\jb{n_2}^2}.
\end{split}
\label{S9H}
\end{align}

\noi
By 
separately estimating the contributions
from $|n_1|\sim|n_2| \gg |n|$
and $|n_1| \sim |n| \ges |n_2|$ (or $|n_2| \sim |n| \ges |n_1|$)
with Lemma \ref{LEM:SUM},
we  see that 
the contribution from $\1(n, t)$ also satisfies~\eqref{H20a}
for $0 < \al < 1$.
This proves \eqref{H20a}.

Next, we show that when $\al \geq 1$,  $\{ \<20>_N \}_{N \in \N}$ does not converge in $C([0,T]; \D'(\T^2))$ for any  $T> 0$ almost surely.
From Remark \ref{REM:divcov},
it suffices to show that
\begin{align}
\lim_{N \to \infty} \E \big[ |\<20>_N(n,t)|^2 \big] = \infty
\label{S8H}
\end{align}
for $\al \ge 1$ under an appropriate assumption on $t>0$.

Since the second term $\II(n, t)$ in \eqref{S5H} does not involve any summation, 
it is finite.
From~\eqref{S9H}, 
it is easy to see that 
 the contribution to $\1(n, t)$
from  $|n|\sim \max(|n_1|, |n_2|)$ is finite.
Indeed, assuming $|n_1| \les |n|\sim|n_2|$ without loss of generality, 
the contribution to~\eqref{S9H} from this case is bounded by 
\begin{align*}
\frac{1}{\jb{n}^{6-2\al}}
\sum_{\substack{n_1 \in \Z^2\\|n_1| \les|n|}}
\frac{1}{\jb{n_1}^{2-2\al} }
\les \jb{n}^{-6+4\al}. 
\end{align*}

It remains to estimate  $\1 (n,t)$ under the constraint  $|n| \ll |n_1| \sim |n_2|$.
When $t \gg |n|^{-2}$, 
the contribution to $\1(n, t)$ from this case  is bounded below by
\[
\begin{split}
\frac{1}{\jb{n}^2}
\sum_{\substack{n_1 \in \Z^2 \\ |n| \ll |n_1| \le N}} \frac{1}{\jb{n_1}^{6-4\al}}
&\ges _n
\begin{cases}
\log N, &\text{if } \al =1 \\
N^{-4+4\al}, &\text{if } \al>1
\end{cases} \\
&\longrightarrow  \infty
\end{split}
\]
as $N \to \infty$.
This proves \eqref{S8H} for $t \gg |n|^{-2}$,  when $\al \ge 1$.

\subsection{Proof of Lemma \ref{LEM:heat2}}

As in Subsection \ref{SUBSEC:33}, 
it suffices to show 
\begin{align}
\E \big[ |\Ft [\<20>_N \pe \<1>_N](n,t)|^2 \big]
\les_T \jb{n}^{-2-2(2-3\al)}
\label{Y1H}
\end{align}
for $n \in \Z^2$ and $0 \le t \le T$, uniformly in $N \in \N$.
As in \eqref{Y1}, we decompose $\<20>_N \pe \<1>_N$ into two parts:
\begin{align*}
\Ft [\<20>_N \pe \<1>_N] (n, t)
&=\frac{1}{4\pi^2}
\sum_{\substack{n = n_1 + n_2+n_3\\|n_1 + n_2| \sim |n_3|\\ n_1 + n_2 \ne 0}} 
\int_0^t  e^{-(t-t') \jb{ n_1 + n_2 }^2} 
\ft{\<1>}_N(n_1, t')\ft{\<1>}_N(n_2, t') dt' \cdot  \ft{\<1>}_N(n_3, t)
\notag\\
&\quad+  
\frac{1}{4\pi^2}
\sum_{\substack{n_1\in \Z^2 \\ |n_1| \le N}} \ind_{|n|\sim 1} \int_0^t e^{-(t - t')} \cdot \big(  |\ft{\<1>}(n_1, t')|^2  -  \kk_{n_1}(t')\big) 
dt' \cdot  \ft{\<1>}_N(n, t) \notag\\
& =: \ft \RR_1(n, t) + \ft \RR_2(n, t).
 \end{align*}
Moreover, we decompose $\RR_1$ as 
 \begin{align}
\ft \RR_1(n, t)
& = 
\frac{1}{4\pi^2}
\sum_{\substack{n = n_1 + n_2+n_3\\|n_1 + n_2| \sim |n_3|\\ (n_1 + n_2)(n_2 + n_3) (n_3 + n_1) \ne 0}} 
\int_0^t  e^{-(t-t') \jb{ n_1 + n_2 }^2} 
\ft{\<1>}_N(n_1, t')\ft{\<1>}_N(n_2, t') dt' \cdot  \ft{\<1>}_N(n_3, t) \notag\\
& \hphantom{X} 
+ \frac{1}{2\pi^2} 
\int_0^t \ft{\<1>}_N(n, t') 
\bigg[\sum_{\substack{n_2 \in \Z^2\\|n_2| \sim |n+n_2|\ne 0 \\ |n_2| \le N}}  e^{-(t-t') \jb{ n + n_2 }^2} \notag\\
& \hphantom{XXXXXXXXXlllllXX}
\times 
\Big(\ft{\<1>}(n_2, t')\ft{\<1>}(-n_2, t) - \kk_{n_2}(t,  t')\Big)\bigg] dt'    \notag\\
& \hphantom{X} 
+ \frac{1}{2\pi^2}
\int_0^t \ft{\<1>}_N(n, t') 
\bigg[\sum_{\substack{n_2 \in \Z^2\\| n_2| \sim |n+n_2| \ne 0 \\ |n_2| \le N}} e^{-(t-t') \jb{ n + n_2 }^2} 
 \kk_{n_2}(t, t')\bigg] dt'   \notag\\
& \hphantom{X} 
- \frac{1}{4\pi^2}\cdot \ind_{n\ne0}
\int_0^t e^{-(t-t') \jb{ 2n }^2}
\, (\ft{\<1>}_N(n, t'))^2dt' \cdot \ft{\<1>}_N(-n, t)\notag\\
&
=: \ft \RR_{11}(n, t)+  \ft \RR_{12}(n, t)+  \ft \RR_{13}(n, t)+  \ft \RR_{14}(n, t).
\label{Y3H}
\end{align}

\noi
Proceeding as in the proof of 
 Proposition \ref{PROP:sto2}, 
 we can easily show that 
$ \ft \RR_{11}$, 
$ \ft \RR_{12}$, and $\ft \RR_{14}$ satisfy \eqref{Y1H}.

It remains to consider  $\ft \RR_{13}$.
Under the constraint $| n_2| \sim |n+n_2|$, 
we have $|n_2 |\ges |n|$.
Then, from \eqref{Y3H} with \eqref{sigma2H}, 
we have 
\begin{align*}
\E\big[|\ft \RR_{13}(n, t)|^2\big]
&\les_T \frac{1}{\jb{n}^{2-2\al}}
\sum_{\substack{n_2 \in \Z^2\\| n_2| \sim |n+n_2| }}
\frac{1}{\jb{n+n_2}^2\jb{n_2}^{2-2\al}}\\
& \hphantom{XX}
\times \sum_{\substack{n_2' \in \Z^2\\| n_2'| \sim |n+n_2'| }} 
\frac{1}{\jb{n+n'_2}^2\jb{n'_2}^{2-2\al}}\\
& \les \jb{n}^{-6 + 6\al}
\end{align*}

\noi
for $0 < \al < 1$, 
verifying \eqref{Y1H}.

\begin{ackno}\rm 

T.O.~ was supported by the European Research Council (grant no.~637995 ``ProbDynDispEq''
and grant no.~864138 ``SingStochDispDyn").
M.O.~was supported by JSPS KAKENHI Grant numbers
JP16K17624 and JP20K14342.
The authors would like to thank Martin Hairer
for a helpful conversation.
M.O.~would like to thank the School of Mathematics at the University of Edinburgh
for its hospitality, where this manuscript was prepared.
The authors would like to thank the anonymous referee for numerous  helpful comments,
which have improved the presentation of the paper.
\end{ackno}

\end{document}